\documentclass[11pt,a4paper,leqno]{article}
\usepackage{mymacros}
\usepackage{cite}
\newcommand{\wick}[1]{\mathopen{:}#1\mathclose{:}}
\newcommand{\wwick}[1]{\mathopen{::}#1\mathclose{::}}
\newcommand{\commentrb}[1]{\begin{quote}\bf Comment RB: #1\end{quote}}

\newcommand{\Eg}{\mathsf{E}}
\newcommand{\Pg}{\mathsf{P}}
\newcommand{\Qg}{\mathsf{Q}}

\usepackage[titles]{tocloft}
\newcommand{\dsemnorm}[1]{[\![#1]\!]}
\newcommand{\dnorm}[1]{|\!|\!|#1|\!|\!|}

\usepackage{scalerel}[2014/03/10]
\usepackage[usestackEOL]{stackengine}
\def\mint{\,\ThisStyle{\ensurestackMath{%
    \stackinset{c}{0\LMpt}{c}{0\LMpt}{\SavedStyle-}{\SavedStyle\phantom{\int}}}%
    \setbox0=\hbox{$\SavedStyle\int\,$}\kern-\wd0}\int} 

\title{Holley--Stroock uniqueness method for the $\varphi^4_2$ dynamics}
\author{Roland Bauerschmidt\footnote{Courant Institute of Mathematical Sciences, NYU. E-mail: {\tt bauerschmidt@cims.nyu.edu}.}
\and Benoit Dagallier\footnote{CEREMADE, Université Paris-Dauphine, PSL University. E-mail: {\tt dagallier@ceremade.dauphine.fr}.}
\and Hendrik Weber\footnote{Institut für Analysis und Numerik, Universit\"at M\"unster. E-mail: {\tt hendrik.weber@uni-muenster.de}.}}

\date{April 11, 2025} 

\begin{document}
\maketitle

\begin{abstract}
  The approach initiated by Holley--Stroock establishes the uniqueness of invariant measures of Glauber dynamics of lattice spin systems
  from a uniform log-Sobolev inequality.
  We use this approach to prove uniqueness of the invariant measure of the $\varphi^4_2$ SPDE up to the critical temperature
  (characterised by finite susceptibility).
  The approach requires three ingredients: a uniform log-Sobolev inequality (which is already known),
  a propagation speed estimate, and a crude estimate on the relative entropy of the law of the finite volume dynamics at time $1$
  with respect to the finite volume invariant measure.
  The last two ingredients are understood very generally on the lattice,
  but the proofs do not extend to SPDEs,
  and are here established in the instance of the $\varphi^4_2$ dynamics.
\end{abstract}

\setcounter{tocdepth}{1}
\tableofcontents

\section{Introduction and main results}

Given $\lambda>0$ and $\mu\in\R$, we consider the $\varphi^4$ SPDE on $\R^2$, obtained as the limit $\epsilon\to 0$ of
\begin{equation} \label{e:SPDE}
  \dot \varphi = - A \varphi - \lambda \varphi^3 - (\mu-a(\lambda))\varphi + \sqrt{2} \dot W, 
\end{equation}
where $A = -\Delta +1$ is the Laplacian (the mass term $+1$ included in $A$ is notationally convenient later and equivalent to replacing $\mu$ by $\mu-1$),
$a(\lambda)=a_\epsilon(\lambda)$ is the divergent counterterm (which is chosen independently of $\mu$ and can, in our
two-dimensional situation, be taken to be $\frac{3}{2\pi} \log (\frac{1}{\epsilon}) \lambda$),
and $\dot W = \dot W^\epsilon$ is space-time white noise regularised suitably at scale $\epsilon>0$.
The SPDE is well posed starting from initial data in the weighted Besov--H\"older space
$C^{-\alpha}(\rho)$ with $\alpha>0=\frac12 (d-2)$ small and any decaying polynomial weight $\rho$ on $\R^2$,
see Section~\ref{sec:preliminaries} for this and other preliminaries. 
It is also known that its space-periodic version on the torus $\T_L^2$
has a unique invariant measure given by the torus $\varphi^4$ measure which will be denoted $\nu_L$.
On the other hand, in the infinite volume limit $L \to \infty$, invariant measures need not be unique.

Our goal is to prove uniqueness of the invariant measure above the static critical temperature.
This critical temperature is characterised in terms of the susceptibility, defined by
\begin{equation} \label{e:chi-def}
  \chi(\lambda,\mu)
  = \sup_{L}\frac{\E_{\nu_L}\big[(\varphi,1)^2\big]}{|\T_L^d|}
  = \sup_L\int_{\T_L^d} \E_{\nu_L}\big[\varphi(0)\varphi(x)\big]   \, dx,
\end{equation}
where the last equality is formal, and 
\begin{equation}
  \mu_c(\lambda)=\inf\{\mu\in \R: \chi(\lambda,\mu)<\infty\}.
\end{equation}
It is known that $\mu_c(\lambda) \in (-\infty,\infty)$ for every $\lambda>0$, see \cite{MR0363256,MR4426324}.

\begin{theorem} \label{thm:main-uniqueness}
  For any $\lambda>0$ and $\mu>\mu_c(\lambda)$,
  the infinite volume $\varphi^4$ SPDE \eqref{e:SPDE} on $\R^2$
  has a unique invariant measure $\nu$ supported on $C^{-\alpha}(\rho)$, where $\alpha,\sigma>0$ and $\rho(x)=(1+|x|^2)^{-\sigma/2}$.
  
  Moreover, for any deterministic initial condition $\varphi_0 \in C^{-\alpha}(\rho)$ and any test function $f\in C_c^\infty(\R^2)$,
  for sufficiently large $t>0$ (depending on $f$):
  \begin{equation} \label{e:unique-rate-main}
    \Big|\E\big[e^{i(f,\varphi_t)}\big] -\E_{\nu}\big[e^{i(f,\varphi)}\big]\Big|
    \lesssim
    (1+\|\varphi_0\|_{-\alpha,\rho}^{4})e^{-\gamma t(1+o(1))}
    ,
  \end{equation}
  where $\gamma>0$ is a uniform bound on the log-Sobolev constant of $\nu_L$ and $o(1)$ is independent of the initial condition.
  The $C^{-\alpha}(\rho)$-norm $\|\cdot\|_{-\alpha,\rho}$ is defined in Section~\ref{sec:norms-def}.
\end{theorem}

The method of proof for Theorem~\ref{thm:main-uniqueness} is a continuum version of the Holley--Stroock strategy for
uniqueness of invariant measures for spin systems on $\Z^d$, see \cite{MR893137,MR428984,MR410985,MR1153990,MR1292280,MR1370101}
for the development of this strategy in various settings, and also \cite{MR2352327,MR1971582} for textbook treatments.
This strategy requires three main ingredients:
\begin{enumerate} \parskip=0pt
  \item[(1)] A propagation speed estimate that compares the dynamics of a finite volume
    approximation with the dynamics on the full space (or a large volume);
  \item[(2)] A relatively crude quantitative estimate on the density of the law of the dynamics on a finite volume at a finite time;
  \item[(3)] A log-Sobolev inequality for the infinite volume dynamics (or a uniform bound on the finite volume log-Sobolev constant).
\end{enumerate}
On the lattice, (1) and (2) can be established very generally, and only the log-Sobolev inequality is really model specific.
For the $\varphi^4_d$ measures (in the continuum), the uniform log-Sobolev inequality was proved in \cite{MR4720217},
also in dimension 3.
However, for SPDEs in the continuum, (1) and (2) have not been proved, and the typical lattice arguments
do not extend.
For example, typical existing proofs of (1) on the lattice involve comparisons of generators
in which it is exploited that the lattice Laplace operator is a bounded operator (leading to wrong scaling in the
continuum limit where it becomes unbounded), and (2) can be proved on the lattice by exploiting ultracontractivity
of the finite dimensional finite volume dynamics on the lattice.
In this work, we address (1) and (2) for the $\varphi^4_2$ measure, making use of a combination of pathwise SPDE
techniques (such as the a priori bounds \cite{MR4164267}) and of stochastic calculus (such as the ``It\^o trick'' from
\cite{MR2593276,KoenigPerkowskivanZuiljen_PAM}).
We expect that the argument leading to (1) can be extended
to higher dimensions with a more careful and systematic treatment.
The extension of (2) to higher dimensions remains an interesting  question.
Only relatively crude but quantitative bounds are required, which we certainly expect to be true very generally, but
our proof uses the simplifying feature that the finite volume
$\varphi^4_2$ measure is absolutely continuous with respect to the free field (and in fact that this can be shown for the
fixed time distributions in finite volume as well, similarly to \cite{MR2122959,MR4476105}).
Finally, we mention that the typical lattice arguments do not even apply to the continuum limit in one dimension,
which is however very well understood nonetheless due to the connection with diffusion processes and the absense of divergent counterterms;
see, for example, \cite{MR2127731} and references.

Our result answers half of the sharpness question of the phase transition for invariant measures of the $\varphi^4_2$ SPDE:
For $\mu>\mu_c$ we prove there is a unique invariant measure.
For $\mu<\mu_c$ we expect there are two ergodic invariant measures with opposite non-zero magnetisation. 
Such results are well established for Ising and percolation lattice models, including the lattice $\varphi^4$ model,
see \cite{MR894398} and many more recent developments, but so far not in the continuum limit.
We also refer to \cite{2211.00319} for recent finer results concerning the structure of Gibbs measure of lattice $\varphi^4$ models.

It is well known that invariant measures for Glauber dynamics and Gibbs measures coincide on the lattice,
see for example \cite{MR512335} (and also \cite{MR451455}).
This correspondence is less developed in the continuum, partly because the notion of Gibbs state requires discussion of boundary conditions
and is therefore  less convenient in the continuum, in particular when the dimension is greater than two (where different boundary conditions
lead to singular measures and, for example, the ``half-Dirichlet'' states considered in \cite{MR0378670,MR0489552} cannot be used),
but we expect that appropriate notions of Gibbs states and invariant measures would nonetheless coincide, certainly in two dimensions.
An advantage of (our implementation of) the Holley--Stroock uniqueness method is that boundary conditions need not be considered.
For weak coupling corresponding to $\mu \gg \mu_c(\lambda)$,
the uniqueness of Gibbs states of the $\varphi^4_2$ model is proven in \cite{MR1003426,MR990999}.

It is a direct consequence of the uniqueness
(together with the Euclidean invariance of the SPDE \eqref{e:SPDE} which can be defined using a Euclidean invariant regularisation on the full plane \cite{MR3693966},
see also Corollary~\ref{cor:infvol-existence} below)
that the unique infinite volume invariant measure for $\mu>\mu_c(\lambda)$ is Euclidean invariant.
The Euclidean invariance of an invariant measure obtained from a finite volume approximation defined on the sphere
was also recently established in \cite{2311.04137}, for all $\mu \in \R$,
and see also \cite{2102.08040} for rotational invariance of such a measure in three dimension. These results do not establish the mixing property,
which must fail when the invariant measure is not unique as expected for $\mu<\mu_c(\lambda)$.
While we do not explicitly derive the mixing property, we expect it would be a relatively simple extension using
the same methods and under the same assumption $\mu>\mu_c(\lambda)$,
  see for example \cite[Section~8]{MR1971582}.
For sufficiently large $\mu$ our uniqueness result would in any case allow to apply the existing correlation decay estimates from \cite{MR723546}.
We also mention the study of uniqueness and correlation decay for invariant measures of SPDEs with exponential interaction using convexity of the potential \cite{MR4872106}.

\medskip

In the final stages of the preparation of this manuscript we learned that related results are in preparation \cite{DuchHairer}.
An advantage of the results of \cite{DuchHairer} seems to be that they apply in $d=3$ and do not rely on correlation inequalities in any way.
On the other hand, with such methods, it seems impossible to reach the optimal condition $\mu>\mu_c(\lambda)$ that we obtain,
and we expect that our proof will eventually be extended to $d=3$ as well.

\section{Preliminaries on the SPDE}
\label{sec:preliminaries}

In this section we collect preliminaries on the $\varphi^4_2$ SPDE,
and also set the precise conventions (for example for time and volume independent counterterms).
The preliminaries are essentially all known, 
but often not stated in the literature in precisely the version that we need (for example not with the correct sharp weights). To make our
treatment complete and accessible we have included essentially self-contained proofs in Appendices~\ref{app:norms}--\ref{app:apriori}.

\subsection{Local Besov--H\"older norms}
\label{sec:norms-def}

We use local and weighted versions of the Besov--H\"older spaces $C^{\alpha}$,
defined in this section similarly to \cite[Section~2.1]{MR4164267}.
Since the local versions are not completely standard (but nice), we include proofs of the properties we need in Appendix~\ref{app:norms}.

Throughout the bulk of the paper, we only use the standard ``elliptic'' versions of these norms
(defined in terms of the Euclidean distance)
appropriate for distributions in the spatial variable $x\in \R^d$.
In Appendix~\ref{app:apriori}, we also make use of analogous parabolic versions of these norms for distributions in the space-time variable $(x,t) \in \R^d \times \R$;
the precise definitions are recalled there.
We sometimes write  $\R^d_x$ instead of $\R^d$ to emphasise space and $\R_t$ for time.

\paragraph{Scaling and test function}
For $x\in \R_x^d$ we denote the Euclidean norm by $|x|$ and write $B_R(x)$ for the corresponding ball of radius $R$:
\begin{equation}
B_R(x) = \{\bar x \in \R_x^d \colon: |\bar x-x|<R\}.
\end{equation}
For a function $f \colon \R_x^d \to \R$ and a scaling parameter $R>0$, we denote the rescaled function by
\begin{equation}
f_R(x) = R^{-d} f\Big(\frac{x}{R}\Big).
\end{equation}
The negative regularity H\"older norms
are defined in terms of a smooth test function $\Psi$, supported in the unit ball of $\R^d$ and normalised such that $\int \Psi \, dx =1$,
that is fixed from now on.

\paragraph{Norms}

Given a set $C\subset \R^d$ respectively a bounded weight $\rho : \R^d \to (0,\infty)$,
we denote the uniform norm respectively weighted uniform norm by
\begin{equation}
  \|f\|_C = \sup_{x\in C}|f(x)|, \qquad \|f\|_\rho = \sup_{x} \rho(x) |f(x)|.
\end{equation}
To avoid risks of confusion the uniform norm will also be written $\|f\|_{\bbL^\infty(C)}$ respectively $\|f\|_{\bbL^\infty(\rho)}$. 
For $\alpha<0$, local versions of the $C^\alpha$ norm are defined by
\begin{align}
    \label{e:norm-minus-C}
  \|f\|_{\alpha,C} &= \sup_{\substack{R\leq 1, x\in C:\\B_R(x) \subset C}} |\Psi_R*f(x)| R^{-\alpha},
  \\
  \label{e:norm-minus-rho}
  \|f\|_{\alpha,\rho} &= \sup_{R\leq 1} \|\Psi_R*f\|_{\rho} R^{-\alpha}.
\end{align}
For $\alpha\in (0,1)$, the H\"older seminorms are defined by
\begin{align}
  \label{e:norm-plus-C}
  [f]_{\alpha,C} &= \sup_{\substack{x\neq y \in C\\|x-y|\leq 1}} \frac{|f(x)-f(y)|}{|x-y|^\alpha},
  \\
    \label{e:norm-plus-rho}
  [f]_{\alpha,\rho} &= \sup_{x\in\R^d} \rho(x) \sup_{0<|z|\leq 1} \frac{|f(x)-f(x+z)|}{|z|^\alpha},
\end{align}
and the corresponding weighted H\"older norm, for $\alpha \in (0,1)$, is
\begin{equation}
  \|f\|_{\alpha,\rho} = \|f\|_{\rho} + [f]_{\alpha,\rho}.
\end{equation}

For a sufficiently regular weight $\rho$, a version of the weighted Besov--H\"older space $C^{-\alpha}(\rho)$
can be defined as the completion of $C_c^\infty(\R^d)$ with respect to $\|\cdot\|_{\alpha,\rho}$,
see Remark~\ref{rk:completion} and also \cite[Remark~7]{MR3693966} for a comparison of this definition with the restriction
of the space of Schwartz distributions to the subspace of those with finite norm.
Different choices of bump function $\Psi$ lead to equivalent norms on the full space,
see Proposition~\ref{prop:norm-tilde} in Appendix~\ref{app:norms}.

The above definitions give the inhomogeneous versions of the Besov--H\"older norms
(which we will use throughout the bulk of the paper),
in which the distances $|x-y|$ and $R$ are restricted to $(0,1]$. The homogeneous versions of these norms would not have this restriction;
their space-time versions will be more convenient in Appendix~\ref{app:apriori} and are recalled there.

\paragraph{Periodisation and weights}
Norms on the torus $\T_L^d$ are defined by viewing $f: \T_L^d \to \R$ as a $2L$-periodic function on $\R^d$.
The space $C^\alpha(\T_L^d)$ without explicit specification of the weight refers to the weight $\rho=1$
but it is equivalent to the weighted versions we will use since all weights will be bounded above and below
on compact sets. Indeed, we will only use the polynomial weights
\begin{equation} \label{e:weight}
  \rho(x)=\rho^\sigma(x)=(1+|x|^2)^{-\frac{\sigma}{2}}, \qquad x\in\R^2,
\end{equation}
with $\sigma \geq 0$.

\subsection{Space-periodic dynamics}\label{subsec_periodised_dynamics}

Let $W = (W_t)_{t\geq 0}$ denote a cylindrical Wiener process on $L^2(\R^d)$
defined on some probability space that is fixed from now on, see, e.g., \cite{MR3236753}.
Thus its distributional time derivative $\dot W$ is a space-time white noise on $\R_+ \times \R^d$.
Even though we always choose $d=2$, we sometimes write $d$ for emphasis.
We will consider the SPDE \eqref{e:SPDE} on the full space $\R^d$ and its space-periodic version on the torus $\T^d_L=[-L,L)^d/\sim$ which is formally given by
\begin{equation}   \label{e:SPDE-eps}
  \dot \varphi^{L} 
  = -A\varphi^L - \lambda (\varphi^L)^3 - (\mu-a(\lambda))\varphi^L + \sqrt{2}\dot W^L, 
\end{equation}
where the divergent counterterm $a(\lambda)=a_\epsilon(\lambda)$ is chosen to be the same as in \eqref{e:SPDE} and
we also assume that the full space space equation and the torus equation are coupled through the same driving noise, i.e.,
$W^L$ is the space-periodised version of $W$:
\begin{equation}\label{e:xiL}
  (W^L_t,f) = (W_t,f^L), 
  \qquad 
  f^L(x) = 1_{(- L,L]^d}(x)\sum_{n\in \Z^d} f(x+2Ln).
\end{equation}
Given an initial condition $\varphi_0 \in \cS'(\R^d)$, we define a periodised initial condition $\varphi_0^L \in \cS'(\T_L^d)$
by
\begin{equation} \label{e:phi0L}
  \varphi_0^L = \sum_{n\in \Z^d} T_{2Ln}(\chi_L\varphi_0)
\end{equation}
where $\chi \in C_c^\infty(B(0,1))$ is a bump function with $\chi(x)=1$ for $|x|\leq \frac{99}{100}$, $\chi_L(x) =\chi(x/L)$,
and $T_xf$ denotes translation of the distribution $f$ by $x$.
Clearly,
\begin{equation}
(\varphi^L_0,f)
=
(\varphi_0,f)
,
\label{eq_def_varphi_L0}
\end{equation}
for any smooth $f$ supported on $[-\frac{99}{100}L,\frac{99}{100}L]^2$ 
and if $\|\varphi_0\|_{-\alpha,\rho}<\infty$,
\begin{equation}
\lim_{L\to\infty}\|\varphi^L_0-\varphi_0\|_{-\alpha,\rho}
=
0
,
\label{eq_cv_varphi_L0}
\end{equation}
where  $\|\cdot\|_{-\alpha,\rho}$  is the weighted $C^{-\alpha}$ 
norm defined in \eqref{e:norm-minus-rho}. 
We omit the proof of the last statement and refer to
\cite[Lemma 13]{MR3693966} for a very similar statement.

\subsection{Gaussian estimates}
\label{sec_gaussian_estimates}

We next discuss the stochastic heat equation~\eqref{eq_heat}, also referred to as Ornstein--Uhlenbeck process, 
its periodised version~\eqref{eq_heat_periodised}, and their Wick powers \eqref{e:Wick-convention} with volume and mass independent renormalisation
as appropriate for the infinite volume limit.
Because we could not find a reference that contains exactly what we need, especially not the quantitative continuity theorem in time and the quantitative
convergence of the finite volume approximation,
the proofs of these estimates are included in Appendix~\ref{app_gaussian_estimates}.
The dimension is always $d=2$.

Let $\dot W$ be the space-time white noise defined in Section~\ref{subsec_periodised_dynamics}.
The Ornstein--Uhlenbeck process $Z$ is defined by
\begin{equation} \label{e:OU}
  (\partial_t+A)Z =  \sqrt{2}\dot  W, \qquad Z_0=\varphi_0,
\end{equation}
where $A=-\Delta+1$, and the solution to~\eqref{e:OU} is interpreted as the stochastic integral
\begin{equation}
  Z_t = e^{-tA}\varphi_0 + \int_0^t e^{-(t-s)A} \, \sqrt{2}dW_s
  ,\qquad
  t\geq 0
  .
  \label{eq_heat}
\end{equation}
Its torus version $Z^L$ is defined using the periodisation of the same driving noise, see \eqref{e:xiL}, by:
\begin{equation}
Z^L_t 
=
e^{-tA}\varphi^L_0 + \int_0^t e^{-(t-s)A}\, \sqrt{2}dW^L_s
,\qquad
t\geq 0
.
\label{eq_heat_periodised} 
\end{equation}
The Ornstein--Uhlenbeck process is well posed in $C^{-\alpha}(\rho)$ from any initial condition $Z_0\in C^{-\alpha}(\rho)$,
for any $\alpha>0$ (in dimension two) and where $\rho$ denotes the polynomial weight \eqref{e:weight} with arbritary $\sigma>0$, see  Appendix~\ref{app_gaussian_estimates}.
The covariance of $Z_t$ and $Z_s$ is given by
  \begin{equation} \label{e:OU-cov}
    2\int_0^{s\wedge t} e^{-(t+s-2u)A}\, du
    = \int_{|t-s|} ^{t+s} e^{-uA}\, du
    = \int_{|t-s|} ^{t+s} e^{-u}e^{u\Delta}\, du
  \end{equation}
  and the mean of $Z_t$ by
  \begin{equation} \label{e:OU-mean}
    e^{-tA}\varphi_0 = e^{-t}e^{t\Delta} \varphi_0
    ,
  \end{equation}
  where $(e^{t\Delta})(x,y)=p_t(x-y)$ is the heat kernel on $\R^d$. 
  Similar identities hold for $Z^L$, in terms of the periodised heat kernel $p_t^L$ on $\T_L^d$. 
  These kernels are given by:
    \begin{align}
p_t(x) &=
\frac{1}{(4\pi t)^{d/2}}e^{-|x|^2/4t} 
,\qquad 
x \in\R^d
,
\label{eq_OU_kernel}
\\
 p_t^L(x) &= \sum_{a\in\Z^2}p_t(x-2aL)
 , \qquad x \in \T_L^d.
 \label{eq_OU_kernel_periodic}
\end{align}

For any initial condition $\varphi_0 \in C^{-\alpha}(\rho)$ on $\R^2$, 
we recall that $\varphi_0^L$ denotes its periodised version \eqref{e:phi0L}.
From now on we write $Z$, $Z^L$ for the solutions of~\eqref{eq_heat}, \eqref{eq_heat_periodised} respectively starting from $\varphi_0$, $\varphi^L_0$,
and also denote by $\tilde Z$, $\tilde Z^L$ the solutions with $0$ initial condition:
\begin{equation}\label{e:def-Gaussian-without-Initial-datum}
\tilde Z_t
:=
Z_t - e^{-tA}\varphi_0
,\qquad
\tilde Z^L_t
:=
Z^L_t - e^{-tA}\varphi^L_0
,\qquad
t\geq 0
.
\end{equation}

We define Wick powers $\wick{\tilde Z^n}$ of $\tilde Z$ with counterterm $a_\epsilon = \frac{1}{2\pi} \log \frac{1}{\epsilon}$ 
independent of $t$ and $L$.
Thus they are given as space-time distributions by, e.g., for $n=2,3$:
\begin{equation} \label{e:Wick-convention}
  \wick{\tilde Z_t^2} = \lim_{\epsilon\to 0} \pB{ (\eta_\epsilon*\tilde Z_t)^2-a_\epsilon},
  \qquad
  \wick{\tilde Z_t^3} = \lim_{\epsilon\to 0} \pB{ (\eta_\epsilon*\tilde Z_t)^3-3a_\epsilon(\eta_\epsilon* \tilde Z_t)},
\end{equation}
and correspondingly for higher Wick powers. 
The function $\eta: \R^d \to [0,\infty)$ is a smooth compactly supported mollifier with unit integral and $\eta_\epsilon(x) = \epsilon^{-d}\eta(x/\epsilon)$. 
As indicated in the notation, these space-time distributions can be interpreted as distributions
in the spatial variable for fixed $t>0$.

The Wick powers with non-vanishing initial conditions $\varphi_0\in C^{-\alpha}(\rho)$ are then defined by the formula:
\begin{equation} 
\wick{Z_t^n}
=
\sum_{\ell=0}^n\binom{n}{\ell}\wick{\tilde{Z}_t^\ell}(e^{-tA}\varphi_0)^{n-\ell}
,
\label{eq_link_wick_powers_w_w/o_IC}
\end{equation}
where the products are well-defined for $t>0$ by the multiplicative property of Besov spaces and since $e^{-tA}\varphi_0$ is smooth
(see Propositions~\ref{prop:besov-mult} and~\ref{prop:Besov-heat}).
Analogous definitions (with the same counterterm $a_\epsilon$ independent of $L$) apply to the Wick powers of $Z^L$.
For ease of notation we write:
\begin{equation}
\wick{(\tilde Z^\infty_t)^n}
\overset{\text{def}}{=}
\wick{(\tilde Z_t)^n},
\qquad
\wick{(Z^\infty_t)^n}
\overset{\text{def}}{=}
\wick{(Z_t)^n}
\ .
\end{equation}

The estimates we need are summarised in the following propositions, proved in Appendix~\ref{app_gaussian_estimates}.
Recall that $\rho(x)=(1+|x|^2)^{-\sigma/2}$ and $Z^\infty:=Z$.

\begin{proposition}\label{prop_gaussian_bounds}
 Let $n\geq 1$ be an integer, 
  let $\alpha,\sigma>0$ and recall that $\rho(x)=(1+|x|^2)^{-\sigma/2}$ and $Z^\infty:=Z$. 
There is $\beta>0$ such that, for each $r\in[0,2]$ and some $\epsilon_r,\epsilon'_r>0$:
\begin{align}
\sup_{L\in[3,\infty]} \E\bigg[\exp\Big[\epsilon_r\, \Big(\sup_{s\in[0,t]}  (s^{n\alpha}\wedge 1)\| \wick{(Z^L_s)^n}\|_{-n\alpha,\rho^n}\Big)^{r/n}\, \Big]\bigg]
&\lesssim
(1+ t)^{\beta}\exp\Big[\epsilon'_r\big\|\varphi_0\|_{-\alpha,\rho}^{r}\Big]
,
\label{eq_gaussian_bound_non0IC_rho_norm}
\end{align}
where the proportionality constant is independent of $t,\varphi_0,L$. 
\end{proposition}

We remark that the bound~\eqref{eq_gaussian_bound_non0IC_rho_norm} holds in particular with the optimal exponent $r=2$.
It is convenient to have the statement with $\|\varphi_0\|_{-\alpha,\rho}^r$ on the right-hand side with $r<2$, however,
as in Section~\ref{sec_main_argument} we will need to integrate the right-hand side under the $\varphi^4$ measure for which only stretched integrability
in Besov--H\"older norms is known (see Corollary~\ref{cor:nuL-tight}).

\begin{proposition}\label{prop_gaussian_difference}
Let $\alpha>0$ and $n\in\N\setminus\{0\}$. 
There is $c>0$ such that, for each $L\geq 12$, $t>0$ and each test function $f$ supported on $[-\frac23 L,\frac23 L]^2$:
\begin{align}
\E\Big[\, \big(\wick{Z_t^n}-\wick{(Z^L_t)^n},f\big)^2\, \Big]
&\lesssim 
\, (t^{-n\alpha}\vee 1) \|f\|_\alpha^2 (1+\|\varphi_0\|^{2n}_{-\alpha,\rho})\, e^{-cL^2/t}
,
\nnb
\E\Big[\, \big\|\wick{Z_t^n}-\wick{(Z^L_t)^n}\big\|_{-n\alpha,\, [-\frac23 L,\frac23 L]^2 }^2\, \Big]
&\lesssim 
\, (t^{-n\alpha}\vee 1) (1+\|\varphi_0\|^{2n}_{-\alpha,\rho})\, e^{-cL^2/t}
.
\label{eq_bound_Z-ZL_sec2}
\end{align}
 One also has, for each $\sigma'>\sigma$, 
 writing $\rho'=(1+|\cdot|^2)^{-\sigma'/2}$:
\begin{equation}
\E\Big[\, \big\|\wick{Z_t^n}-\wick{(Z^L_t)^n}\big\|_{-n\alpha,(\rho')^n}^2\, \Big]
\lesssim 
\frac{1}{L^{2n(\sigma-\sigma')}}\, (t^{-n\alpha}\vee 1)\,  (1+\|\varphi_0\|^{2n}_{-\alpha,\rho})
.
\label{eq_bound_Z-ZL_weights}
\end{equation}
\end{proposition}

\subsection{Da Prato--Debussche decomposition}
\label{sec_background}

We solve the SPDE  \eqref{e:SPDE} using the well-known method of Da Prato--Debussche \cite{MR2016604},
by decomposing the corresponding process as $\varphi=Z+v$ with $Z$ a solution to the Ornstein--Uhlenbeck process \eqref{e:OU}.
Thus $\varphi=Z+v$ should solve \eqref{e:SPDE} if $v$ solves the remainder equation
\begin{equation}
  \dot v
  = -Av -\mu v- \lambda(v+Z)^3 + a(\lambda)(v+Z)
  ,
\end{equation}
where the remainder equation is then interpreted in terms of the Wick powers of $Z$.
Ideed, using that $a_\epsilon(\lambda) = 3 a_\epsilon \lambda$ in $d=2$ and that the Wick powers \eqref{e:Wick-convention}
are defined in terms of $a_\epsilon$, the last equation becomes (in the limit $\epsilon \to 0$)
\begin{equation}
  \dot v
  = 
  -Av - \mu v-\lambda \qB{v^3+3v^2 Z + 3v\wick{Z^2}+ \wick{Z^3}}, 
\end{equation}
interpreted in the integral form:
\begin{equation} \label{e:Y-Duhamel}
  v_t = 
  \int_0^t e^{-(t-s)A} \Big[-\mu v - \lambda [v_s^3+3v_s^2 Z_s + 3v_s\wick{Z_s^2}+ \wick{Z_s^3}]\Big] \, ds
  .
\end{equation}
The initial condition $\varphi_0$ can be distributed into $(Z_0,v_0)$ such that $\varphi_0=Z_0+v_0$,
but for most purposes it will be convenient for us to make the choice
\begin{equation}\label{eq:initial-data-choice}
  Z_0 = \varphi_0, \qquad v_0=0.
\end{equation}
The analogous decomposition defines the solution $\varphi^L = Z^L+v^L$ of~\eqref{e:SPDE} on the torus $\T^d_L$,
where $Z^L$ is the space-periodic Ornstein--Uhlenbeck process \eqref{eq_heat_periodised}
and $\varphi^L_0$ is the periodised initial condition \eqref{e:phi0L}:
\begin{equation} \label{e:YL-Duhamel}
  v^L_t = 
  \int_0^t e^{-(t-s)A} \Big[-\mu v_s^L - \lambda [(v^L_s)^3+3(v^L_s)^2 Z^L_s + 3v^L_s\wick{(Z^L_s)^2}+ \wick{(Z^L_s)^3}\big]\Big] \, ds
  .
\end{equation}
In particular, all equations are coupled through the same periodised noise $W^L$ as in~\eqref{e:xiL}. 

The local wellposedness of \eqref{e:Y-Duhamel} on the torus is the main result of  \cite{MR2016604},
where it was also observed that the torus $\varphi^4$ measure $\nu_L$ is invariant for this dynamics.
This measure is defined by
\begin{equation} \label{e:phi42measure}
  \nu_L(d\varphi) \propto
  \exp\qa{-\int_{\T_L^d} \pa{\frac{\lambda}{4}\wick{\varphi^4} + \frac{\mu}{2} \wick{\varphi^2}} \, dx} \, d\nu^{\rm GFF}_L,
\end{equation}
where the Wick powers are defined with the same conventions as in Section~\ref{sec_gaussian_estimates} (corresponding to the $t\to\infty$ limit), i.e., counterterms
independent of $L$, and $\nu^{\rm GFF}$ is the stationary distribution of the Ornstein--Uhlenbeck process, namely the Gaussian measure with covariance $A^{-1}$.

For more background on the definition of  the finite volume $\varphi^4_2$ measure, see \cite{MR0489552} and \cite{HairerASA}.
We record the invariance as the following theorem
(also see \cite{MR3626040} for a more detailed discussion).

\begin{theorem} \label{thm:nuL-invariant}
  The measure $\nu_L$ is invariant for the $\varphi^4$ dynamics on $\T_L^d$.
\end{theorem}

\noindent {\bf Notation.}
From now on, $\E_{\nu_L}$ denotes the expectation with respect to the dynamics started from $\nu_L$. 
Abusing notation somewhat,
we also use $\E_{\nu_L}$ to denote expectation against the measure $\nu_L$ when there is no risk of confusion: $\E_{\nu_L}[F] := \E_{\nu_L}[F(\varphi^L_0)]$. 

\subsection{Local a priori bounds and consequences}

Using the method of \cite{MR4164267}, a priori bounds on the remainder equation are proved in Appendix~\ref{app:apriori}.
The following is a streamlined version for solutions of the remainder equation  \eqref{e:Y-Duhamel} that will be used throughout the paper.
We recall the weight
\begin{equation} \label{e:weight-bis}
  \rho(x)=(1+|x|^2)^{-\frac{\sigma}{2}}
  .
\end{equation}

\begin{theorem}\label{thm:apriori_bounds} 
  Let $\alpha' \in [0,1)$ (note that $0$ is included), let  $\alpha >0$ be small enough, and set $\eta= \frac{1+\alpha'}{1-3\alpha}$.
  For $L\in \N\cup\{\infty\}$, let $v^L$ be the solution of the remainder equation  \eqref{e:YL-Duhamel}.
  
  \smallskip
  \noindent
  (i)  
  For initial condition $v_0=0$ (as in our standard convention \eqref{eq:initial-data-choice}),
  for each $t\geq 1$ and each ball $B=B_1(x)\subset\R^2$ in the spatial variable:
\begin{equation}
\sup_{0 \leq s \leq t} \|v^L_s\|_{\alpha',B}
+ \sup_{\substack{0 \leq \bar s< s \leq t \\ |s-\bar s|\leq 1}} \frac{\|v^L_s-v^L_{\bar s}\|_B}{|\bar s-s|^{\alpha'/2}}
\lesssim 
1+ 
  \sup_{0<s\leq t}\max_{n=1,2,3} \Big\{
  \Big( 
  (s^{n\alpha}\wedge 1)\|\wick{(Z^L_s)^n} \|_{-n\alpha,2B} 
  \Big)^{\frac{1}{n}}
\Big\}^\eta .
\label{eq_desired_bounds-BB}
\end{equation}
Furthermore, for any $\sigma>0$, 
\begin{equation}
\sup_{0 \leq s \leq t } \|v^L_s\|_{\alpha',\rho}
+ \sup_{\substack{0 \leq \bar s< s \leq t \\ |s-\bar s|\leq 1}} \frac{\|v^L_s-v^L_{\bar s}\|_\rho}{|\bar s-s|^{\alpha'/2}}
\lesssim 
1+ 
  \sup_{0<s \leq t}\max_{n=1,2,3} \Big\{
  \Big(
  (s^{n\alpha}\wedge 1)\|\wick{(Z^L_s)^n} \|_{-n\alpha,\rho^{\frac{n}{\eta}}} 
\Big)^{\frac{1}{n}}
\Big\}^\eta 
.
\label{eq_desired_bounds-rhorho}
\end{equation}

  \smallskip
  \noindent
  (ii)
  For arbitrary initial condition $v_0 \in C^{-\alpha}(\rho)$ and $t\geq 1$, one also has
\begin{equation}
  \|v^L_t\|_{\alpha',\rho}
  \lesssim 
  1+ 
  \sup_{0<s \leq t}\max_{n=1,2,3} \Big\{
  \Big(
  (s^{n\alpha}\wedge 1)\|\wick{(Z^L_s)^n} \|_{-n\alpha,\rho^{\frac{n}{\eta}}} 
  \Big)^{\frac{1}{n}}
  \Big\}^\eta 
.
\label{eq_desired_bounds-rhorho-alternative}
\end{equation}
The implicit constants  are all independent of $t,L$ and $v_0, \varphi_0$.
\end{theorem}

The a priori bounds have a number of consequences that we need. The first consequence is that the SPDE on the full space $\R^2$ is globally well-posed
for initial data in $C^{-\alpha}(\rho)$ for $\alpha >0$ and $\rho = (1 +|x|^2)^{-\frac{\sigma}{2}} $ for any $\sigma >0$. This is essentially the main result of 
\cite{MR3693966} and  a version of this statement can also be found in \cite{MR3951704}.
Since we will need a version of this statement with sharper weights than the corresponding statement of  \cite{MR3693966} we provide it as a separate corollary.

\begin{corollary} \label{cor:infvol-existence}
  Let $\sigma,\alpha, \alpha'>0$ 
  with $\alpha$ small enough as in Theorem~\ref{thm:apriori_bounds} and $1>\alpha'>\alpha$, and let
    $\eta = \frac{1+\alpha'}{1-3\alpha}$. 
  For any $\varphi_0 \in C^{-\alpha}(\rho)$, 
  there is a unique solution $\varphi_t$ to the SPDE
  \eqref{e:SPDE} in the sense that $\varphi_t = v_t + Z_t$ where 
  $Z$ is the Ornstein--Uhlenbeck process \eqref{eq_heat} and $v \in \bbL^\infty_t (C^{\alpha'}(\rho^{\eta}))$ solves the remainder equation \eqref{e:Y-Duhamel}.
\end{corollary}

\begin{proof}
The argument for existence of solutions is similar to \cite[Theorem 8.1]{MR3693966} and uses approximation by periodic solutions on larger and larger tori $\mathbb{T}_L^2$ and compactness as we now show. 

Let $v^L$ be the solution to the remainder equation \eqref{e:YL-Duhamel} on $\mathbb{T}^2_L$, with initial datum $\varphi_0^L\in C^{-\alpha}(\rho)$ as defined in \eqref{e:phi0L} 
(see \cite[Section 6]{MR3693966} for the existence and uniqueness of $v^L$). 
Let $\alpha',\eta$ be as in Theorem~\ref{thm:apriori_bounds} and define $\tilde\rho:=\rho^{\eta}\lesssim \rho$.  
Then $\|\varphi^L_0\|_{-\alpha,\tilde\rho}\lesssim \|\varphi^L_0\|_{-\alpha,\rho}$ and 
Proposition~\ref{prop_gaussian_bounds} gives:
\begin{equation}\label{e:compactness-for-vL}
\sup_{\substack{L\in\N \\ L\geq 3}}   
\sup_{0<t<T}(t^{\alpha n} \wedge 1)  \| \wick{(Z^L_t)^n}\big\|_{-n\alpha, \rho^{n}}
=
\sup_{\substack{L\in\N \\ L\geq 3}}   
\sup_{0<t<T}(t^{\alpha n} \wedge 1)  \| \wick{(Z^L_t)^n}\big\|_{-n\alpha, \tilde\rho^{\frac{n}{\eta}}}
 < \infty \qquad \text{a.s.}
\end{equation}
Plugging this information into \eqref{eq_desired_bounds-rhorho} we obtain
\begin{equation}
\sup_{\substack{L\in\N \\ L\geq 3}} 
\sup_{s\in [0,T]} \|v^L_s\|_{\alpha',\tilde \rho} < \infty,\qquad
\sup_{\substack{L\in\N \\ L\geq 3}} \sup_{\substack{0 \leq \bar s< s \leq T \\ |s-\bar s|\leq 1}} \frac{\|v^L_s-v^L_{\bar s}\|_{\tilde \rho}}{|\bar s-s|^{\alpha'/2}}
<
\infty
\qquad 
\text{a.s.}
\end{equation}
By the Arzela--Ascoli theorem there is a sequence $L \to \infty$ along which $v^L$ converges uniformly in $[0,T]\times\R^2$ with spatial weight $\tilde\rho$ to a limit $v$. 
The bound  $\sup_{L\geq 3}\sup_{s\in[0,T]}\|v^L_s\|_{\alpha',\tilde \rho}<\infty$ implies
that the limit satisfies $\sup_{s\in [0,T]} \|v_s\|_{\alpha',\tilde \rho}<\infty$ almost surely. 

To see that the limit $v$ satisfies the correct remainder equation, we start from
Proposition \ref{prop_gaussian_difference} which gives for any $n = 1,2,3$:
\begin{equation}
\lim_{L\to\infty}\E\Big[\, \int_0^T (t^{n\alpha}\wedge 1)\big\|\wick{Z_t^n}-\wick{(Z^L_t)^n}\big\|_{-n\alpha,\rho^n}\, dt \Big] 
=
0.
\end{equation}
There is therefore a further subsequence along which the  convergence
$\|\wick{Z_t^n}-\wick{(Z^L_t)^n}\|_{-n\alpha,\rho^n} \to 0$ takes places almost surely.  
 The topology is strong enough to pass to take a pointwise limit in~\eqref{e:YL-Duhamel} so that the limit $v$ indeed satisfies the correct remainder equation.

 The uniqueness argument in \cite[Section 9]{MR3693966} requires only weaker assumptions on the solution $v$ than we have established.
 Indeed, let $v^{(1)}, v^{(2)} \in  \bbL^\infty_t (C^{\alpha'}(\rho^{\eta}))$ be two solutions of the remainder equation  
\eqref{e:Y-Duhamel}. Then $E = v^{(1)} -  v^{(2)}$ solves 
\begin{equation}\label{eq:error_uniqueness}
E_t = \int_0^t   e^{-(t-s)A}\big( E_s U_s \big) ds ,
\end{equation}
where 
\begin{equation}
U_s = - \mu - \lambda \Big[ \big((v^{(1)})^2  + v^{(1)} v^{(2)} + (v^{(2)})^2 \big) + 3 (v^{(1)}+ v^{(2)}) Z   +3 \wick{Z^2} \Big].
\end{equation}
The assumption $v^{(1)}, v^{(2)} \in  \bbL^\infty_t (C^{\alpha'}(\rho^{\eta}))$, the estimate \eqref{e:compactness-for-vL} on $Z$ and $\wick{Z^2}$ as well as the multiplicative inequality \eqref{e:besov-mult-rho} imply that 
\begin{equation}\label{eq:potential_bound}
\sup_{0<t<T}(t^{2\alpha} \wedge 1)  \| U_t\big\|_{-2\alpha, \rho^{1+\eta}} < \infty   \qquad 
\text{a.s.}
\end{equation}
Now \cite[Theorem~9.1]{MR3693966} asserts that under an a priori regularity assumption on $E$ (which is implied by our $E\in  \bbL^\infty_t (C^{\alpha'}(\rho^{\eta}))$) and on $U$ (which is implied by \eqref{eq:potential_bound})
any solution $E$ to \eqref{eq:error_uniqueness}  vanishes.
This is shown by iterating \eqref{eq:error_uniqueness} where in each step one obtains a slightly better small constant (from the integration) at the price of a slightly worse spatial weight. 
Under our regularity assumptions, the small constant from iterated integration ultimately beats the effect of the deteriorating weight, see \cite[Proposition 22, Lemma 15]{MR3693966} for details. 
\end{proof}

Secondly, we will also apply bounds on the Wick powers of the $\varphi^L_t$ field ($t\geq 0, L\in (0,\infty]$),  
defined as in \eqref{eq_link_wick_powers_w_w/o_IC}:
\begin{equation}
\wick{(\varphi_t^L)^n}
=
\wick{(Z^L_t + v^L_t)^n}
=
\sum_{\ell=0}^n \binom{n}{\ell} \wick{(Z_t^L)^\ell} \ (v_t^L)^{n-\ell}
.
\label{eq_def_Wick_varphi}
\end{equation}

\begin{corollary}\label{cor:Wick-power-moment}
  Let $\alpha,\sigma>0$ be small enough, let $\alpha'>\alpha$, and let $n \geq 1 $ be an integer,
  and define $\eta_n=\frac{1+n\alpha'}{1-3\alpha}$. 
  Then there exists $\beta>0$ such that for all $r \in (0,2]$ and $\theta>\eta_n$, 
  there are $\varepsilon_r$, $\varepsilon_r'>0$ such that
\begin{equation} \label{e:apriori_wick}
  \E\qa{  \exp \Big[ \varepsilon_r    \Big(\sup_{s\leq t}(s^{n\alpha} \wedge 1)\|\wick{\varphi_s^{n}}\|_{-n\alpha,\rho^{n\eta_n}}\Big)^{\frac{r}{n \theta}} \Big]  } \lesssim  (1+ t)^{\beta}\exp\Big[\epsilon_r' \|\varphi_0\|_{-\alpha,\rho}^{r}\, \Big] .
\end{equation}
The implicit constant  is independent of $t,\varphi_0,L$.

Moreover,~\eqref{e:apriori_wick} also holds for $(\varphi^L_s)_{s\leq t}$ and with the $\exp$ function replaced on both sides by $x\mapsto 1+x^p$, 
for any $p\geq 1$.
\end{corollary}

\begin{proof}
Write for short $\rho' = \rho^{\eta_n}$. Then,
   by the multiplicative inequality for weighted Besov spaces \eqref{e:besov-mult-rho}, 
   for any $s\leq t$:
\begin{align}
(s^{n\alpha}\wedge 1)\, \|\wick{\varphi_s^{n}}\|_{-n\alpha,(\rho')^n} 
& 
\lesssim 
\sum_{\ell=0}^n  (s^{\ell\alpha}\wedge 1)\, \| \wick{Z_s^\ell} \|_{-\ell \alpha, (\rho')^{\ell} }  \| v_s \|_{ \ell \alpha', {\rho'} }^{n-\ell} 
\nnb
&\lesssim  
\max_{1\leq \ell\leq n} (s^{\ell\alpha}\wedge 1)\, \| \wick{Z_s^\ell} \|_{-\ell \alpha, (\rho')^{\ell} }^{\frac{n}{\ell}}  +   \| v_s \|_{ n\alpha', {\rho'} }^{n}
.
\label{e:phi-n-dec}
\end{align}
Let $r\in(0,2]$. 
For the $\wick{Z^\ell}$ terms on the right-hand side of \eqref{e:phi-n-dec},  Proposition~\eqref{prop_gaussian_bounds} yields, for some $\epsilon_r,\epsilon_r>0$ (note the weight $\rho$ rather than $\rho'\leq \rho$ since the proposition is valid for any value of $\sigma>0$):
\begin{equation}
 \E\bigg[ \exp\Big[\epsilon_r \sup_{s\in[0,t]}  \Big(    \max_{1\leq \ell\leq n} (s^{\ell\alpha}\wedge 1)\, \| \wick{(Z_s)^\ell}\|^{\frac{n}{\ell}}_{-\ell\alpha,\rho^\ell} \Big)^{\frac{r}{n}}\Big]\bigg] 
\lesssim
(1+ t)^{\beta}\exp\Big[\epsilon_r' \|\varphi_0\|_{-\alpha,\rho}^{r} \Big] .
\label{eq_concentration_Z_cor27}
\end{equation}
In particular the same estimate is true with weight $\rho'\leq \rho$.

For the  second term on the right hand side of \eqref{e:phi-n-dec} we invoke \eqref{eq_desired_bounds-rhorho} (with $n\alpha'$ in the role of $\alpha'$ and $\rho'$ in the role of $\rho$)
\begin{align}
\sup_{s\in [0,t]} \|v_s\|_{n\alpha',\rho'}
&\lesssim 
1+ 
  \sup_{s\leq t}\max_{m=1,2,3} \Big\{
  \Big(
  (s^{m\alpha}\wedge 1)\|\wick{(Z_s)^m} \|_{-m\alpha,(\rho')^{\frac{m}{\eta_n}}} 
\Big)^{\frac{1}{m}}
\Big\}^{\eta_n} 
\nnb
&=
1+ 
  \sup_{s\leq t}\max_{m=1,2,3} \Big\{
  \Big(
  (s^{m\alpha}\wedge 1)\|\wick{(Z_s)^m} \|_{-m\alpha,\rho^{m}} 
\Big)^{\frac{1}{m}}
\Big\}^{\eta_n} 
.
\end{align}
The bound~\eqref{eq_concentration_Z_cor27} then implies the claim. 
\end{proof}

Finally, 
Theorem~\ref{thm:apriori_bounds} implies that invariant measures are tight in $C^{-\alpha}(\rho)$.
In two dimensions,
the existing tightness results again apply in somewhat larger spaces (though the restriction is not for serious reasons),
see \cite{MR3693966,Shen2021AnSA,MR4252872}.
For our applications, the tightness in $C^{-\alpha}(\rho)$ with the optimal weight $\rho(x)=(1+|x|^2)^{-\sigma/2}$ with arbitrarily small $\sigma>0$ will be useful.

\begin{corollary} \label{cor:nuL-tight}
  Let $\sigma,\alpha>0$ be small enough. 
  For any $\theta\in(0,2)$ and some $\epsilon_\theta>0$,
  \begin{equation} \label{e:nu-expmoments}
    \sup_L \E_{\nu_L}\qa{\exp\Big[ \varepsilon_\theta\|\varphi\|_{-\alpha,\rho}^{2-\theta}\Big]} <\infty.
  \end{equation}
  In particular, the family $(\nu_{L})$ extended periodically to $\R^2$ is tight in $C^{-\alpha}(\rho)$
  and any limit point satisfies the bound \eqref{e:nu-expmoments} as well as the analogue of \eqref{e:apriori_wick}:
  for $\varepsilon_\theta>0$ small enough,
  \begin{equation}\label{eq:moment-estimates-Wick-powers}
    \sup_L \E_{\nu_L}  \qa{ \exp  \Big[ \varepsilon_\theta \|  \wick{\varphi^{n}}\|_{-n\alpha,\rho}^{\frac{2-\theta}{n}} \Big]} <\infty.
  \end{equation}
\end{corollary}

\begin{proof}
To get \eqref{e:nu-expmoments}, it suffices to show that, for any $\theta\in(0,2)$ and some $C_\theta>0$, for some $t>0$:
\begin{equation}
  \sup_L \E_{\nu_L}\qa{\exp\Big[ \varepsilon_\theta \|\varphi_t\|_{-\alpha,\rho}^{2-\theta}\Big]} < \infty
,
\label{eq_toprove_tightness}
\end{equation}
where $\E_{\nu_L}$ denotes the expectation of the dynamics started at $\varphi_0^L \sim \nu_L$. 
Indeed, if~\eqref{eq_toprove_tightness} holds, 
then by invariance $\varphi_t$ is also distributed according to $\nu_L$ so that
\eqref{e:nu-expmoments} follows.
The embedding $C^{-\alpha}(\rho) \subset C^{-\alpha'}(\rho')$ is compact for $\alpha'>0$ and $\rho'=(1+|x|^2)^{-\sigma'/2}$ with $\sigma'>\sigma$,
see Proposition~\ref{prop:arzela}.
Hence $K_r = \{\|\varphi\|_{-\alpha,\rho} \leq r\}$ is compact in $C^{-\alpha'}(\rho')$ and $\sup_L\nu_L(K_r^c) \to 0$
as $r\to\infty$
by \eqref{e:nu-expmoments}.
Thus $(\nu_L)_L$ is a tight family of probability measures on $C^{-\alpha'}(\rho')$.
Since $\alpha,\sigma>0$ are arbitrary, the statement follows.

To prove~\eqref{eq_toprove_tightness}, say for $t=3$,
decompose $\varphi^L_t$ as $\varphi^L_t = \tilde Z^L_t + \tilde v^L_t$, 
where now it is convenient to include the initial condition of the dynamics in $\tilde v^L_t$:
\begin{equation}
\tilde v^L_0
=
\varphi^L_0,\qquad 
\tilde Z^L_0
=
0
.
\end{equation}
This alternative decomposition is convenient here, because it allows to invoke
Theorem~\ref{thm:apriori_bounds}~(ii) which applies with general initial datum $v_0$.
Indeed, $\tilde{v}$ satisfies the remainder equation \eqref{e:Y-Duhamel}, with the processes $Z_t$, $\wick{Z_t^2}$ and $\wick{Z_t^3}$ replaced by 
$\tilde{Z}_t$, $\wick{\tilde{Z}_t^2}$ and $\wick{\tilde{Z}_t^3}$ defined in \eqref{e:def-Gaussian-without-Initial-datum} and \eqref{e:Wick-convention} above (or equivalently $Z_t$, $\wick{Z_t^2}$ and $\wick{Z_t^3}$ 
for initial datum $\varphi_0=0$). For these processes Proposition~\ref{prop_gaussian_bounds} takes the form  
\begin{equation}
\sup_{L\in[3,\infty]} \E\bigg[\exp\Big[\epsilon_2\, \Big(\sup_{s\in[0,t]}  (s^{n\alpha}\wedge 1)\| \wick{(\tilde{Z}^L_s)^n}\|_{-n\alpha,\rho^n}\Big)^{\frac{2}{n}}\, \Big]\bigg]
\lesssim
(1+ t)^{\beta}.
\end{equation}
Using this estimate, the a priori bound \eqref{eq_desired_bounds-rhorho-alternative} implies
\eqref{eq_toprove_tightness},
and the argument for \eqref{eq:moment-estimates-Wick-powers} is analogous to the proof of Corollary~\ref{cor:Wick-power-moment} in this modified decomposition so we omit it.
\end{proof}

\subsection{Markov process of the $\varphi^4_2$ dynamics}

The torus $\varphi^4_2$ dynamics (in the Da Prato--Debussche sense, see Section~\ref{sec_background})
coincides with the Markov process generated by the standard Dirichlet form
associated with the torus $\varphi^4$ measure. This was established in \cite{MR3626040} and is recalled now.

The state space of the Markov process can be taken to be $C^{-\alpha}(\T_L^d)$.
Given an initial condition $\varphi_0=\varphi \in C^{-\alpha}(\T_L^d)$ and $F: C^{-\alpha}(\T_L^d) \to \R$ bounded and Borel measurable, define the semigroup
\begin{equation} \label{e:semigroup}
  T_tF(\varphi)=\E_{\varphi_0=\varphi}[F(\varphi_t)].
\end{equation}
Let $\cF C_b^\infty(\T_L^d)$ denote the set of cylinder functions:
\begin{equation} \label{e:cylinder}
\cF C_b^\infty(\T_L^d)
:=
\bigcup_{n\in\N}\Big\{ \varphi\mapsto G((\varphi,f_1),\dots,(\varphi,f_n)): 
G\in C^\infty_b(\R^n,\R), f_i \in C^\infty_b(\T^2_L)\Big\}
.
\end{equation}
For any cylinder function $F \in \cF C_b^\infty(\T_L^d)$ as above, the $L^2(\T_L^d)$ gradient $\nabla F$ is defined by
\begin{equation}
  [\nabla F(\varphi)](x) = \sum_{i=1}^n \partial_i G( (\varphi,f_1), \dots, (\varphi,f_n)) f_i(x).
\end{equation}
The standard Dirichlet form is defined by
\begin{equation} \label{e:Dirichlet-continuum}
  D(F,G) = \E_{\nu_L}\qB{(\nabla F,\nabla G)_{\bbL^2(\T^2_L)}} \qquad
  \text{for } F, G \in \cF C^\infty_b(\T_L^d).
\end{equation}
The following was proven in \cite{MR1113223,MR3626040}, see \cite[Theorem~3.9]{MR3626040} and the discussion preceding it.

\begin{theorem} \label{thm:Dirichlet}
  The Dirichlet form  \eqref{e:Dirichlet-continuum} is closable on $\bbL^2(H, \nu_L)$,
  where $H$ is a suitable negative regularity Sobolev space on $\T_L^d$, 
  and it generates an associated Markov process.
  For initial conditions $\varphi_0 \in C^{-\alpha}(\T_L^d)$ with $\alpha>0$ sufficiently small,
  the semigroup of this Markov process   is given by the semigroup \eqref{e:semigroup} associated with the SPDE.
\end{theorem}

\subsection{Log-Sobolev inequality}

The next theorem states that the Markov processes associated with the $\varphi^4_2$ dynamics on the torus satisfies a uniform log-Sobolev inequality
and converges exponentially to equilibrium for initial conditions with finite relative entropy.
For a probability measure $\nu$ and a nonnegative measurable function $F$ defined
on the probability space on which $\nu$ is defined, we denote the relative entropy by
\begin{equation}
  \ent_\nu(F) = \E_{\nu}[F \log F] - \E_\nu[F]\log \E_\nu[F].
\end{equation}
Given a probability measure $\mu$ with $d\mu/d\nu =F$ 
we also write the relative entropy as
\begin{equation} \label{e:H-ent-def}
  \bbH(\mu|\nu) =\ent_{\nu}(F),
\end{equation}
with the convention that $\bbH(\mu|\nu) = +\infty$ if $\mu$ is not absolutely continuous with respect to $\nu$.

\begin{theorem} \label{thm:LS}
  For $d\in\{2,3\}$, 
  let $\chi_L(\lambda,\mu)=|\T^d_L|^{-1}\E_{\nu_L}[(\varphi,1)^2]$, let $\bar\chi>0$, and assume that $\chi_L(\lambda,\mu)<\bar\chi$.  
  Then there is $\gamma>0$ depending only on $(\lambda,\mu,\bar\chi)$ but not on $L$ such that, 
  for every nonnegative cylinder function $F \in \cF C_b^\infty(\T_L^d)$ defined in \eqref{e:cylinder},
\begin{equation} \label{e:LS-cylinder}
  \ent_{\nu_{L}}(F) \leq \frac{2}{\gamma}\,  \E_{\nu_{L}} \qa{\|\nabla \sqrt{F}\|_{\bbL^2(\T_L^2)}^2}.
\end{equation}
As a consequence, if $m_t^L$ denotes the law of the $\varphi^4_2$ SPDE on $\T_L^d$ at time $t\geq 0$, for any $0<s\leq t$,
\begin{equation}
  \bbH(m^L_t|\nu_L) \leq e^{-2\gamma (t-s)} \bbH(m^L_s|\nu_L)
  .
  \label{eq_decay_entropy}
\end{equation}
\end{theorem}

The lattice discretised version of the statement of the theorem was proved in \cite{MR4720217}, uniformly in lattice spacing.
Lattice approximations of the SPDE on the discretised torus $\Lambda_{\epsilon,L} = L\T^d \cap \epsilon\Z^d$
are defined analogously to \eqref{e:SPDE} and have unique invariant measures $\nu_{\epsilon,L}$ given by
\begin{equation}
  \nu_{\epsilon,L}(d\varphi) \propto e^{-H_{\epsilon,L}(\varphi)}\prod_{x\in\Lambda_{\epsilon,L}} d\varphi(x),
\end{equation}
where
\begin{equation}
  H_{\epsilon,L}(\varphi) = \frac 12\|\nabla_\epsilon \varphi\|_{L^2(\Lambda_{\epsilon,L})}^2
  + \frac{\lambda}{4}\|\varphi\|_{L^4(\Lambda_{\epsilon,L})}^4 + \frac{1+\mu-a_\epsilon(\lambda)}{2} \|\varphi\|_{L^2(\Lambda_{\epsilon,L})}^2
\end{equation}
and
\begin{equation}
  \|f\|_{L^p(\Lambda_{\epsilon,L})} = \qa{\epsilon^d \sum_{x\in\Lambda_{\epsilon,L}} |f(x)|^p}^{1/p}.
\end{equation}
Standard arguments (see, for example, \cite{MR2352327}) show that the SDE of the lattice approximation is
the Markov process defined by the Dirichlet form
\begin{equation}
  \E_{\nu_{\epsilon,L}} \qa{\|\nabla_\epsilon F\|_{\bbL^2(\Lambda_{\epsilon,L})}^2},
  \qquad
  \nabla_\epsilon F(\varphi,x) = \epsilon^{-d}\ddp{F(\varphi)}{\varphi(x)},
\end{equation}
i.e., $\nabla_\epsilon$ denotes the gradient with respect to the $\bbL^2(\Lambda_{\epsilon,L})$ inner product.
The following uniform log-Sobolev inequality for the lattice regularised dynamics was proved in \cite{MR4720217}.
Assume
\begin{equation}
  \chi_{\epsilon,L}(\mu,\lambda) = \frac{\E_{\nu_{\epsilon,L}}[(\varphi,1)^2]}{|\T_L^d|} = \epsilon^d \sum_{x\in\Lambda_{\epsilon,L}} \E_{\nu_{\epsilon,L}}\big[\varphi(0)\varphi(x)\big]  \leq \bar\chi <\infty
  .
\end{equation}
Then there is a constant $\gamma>0$ depending only on the parameters $\lambda,\mu$ in \eqref{e:SPDE} and on $\bar\chi$, 
but independent of $\epsilon$ and $L$, 
such that for every $F: \Lambda_{\epsilon,L} \to \R_+$, the following log-Sobolev inequality holds:
\begin{equation} \label{e:LS-eps}
  \ent_{\nu_{\epsilon,L}}(F) \leq \frac{2}{\gamma} \E_{\nu_{\epsilon,L}} \qa{\|\nabla_\epsilon \sqrt{F}\|_{\bbL^2(\Lambda_{\epsilon,L})}^2}.
\end{equation}

\begin{proof}[Proof of Theorem~\ref{thm:LS}]
  It is well known that $\nu_L$ is the weak limit of $\nu_{\epsilon,L}$ extended appropriately to a probability measure on $C^{-\alpha}(\T_L^d)$ as $\epsilon\to 0$
  and, for any smooth bounded $f:\T^d_L\to\R$, $(\varphi,f)^2$ is uniformly integrable under $\nu_{\epsilon,L}$ as $\epsilon \to 0$, 
  see~\cite[Section 8]{MR723546}.
  Thus   $\E_{\nu_{\epsilon,L}}\q{(1,\varphi)^2} \to \E_{\nu_{L}}\q{(1,\varphi)^2}$ and for any $\bar\chi> \chi_L(\lambda,\mu)$
  and $\epsilon>0$ sufficiently small,   $\chi_{\epsilon,L}(\lambda,\mu) \leq \bar\chi$.
  For every cylinder functional $F$ that is also bounded below,
  it is easy to verify that the left- and right-hand sides of \eqref{e:LS-eps} converge to their continuum versions.
  The lower bound on test functions can be removed in a standard way, 
  approximating $F\in\cF C^\infty_b(\T^d_L)$ by $F_\eta=F\vee\eta$ with $\eta\searrow 0$ and using the lower semicontinuity of the relative entropy. 
  The log-Sobolev inequality~\eqref{e:LS-cylinder} for cylinder functionals thus follows from this and the
  log-Sobolev inequality \eqref{e:LS-eps} for $\nu_{\epsilon,L}$, recalling that the constant there depends only on $(\lambda,\mu,\bar\chi)$ but not directly on $\epsilon$ (or $L$).
  
  Using that the $\varphi^4_2$ SPDE dynamics coincides with the Markov process defined by the
  standard Dirichlet form, see Theorem~\ref{thm:Dirichlet},
  the entropy decay \eqref{eq_decay_entropy} is a standard consequence of the log-Sobolev inequality (see~\cite[Definition 5.1.1 and Theorem 5.2.1]{MR3155209}) 
  and the fact that the cylinder functionals \eqref{e:cylinder} are dense in the domain of the Dirichlet form. 
\end{proof}

\section{Main argument}\label{sec_main_argument}

The main result follows a version of the strategy proposed by Holley--Stroock \cite{MR893137}
and relies on a uniform (in the size $L$) bound on the log-Sobolev constant for the space-periodic dynamics~\eqref{e:SPDE-eps}. 

\begin{enumerate}
	\item We first show that the infinite volume dynamics~\eqref{e:SPDE} at time $t> 0$ can be approximated by the periodised dynamics~\eqref{e:SPDE-eps} on a large enough box of side length $L \gg t^C$ for a suitable constant $C>0$, 
	with an error that vanishes when $t \gg 1$ (Theorem~\ref{claim:propagation}). 
      \item In finite volume, we use the uniform in the volume log-Sobolev inequality of Theorem~\ref{thm:LS} to prove exponential convergence to equilibrium.
        The log-Sobolev inequality implies convergence for random initial conditions with finite entropy, and
        to apply it with determinstic initial condition
        requires a proof that the entropy of the law of the dynamics starting at a deterministic initial condition is finite at time $1$ (Theorem~\ref{thm:ent-claim}). 
\end{enumerate}
      
In the following statements, the dimension is always $d=2$,
the weight is $\rho(x)=(1+|x|^2)^{-\sigma/2}$ on $\R^2$ with $\sigma>0$,
and the Besov--H\"older space $C^{-\alpha}(\rho)$ with norm $\|\cdot\|_{-\alpha,\rho}$ is as in Section~\ref{sec:norms-def}.
Given a compactly supported test function $f: \R^2 \to \R$ we denote by $R_f$ the radius of the smallest ball containing the support of $f$:
\begin{equation}
  R_f = \inf\{r >0: \supp(f) \subset B_r(0)\}.
\end{equation}
The coupling constants $\lambda>0$ and $\mu \in \R$ of the $\varphi^4$ dynamics are fixed (and arbitrary),
and all constants are permitted to depend on them as well as on $\alpha,\sigma$
(but not on $L$, $t$ or the initial condition of the dynamics). The condition on $(\lambda,\mu)$ that restricts to the high temperature phase
only enters through the uniformity of the log-Sobolev constant in the assumption
that the conclusion of Theorem~\ref{thm:LS} holds, imposed in Corollary~\ref{cor:unique} below.

\begin{theorem}[Propagation speed estimate]  \label{claim:propagation}
Let $\lambda >0$ and $\mu \in \R$, and let $\alpha,\sigma>0$ be sufficiently small.
Then there are constants $\delta, C_0,C >0$ such that for any initial condition  $\varphi_0\in C^{-\alpha}(\rho)$, 
any test function $f\in C_c^\infty(\R^2)$, and any $t>0$, $L>0$ with $L\geq C_0(t^{C_0}+R_f+1)$,
\begin{equation} \label{e:propagation}
  \absa{\E\qb{e^{i(f,\varphi_t)} -e ^{i(f,\varphi^L_t)}}}
  \lesssim
  \|f\|_\alpha (1+\|\varphi_0\|^3_{-\alpha,\rho})\, e^{-\delta L^2/t}+
  e^{C\|\varphi_0\|_{-\alpha,\rho}}\, e^{-t^2\log L},
\end{equation}
with the implied constant depending on parameters $\lambda,\mu$ appearing in the SPDE~\eqref{e:SPDE} and on $\alpha,\sigma$.
\end{theorem}

The proof of Theorem~\ref{claim:propagation} is the content of Section~\ref{sec:propagation}.
As a first application we observe that infinite volume limit points of the finite volume invariant measures $\nu_L$ are invariant measures.

\begin{corollary} \label{cor:nu-invariant}
  Any limit point $\nu$ of $(\nu_L)$ is an invariant measure of the SPDE \eqref{e:SPDE} on $\R^2$.
\end{corollary}
\begin{proof}
  Denote the finite and infinite volume semigroups associated with the SPDE by
  \begin{alignat}{2}
    T_t^LF(\varphi) &= \E_{\varphi_0=\varphi}[F(\varphi_t^L)], &&\qquad\varphi\in C^{-\alpha}(\T_L^d),
                      \nnb
    T_tF(\varphi) &= \E_{\varphi_0=\varphi}[F(\varphi_t)], &&\qquad \varphi\in C^{-\alpha}(\rho).
  \end{alignat}
  Using the finite volume invariance $\E_{\nu_L}[T_t^LF] = \E_{\nu_L}[F]$, see Theorem~\ref{thm:nuL-invariant},
  our goal is to show that an infinite volume limit point $\nu$ is also invariant: $\E_{\nu}[T_tF] = \E_\nu[F]$.
  It suffices to consider $F(\varphi)=e^{i(\varphi,f)}$ with $f\in C^\infty_c(\R^2)$ since these test functions characterise the measure.
  The claim essentially follows from the propagation speed estimate from Theorem~\ref{claim:propagation} and the uniform integrability proven in Corollary~\ref{cor:nuL-tight}.

  As a preliminary,
  we first use Theorem~\ref{claim:propagation} to show that $\varphi_0\in C^{-\alpha}(\rho)\mapsto T_tF(\varphi_0)$ is continuous for each $t>0$,
  as a consequence of the well known finite volume version of this property, see~\cite{MR3825880,HairerMattinglystrongFeller}.
  For $\varphi,\tilde\varphi\in C^{-\alpha}(\rho)$ and $L>0$ to be chosen below, write:
 \begin{align}
 T_tF(\tilde\varphi)-T_tF(\varphi)
 &=
 \big(T_tF(\tilde\varphi)-T^L_tF(\tilde\varphi^L)\big)
 +\big(T^L_tF(\varphi^L)-T_tF(\varphi)\big)
 \nnb
 &\qquad \qquad
 +\big(T^L_tF(\tilde\varphi^L)-T^L_tF(\varphi^L)\big)
 .
 \end{align}
 Given $\epsilon>0$ it suffices to show that $|T_tF(\varphi)-T_tF(\tilde\varphi)|\leq 3\epsilon$ if $\|\varphi-\tilde\varphi\|_{-\alpha,\rho}$ is sufficiently small.
 We may assume that $\|\tilde\varphi\|_{-\alpha,\rho} \leq 2\|\varphi\|_{-\alpha,\rho}$.
 Then Theorem~\ref{claim:propagation} implies that for sufficiently large~$L$ (depending on $\|\varphi_0\|_{-\alpha,\rho}$)
 the first two terms in the right-hand side above are bounded by~$\epsilon$.
 The continuity of the truncation operator $\tilde\varphi\mapsto \tilde\varphi^L$ in $C^{-\alpha}(\rho)$,
 see \eqref{eq_cv_varphi_L0}, together with
 the continuity of the finite volume semigroup
 imply that the last term above is also bounded by $\epsilon$ if $\|\varphi-\tilde\varphi\|_{-\alpha,\rho}$ is sufficiently small.
 This gives the desired continuity.

 We now prove that limit points are invariant measures.
 Fix a sequence $(\nu_L)$ converging to $\nu$ (possibly passing to a subsequence from the original family of measures).  
  For each fixed $t>0$, 
  the weak convergence of $(\nu_L)$ and the continuity of $T_tF$ give:
  \begin{equation}
  \E_{\nu}[F] - \E_{\nu}[T_tF]
  =
  \lim_{L\to\infty} \E_{\nu_L}[F] - \lim_{L\to\infty}\E_{\nu_L}[T_tF]
  .
  \end{equation}
  Theorem~\ref{claim:propagation} controls the second expectation in terms of the finite volume dynamics:
  \begin{equation}
  \lim_{L\to\infty} \E_{\nu_L}[T_tF]
  =
  \lim_{L\to\infty} \E_{\nu_L}[T_t^LF]
  ,
  \end{equation}
  where 
  we used Corollary~\ref{cor:nuL-tight} to argue that the expectation of the error term of Theorem~\ref{claim:propagation} under $\nu_L$ is $o_L(1)$.
  Invariance of $\nu_L$ for the finite volume dynamics concludes the proof:
  \begin{equation}
  \E_{\nu}[F] - \E_{\nu}[T_tF]
  =
  \lim_{L\to\infty} \pB{ \E_{\nu_L}[F] - \E_{\nu_L}[T_t^LF] } = 0
  .
  \end{equation}
\end{proof}

The next theorem shows that in finite volume the relative entropy of the law of the dynamics with respect  to the invariant measure
becomes finite after a finite time.

\begin{theorem} \label{thm:ent-claim}
  Let $\lambda>0$ and $\mu\in \R$,
  and let $\alpha,\sigma>0$ be sufficiently small.
  Then there is a constant $\beta>0$ such that for any deterministic initial condition $\varphi_0 \in C^{-\alpha}(\rho)$:
  \begin{equation} \label{e:ent-claim}
    \bbH(m_1^L|m_\infty) \lesssim (1+\|\varphi_0\|_{-\alpha,\rho}^{8}) L^{2\beta},
  \end{equation}
  where $m_t^L$ denotes the distribution of $\varphi_t^L$ under the SPDE~\eqref{e:SPDE-eps} on $\T_L^2$ with initial condition $\varphi_0$.
\end{theorem}

The proof of Theorem~\ref{thm:ent-claim} is given in Section~\ref{sec:entropy}.
The exponents $2\beta$ and $8$ on the right-hand side of \eqref{e:ent-claim} are certainly not optimal (but sufficient).
We remark that similarly crude bounds have been used even in the setting of lattice dynamics, see \cite[Lemma~2.2]{MR1370101}.

The uniqueness of the invariant measure stated in Theorem~\ref{thm:main-uniqueness} 
is a direct consequence of the following corollary, which follows from Theorems~\ref{claim:propagation} and~\ref{thm:ent-claim},
together with the uniform bound on the log-Sobolev constant of $\nu_L$ given by Theorem~\ref{thm:LS}.

\begin{corollary} \label{cor:unique}
  Assume the conclusion of Theorem~\ref{thm:LS} holds, i.e.,
  the log-Sobolev constant $\gamma$ of the $\varphi^4$ measure $\nu_L$ on $\T_L^2$ is uniform in $L$.
  Then the infinite-volume SPDE~\eqref{e:SPDE} has a unique invariant measure
  $\nu$ on $C^{-\alpha}(\rho)$, and for any deterministic initial condition $\varphi_0 \in C^{-\alpha}(\rho)$ and $f\in C_c^\infty(\R^2)$,
  for sufficiently large $t>0$ (depending on $f$):
  \begin{equation} \label{e:unique-rate}
    \Big|\E\big[e^{i(f,\varphi_t)}\big] -\E_{\nu}\big[e^{i(f,\varphi)}\big]\Big|
    \lesssim
    (1+\|\varphi_0\|_{-\alpha,\rho}^{4})e^{-\gamma t(1+o(1))}
    ,
  \end{equation}
  with the implied constants depending on $\lambda,\mu,\alpha,\sigma$.
\end{corollary}

\begin{proof}
  We first prove that there is a probability measure on $C^{-\alpha}(\rho)$ which the dynamics converges to, starting from any initial condition
  in this space, and then use it to conclude on uniqueness of the invariant measure.  
  By the assumption that the conclusion of Theorem~\ref{thm:LS} holds,
  the log-Sobolev constant of $\nu_L$ is bounded below by some $\gamma>0$, uniformly in $L$,
  and thus the relative entropy decays exponentially in time, uniformly in $L$. 
  Together with the entropy estimate at time $1$ stated in Theorem~\ref{thm:ent-claim} this implies
  that for any deterministic initial condition $\varphi_0 \in C^{-\alpha}(\rho)$, 
\begin{equation}
  \bbH(m_t^L|\nu_L) \leq e^{-2\gamma(t-1)} \bbH(m_1^L|\nu_L) \lesssim (1+\|\varphi_0\|_{-\alpha,\rho}^8) L^{2\beta} e^{-2\gamma(t-1)}.
\end{equation}
By Pinsker's inequality, 
\begin{equation}
  \|m_t^L-\nu_L\|_{\rm TV} \leq \sqrt{2\, \bbH(m_t^L|\nu_L)} \lesssim (1+\|\varphi_0\|_{-\alpha,\rho}^4) L^{\beta} e^{-\gamma t},
\end{equation}
where we have absorbed the factor $e^{\gamma}$ in the implicit constant. In particular, 
for any $f\in C^\infty_c$:
\begin{equation}
  \big|\E[e^{i(f,\varphi_t^{L})}]-\E_{\nu_{L}}[e^{i(f,\varphi)}]\big| \lesssim (1+\|\varphi_0\|_{-\alpha,\rho}^4)L^\beta e^{-\gamma t}.
\end{equation}
To satisfy the conditions of Theorem~\ref{claim:propagation}, 
we choose $L= L(R_f,t) = C_0 (t^{C_0} + R_f+1)$ with large enough constant $C_0>0$,
where $\delta>0$ is the constant appearing in Theorem~\ref{claim:propagation}
and $R_f$ is such that $f$ has support in the ball $B(0,R_f)$.
Then, for all $t\geq 1$ and all initial conditions $\varphi_0 \in C^{-\alpha}(\rho)$:
\begin{align}
\big|\E\big[e^{i(f,\varphi_t)}\big]
-\E_{\nu_{L(R_f,t)}}\big[e^{i(f,\varphi)}\big]\big|
&\leq 
(1+\|\varphi_0\|_{-\alpha,\rho}^{4}) L(R_f,t)^{\beta} e^{-\gamma t} 
\nnb
&\qquad + \|f\|_\alpha (1+\|\varphi_0\|^3_{-\alpha,\rho}) e^{-\delta L(R_f,t)^2/t}
  \nnb
&\qquad +  e^{C\|\varphi_0\|_{-\alpha,\rho}} e^{-t^2\log L(R_f,t)}
.
\label{eq_main_estimate}
\end{align}

By Corollary~\ref{cor:nuL-tight}, the family of probability measures $(\nu_L)$ on $C^{-\alpha}(\rho)$ is tight.
It follows that there is a limit point $\nu$ and a sequence $t_k \to \infty$ such that $\nu_{L(R_f,t_k)} \to \nu$ as $k\to \infty$,
weakly as probability measures on $C^{-\alpha}(\rho)$.
Together with \eqref{eq_main_estimate} this gives, 
for any  $\varphi_0\in C^{-\alpha}(\rho)$:
\begin{equation}
  \label{eq_main_estimate_bis}
  \lim_{k\to\infty} \E\big[e^{i(f,\varphi_{t_k})}\big] = \E_{\nu}\big[e^{i(f,\varphi)}\big] .
\end{equation}
We now show that $\nu$ is in fact the only invariant measure of the $\varphi^4$ SPDE on $\R^2$. 
The fact that it is indeed invariant is Corollary~\ref{cor:nu-invariant}. 
To show uniqueness, let $\tilde \nu$ be any invariant probability measure supported on $C^{-\alpha}(\rho)$.
By Corollary~\ref{cor:infvol-existence}, the SPDE on $\R^2$ has a global solution starting from any initial condition in $\supp\tilde \nu \subset C^{-\alpha}(\rho)$. 
By invariance in the first equality below, 
dominated convergence in the second equality, and~\eqref{eq_main_estimate_bis} in the third,
\begin{equation}
  \E_{\tilde \nu}[F]
  = \lim_{k\to\infty}  \E_{\tilde \nu}[T_{t_k} F]
  = \E_{\tilde \nu}[\lim_{k\to\infty} T_{t_k} F]
  = \E_{\tilde \nu}[\E_{\nu} [F]] = \E_{\nu}[F],
\end{equation}
where $F(\varphi)=e^{i(f,\varphi)}$ for any $f\in C_c^\infty$
and
given an initial condition $\varphi_0\in C^{-\alpha}(\rho)$ we denote by $T_tF(\varphi_0)=\E_{\varphi_{t=0}=\varphi}[F(\varphi_t)]$
the semigroup for the $\varphi^4$ dynamics on $\R^2$. Since such test functions $F$ characterise the measure, we conclude that $\tilde\nu=\nu$.

Finally, we observe that equation~\eqref{eq_main_estimate} also provides the rate of convergence
stated in \eqref{e:unique-rate}.
Indeed, by Corollary~\ref{cor:nuL-tight}, we have the bound:
\begin{equation}
\E_{\nu}\Big[ \qa{\exp\Big[ C_\theta\|\varphi\|_{-\alpha,\rho}^{2-\theta}\, \Big]} 
<\infty,
\qquad 
\theta\in(0,2)
.
\end{equation}
Thus \eqref{eq_main_estimate} can be averaged under $\varphi_0\sim \nu$.
Using that $\nu$ is invariant (Corollary~\ref{cor:nu-invariant}), we find:
\begin{equation}
\Big|\E_{\nu}\big[e^{i(f,\varphi)}\big] - \E_{\nu_{L(R_f,t)}}\big[e^{i(f,\varphi)}\big]\Big|
\lesssim
L(R_f,t)^{\beta}e^{-\gamma t} + \|f\|_{\alpha} e^{-\delta L(R_f,t)^2/t} +e^{-t^2\log L(R_f,t)} 
.
\end{equation}
Plugging this bound back into~\eqref{eq_main_estimate} yields:
\begin{align}
&\Big|\E_{\nu}\big[e^{i(f,\varphi)}\big] -\E_{\nu}\big[e^{i(f,\varphi_t)}\big]\Big|
\nnb
&\qquad
\lesssim
(1+\|\varphi_0\|_{-\alpha,\rho}^{4}) L(R_f,t)^{\beta}e^{-\gamma t} + \|f\|_\alpha e^{-\delta L(R_f,t)^2/t} +e^{-t^2\log L(R_f,t)} 
,
\end{align}
which implies the claim.
\end{proof}

\section{Propagation speed estimate}
\label{sec:propagation}

In this section we prove Theorem~\ref{claim:propagation}. 
Recall the Da Prato-Debussche decomposition $\varphi = Z+v$ from Section~\ref{sec_background}. 
Since $|e^{ia} -e^{ib}| \leq |a-b|$ and $|e^{ia}|\leq 1$, 
for any event $\Omega$,
\begin{align} \label{e:propagation-firststep}
  \absa{\E\qb{e^{i(f,\varphi_t)} -e ^{i(f,\varphi^L_t)}}}
  &\leq \E\qb{ |(Z_t+v_t-Z_t^L-v_t^L,f)| {{\bf 1}_\Omega}} + 2 \P\qb{\Omega^c}
    \nnb
  &\leq \E\qb{ (Z_t-Z_t^L,f)^2}^{1/2} + \E\qb{ |(v_t-v_t^L,f)| {\bf 1}_{\Omega}} + 2 \P\qb{\Omega^c}.
\end{align}

The required estimates for the Gaussian terms $Z-Z^L$ are given by Proposition~\ref{prop_gaussian_difference}.
Given $f \in C_c^\infty(\R^2)$, 
recall that $R_f$ denotes the radius of the smallest ball containing the support $f$.
Proposition~\ref{prop_gaussian_difference} implies that if $L \geq 2R_f\vee 12$ and $t\geq 1$ then
\begin{equation}
\E\Big[\big(Z_t-Z^L_t,f\big)^2\Big]
\lesssim
\|f\|_{\alpha}^2 (1+\|\varphi_0\|^{2}_{-\alpha,\rho})\, e^{-cL^2/t}
.
\end{equation}
The H\"older norm $\|\cdot\|_{\alpha}$ and its weighted and local H\"older versions $\|\cdot\|_{\alpha,\rho}$ and $\|\cdot\|_{\alpha,C}$
are defined in Section~\ref{sec:norms-def} and $\sigma$ will always refer to the exponent in the weight $\rho(x)=(1+|x|^2)^{-\sigma/2}$.

Thus it only remains to estimate the remainder contribution from $v-v^L$.  
We will be able to bound this difference on an event of probability close to $1$, 
which is responsible for the second term in the bound of Theorem~\ref{claim:propagation}.
This event is determined in the following lemma and in Proposition~\ref{prop_finite_PS}.

\begin{lemma}[General bounds]\label{lemm_general_bounds}
Let $\alpha,\sigma>0$, $\alpha'>\alpha$ and assume $\alpha$ is sufficiently small that $\eta=\frac{1+\alpha'}{1-3\alpha}$ is well defined. 
The following hold for all $t\geq 0$ and $L\geq \max\{3,t\}$. 

\smallskip
\noindent
(i)
There is $\kappa>0$ and an event $\Omega'_{t,L}(\varphi_0,\alpha,\alpha')$ with probability $1-O(e^{-t^2\log L})$ such that any $(\varphi_s)_{s\leq t}\in\Omega'_{t,L}(\varphi_0,\alpha,\alpha')$ satisfies:
\begin{align}
\sup_{s\leq t}\max\Big\{\|v_s\|_{\alpha',\rho^\eta}, \|v^L_s\|_{\alpha',\rho^\eta}\Big\}
&\leq 
(\log L)^{\kappa} (1+t)^\kappa,
\label{eq_bounds_YZ}
\\
\sup_{s\leq t}\, (s^{n\alpha}\wedge 1)\, \max\Big\{\, \|\wick{(Z_s)^n}\|_{-n\alpha,\rho^n}, \|\wick{(Z^L_s)^n}\|_{-n\alpha,\rho^n}\Big\}
&\leq
(\log L)^{n\kappa} (1+t)^{n\kappa},\qquad 
n\in\{1,2,3\}
\nonumber
.
\end{align}

\smallskip
\noindent
(ii)
More precisely, 
for any $\alpha,\alpha'$ small enough, 
one can choose $\kappa=5\eta$ and $\Omega'_{t,L}(\varphi_0,\alpha,\alpha')$ depending only on $\alpha,\alpha',\sigma$ such that, for some $\epsilon'_0>0$ and $C>0$ independent of $t,\varphi_0,L$: 
\begin{equation}
\bbP\big((\varphi_s)_{s\leq t}\notin \Omega'_t(\varphi_0,\alpha,\alpha')\big)
\leq 
C\exp\Big[ \epsilon'_0 \|\varphi_0\|_{-\alpha,\rho}\Big] e^{-t^2\log L }
.
\label{eq_bound_prob_Omega'}
\end{equation}
\end{lemma}
We track the dependence on $\alpha$ in the notation $\Omega'_{t,L}(\varphi_0,\alpha)$ since we will later need this event with different values of $\alpha$. 
Lemma~\ref{lemm_general_bounds} essentially follows from the a priori bounds of Sections~\ref{sec_gaussian_estimates}--\ref{sec_background} and is proven in Section~\ref{sec_proof_lem_general_bounds}.

\begin{remark}
All terms in~\eqref{eq_bounds_YZ} are on average bounded by some power of $\log(1+t)$. 
The much looser polynomial bounds stated  in~\eqref{eq_bounds_YZ}  are useful to ensure that the exceptional  event has negligible probability 
$O(e^{- t^2\log L})$, 
where the $t^2$ factor ensures that this terms vanishes mush faster than the speed at which the dynamics converges to its steady state in~\eqref{eq_main_estimate} (which is exponentially in $t$), 
and the factor $\log L$ is so that the finite propagation speed estimate of Theorem~\ref{claim:propagation} improves as $L$ becomes large for fixed $t$.
The $t^2$ in the exponential is however arbitrary, 
we could for instance have any $t^r$, $r\geq 2$ by taking a larger $\kappa$. The $\log L$ could similarly be improved, for instance to a small power of $L$.

The stretched exponential dependence on $\varphi_0$ is useful in Section~\ref{sec_main_argument} to get a rate of convergence of the dynamics. 
The fact that we can enforce it comes for free, again up to increasing $\kappa$. 

Note finally that for the Gaussian terms $\wick{(Z^L)^n}$ the value of $\kappa$ is explicit:  $\kappa= n$. 
The fact that $\kappa$ is not explicit above comes from the a priori bounds on $v$, Theorem~\ref{thm:apriori_bounds}.

\end{remark}
A bound on the difference $(v_t-v_t^L,f)$ when the bounds of Lemma~\ref{lemm_general_bounds} hold is
the main result of this section, and is stated next.

\begin{proposition}\label{prop_finite_PS}
  Let $\alpha,\sigma>0$ be sufficiently small.
  Let $\varphi_0 \in C^{-\alpha}(\rho)$, let $f \in C_c^\infty(\R^2)$ and let $t>0$. 
  There are then $C_0,\delta>0$ independent of $f,\varphi_0,t,L$   
  and an event $\Omega_{t,L}(\varphi_0)$ such that the following hold if $L\geq C_0 (t^{C_0}+R_f+1)$:

\smallskip\noindent (i)
There is $C_1$ independent of $f,t,L,\varphi_0$ such that, on $\Omega_{t,L}(\varphi_0)$,
\begin{equation}
|(v_t-v_t^L,f)|
\leq 
C_1\|f\|_{\alpha}(1+\|\varphi_0\|_{-\alpha,\rho}^3)\, e^{-\delta L^2/t}
.
\end{equation}

\smallskip\noindent (ii)
  There are $C_2,\epsilon_0>0$ independent of $f,t,L,\varphi_0$ such that the event $\Omega_{t,L}(\varphi_0)$ satisfies:
  \begin{equation}
  \bbP\big((\varphi_s)_{s\leq t}\notin \Omega_{t,L}(\varphi_0)\big)
   \leq 
  C_2\exp\Big[ \epsilon_0 \|\varphi_0\|_{-\alpha,\rho}\Big]\, e^{-t^2\log L}
  .
  \label{eq_bound_prob_Omega}
\end{equation}
\end{proposition}
\begin{remark}
The statement of Proposition~\ref{prop_finite_PS} is in fact valid even if the infinite volume remainder equation~\eqref{e:Y-Duhamel} admits several different solutions $v_t$, 
in which case it holds for each of them. 
\end{remark}
Before proving Proposition~\ref{prop_finite_PS}, note that Theorem~\ref{claim:propagation}
follows from the proposition by choosing $\Omega = \Omega_{t,L}(\varphi_0)$  in \eqref{e:propagation-firststep}:
\begin{equation} \label{e:propagation-bis}
  \absa{\E\qb{e^{i(f,\varphi_t)} -e ^{i(f,\varphi^L_t)}}}
  \lesssim
  \|f\|_\alpha (1+\|\varphi_0\|^3_{-\alpha,\rho}) \, e^{-\delta L^2/t}+
  \exp\Big[ \epsilon_{0}\|\varphi_0\|_{-\alpha,\rho}\Big]\, e^{-t^2\log L}.
\end{equation}

\subsection{Proof of Proposition~\ref{prop_finite_PS}}\label{subsec_prop_finitePS}

Define $E:=v-v^L$ and recall $E_0=0$ as the initial condition is in the Gaussian part of the Da Prato-Debussche decomposition. 
The quantity $E$ then satisfies the following equation: 
for each $s\leq t$,
\begin{multline}
  E_t = \int_0^t e^{(t-s)\Delta}\bigg( - \lambda E_s \underbrace{\qB{v^2_s+(v^L_s)^2+v_sv^L_s + 3(v+v^L_s)Z_s + 3\wick{Z^2}+(\mu+1)/\lambda}}_{U_s}
  \\
  - \lambda \underbrace{\qB{3v^2_s (Z-Z^L_s) + 3v_s (\wick{Z^2_s}-\wick{(Z^L_s)^2}) + \wick{Z^3_s}-\wick{(Z^L_s)^3}}}_{F_s}\bigg)\, ds
  .
  \label{eq_diff_Y_YL}
\end{multline}
Fix $f\in C^\infty_c(\R^2)$ and $t\geq 0$.  

Let $(\Pg_{s,t})_{0\leq s\leq t}$ be the inhomogeneous semigroup with inhomogeneous generator $\Delta-\lambda U_{t}$, i.e.\ for each test function $g$: 
\begin{equation}
\partial_s \Pg_{s,t}g
=
(-\Delta+\lambda U_{s})\Pg_{s,t}g,
\qquad 
s\leq t
.
\label{eq_def_Pg}
\end{equation}
As a mild solution of~\eqref{eq_diff_Y_YL}, 
$E_\cdot$ also solves~\eqref{eq_diff_Y_YL} weakly and we can write:
\begin{equation}
(E_t,f)
=
(E_t,\Pg_{t,t}f) 
= 
-\lambda \int_0^t \big(F_s,\Pg_{s,t}f\big)\, ds
.
\label{eq_FK_for_Et}
\end{equation}

To estimate the right-hand side of~\eqref{eq_FK_for_Et}, 
we split the integral against $f$ in different regions of space. 
We will then use Feynman-Kac formula to argue that, 
for points sufficiently far away from the region $\{|x_1|\geq L\text{ or }|x_2|\geq L\}$, 
tail bounds for Brownian motion implies that $P^{(R_0)}_{s,t}f$ is very small for $L^2 \gg t$ because $f$ has compact support $S(f)\subset[-L/2,L/2]^2$. 
For points close to the origin, smallness will come from the $Z-Z^L$ term in $F$.  

For $k\in\N$, let $\tilde\cA_k$ denote the annulus:
\begin{equation} \label{e:Ak-def}
\tilde \cA_k 
:=
\Big\{ (x_1,x_2)\in\R^2 : \frac{L^k}{2^k}-1\leq |x_1|,|x_2|< 
\frac{L^{k+1}}{2^{k+1}}-1\Big\}
.
\end{equation}
As the $\tilde\cA_k$ partition $\R^2$ and using the multiplicative estimate of Proposition~\ref{prop:besov-mult} (specifically,~\eqref{eq_besovmult_cubes}) with parameters $\alpha,\alpha'$,  
the quantity to estimate becomes:
\begin{align}
\sum_{k=0}^\infty\int_0^t \big(F^{(R_0)}_s,{\bf 1}_{\tilde \cA_k}\Pg^{(R_0)}_{s,t}f\big)\, ds
&\lesssim 
\sum_{k=0}^\infty L^{(k+1)d}\int_0^t \|F^{(R_0)}_s \|_{-\alpha,\,\cA_k}\big\| \Pg^{(R_0)}_{s,t}f\big\|_{\alpha',\, \cA_k}\, ds
,
\label{eq_splitting_on_Ak_R_0}
\end{align}
where we set:
\begin{equation}
\cA_k 
:= 
\tilde\cA_k + B_3(0)
,\quad 
k\geq 0
.
\label{e:Ak-def2}
\end{equation}
Then:
\begin{equation}
(E_t,f)
\lesssim
\sum_{k=0}^\infty L^{(k+1)d}\int_0^t \|F_s \|_{-\alpha/2,\,\cA_k}\big\| \Pg_{s,t}f\big\|_{\alpha,\, \cA_k}\, ds
.
\label{eq_splitting_on_Ak}
\end{equation}
To estimate each term above,  
we will need the bounds of Lemma~\ref{lemm_general_bounds}. 
Equation~\eqref{eq_bounds_YZ} takes care of the estimate of the irregular terms, i.e., of $F_\cdot$, for $k\geq 1$. 
It remains to estimate $\Pg_{s,t}f$ on any $\cA_k$ ($k\geq 0$) and to get a bound on $F_s$ that vanishes with $L$ close to the origin. 
This is the content of the next two propositions, 
for which we recall that $\sigma>0$ is the parameter in the weight $\rho(x)=(1+|x|^2)^{-\sigma/2}$ appearing in Lemma~\ref{lemm_general_bounds}.

\begin{proposition}[decay for $k\geq 1$]\label{prop_FK_term}
Let $\alpha,\sigma>0$ be small enough. 

There is $\theta>0$ that can be made arbitrarily small if $\alpha,\sigma$ are small and $C_0=C_0(\theta)>0$ such that, if
\begin{equation}
  L\geq C_0(t^{C_0}+ R_f+1), \qquad t>0,
\end{equation}
then, on the event of Lemma~\ref{lemm_general_bounds}: 
\begin{align}
\sup_{s\leq t}\|\Pg_{s,t}f\|_{\alpha,\, \cA_k}
&\lesssim 
{\bf 1}_{k=0}\|f\|_{\alpha} \, e^{L^\theta}
+{\bf 1}_{k\geq 1}\|f\|_{\alpha}\, e^{-cL^{2k}/t}
.
\end{align}
\end{proposition}

\begin{proposition}[decay for $k=0$]\label{prop_gaussian_difference_sec4}
For each $\alpha,\sigma>0$, there is $c>0$ such that, for each $n\in\{1,2,3\}$, $L\geq 3$ and $t>0$:
\begin{equation}
\E\Big[\big\|\wick{Z_t^n}-\wick{(Z^L_t)^n}\|_{-\alpha/2,\, \cA_0+B_2(0)}^2\Big]
\leq
\frac{1}{c}\, (t^{-n\alpha}\wedge 1)\, \big(1+\|\varphi_0\|^{2n}_{-\alpha, \rho}\big)e^{-cL^2/t}
,
\end{equation}
where we recall that $\E$ denotes the expectation of the probability space on which the $Z^L$ are defined, i.e., expectation with respect to the noise in the SPDE~\eqref{eq_heat_periodised}. 
\end{proposition}

Proposition~\ref{prop_gaussian_difference_sec4} is immediate from Proposition~\ref{prop_gaussian_difference} (stated with a norm with exponent $-n\alpha$, but note that the claim is valid for arbitrary $\alpha$) and we complete below the proof of Proposition~\ref{prop_finite_PS}
assuming Lemma~\ref{lemm_general_bounds} and Proposition~\ref{prop_FK_term}. 
It will then remain to establish these two results. 
This is the content of the next two subsections.

\begin{proof}[Proof of Proposition~\ref{prop_finite_PS}]
To prove Proposition~\ref{prop_finite_PS}, 
we start from the splitting~\eqref{eq_splitting_on_Ak} on the different annuli $\cA_k$ ($k\in\N$), recall \eqref{e:Ak-def2}.  

Consider first the $k=0$ term. 
By Proposition~\ref{prop_gaussian_difference_sec4}, 
\begin{equation}
\bbP\bigg(\int_0^t\big\|\wick{Z_s^n}-\wick{(Z^L_s)^n}\big\|_{-\alpha/2,\, \cA_0+B_2(0)}^2 \, ds
\ \geq \ \frac{t}{c} \,
\big(1+\|\varphi_0\|^{2n}_{-\alpha, \rho}\big)e^{-cL^2/2t}\bigg)
\leq 
e^{-cL^2/2t}
.
\end{equation}
Define: 
\begin{align}
\Omega_{t,L}(\varphi_0) 
&:=
\bigcap_{n=1}^3\Big\{ \int_0^t\big\|\wick{Z_s^n}-\wick{(Z^L_s)^n}\big\|_{-\alpha/2,\, \cA_0+B_2(0)}^2 \, ds
\ \leq \ 
\frac{t}{c}\, \big(1+\|\varphi_0\|^{2n}_{-\alpha, \rho}\big)e^{-cL^2/2t}\Big\}
\nnb
&\qquad
\cap \Omega'_{t,L}(\varphi_0,\alpha,\alpha') 
\cap \Omega'_{t,L}(\varphi_0,\alpha/2,\alpha') 
,
\label{eq_def_event_Omega_t}
\end{align}
where $\Omega'_{t,L}(\varphi_0,\beta,\alpha')$ is the event of Lemma~\ref{lemm_general_bounds} ($\beta\in\{\alpha/2,\alpha\}$).  
Then, on this event $\Omega_{t,L}(\varphi_0)$, 
recalling the expression~\eqref{eq_diff_Y_YL} for $F_\cdot$, 
the multiplicative inequality for $\|\cdot\|_{-\alpha,C}$ (Proposition~\ref{prop:besov-mult}), 
the bounds of Lemma~\ref{lemm_general_bounds},
and that the weight satisfies $\rho \gtrsim L^{-\sigma}$ on $\cA_0$, we have for each $s\in(0,t]$:
\begin{align}
\|F_s\|_{-\alpha/2,\, \cA_0}
&\lesssim 
\max_{1\leq n\leq 3} \, \|v_s\|^{3-n}_{\alpha,\, \cA_0 + B_2(0)}\, \|\wick{Z^n_s}-\wick{(Z^L_s)^n}\|_{-\alpha/2,\, \cA_0 +B_2(0)}
\nnb
&\lesssim 
L^{2\sigma\eta} (\log L)^{2\kappa}(1+t)^{2\kappa} \max_{1\leq n\leq 3} \|\wick{Z^n_s}-\wick{(Z^L_s)^n}\|_{-\alpha/2,\, \cA_0 +B_2(0)}
.
\end{align}
The bounds on $\Pg_{s,t}f$ from Proposition~\ref{prop_FK_term} then imply that, on $\Omega_{t,L}(\varphi_0)$,
for $L \geq C_0 (t^{C_0}+R_f+1)$ and some $\delta>0$:
\begin{align}
&L^d\int_0^t \|F_s\|_{-\alpha/2,\, \cA_0}\, \|\Pg_{s,t}f\|_{\alpha,\, \cA_0}\, ds
\nnb
&\lesssim 
\sup_{s\leq t}\|\Pg_{s,t}f\|_{\alpha,\, \cA_0} \,
L^{d+2\sigma\eta}(\log L)^{2\kappa} (1+t)^{2\kappa}
\, \int_0^t \max_{1\leq n\leq 3}\|\wick{(Z^L_t)^n}-\wick{Z_t^n}\|_{-\alpha/2,\, \cA_0+B_2(0)} \, ds
\nnb
&\lesssim 
(1+\|\varphi_0\|^{6}_{-\alpha,\rho})^{1/2}\, \|f\|_{\alpha}\,L^{d+2\sigma\eta}\, (\log L)^{2\kappa}\, (1+t)^{3\kappa +1} \, e^{L^\theta}\, e^{-cL^2/2t}
\nnb
&\lesssim 
(1+\|\varphi_0\|^{3}_{-\alpha,\rho})\, \|f\|_{\alpha}\,  e^{-\delta L^2/4t}
,
\end{align}
where the $\lesssim$ hide constants independent of $t,L,\varphi_0,f$.  
Moreover, by definition~\eqref{eq_def_event_Omega_t} of $\Omega_{t,L}(\varphi_0)$:
\begin{equation}
\bbP\big((\varphi_s)_{s\leq t}\notin\Omega_{t,L}(\varphi_0)\big)
\leq 
2\max_{\beta\in\{\alpha/2,\alpha\}}\bbP\big((\varphi_s)_{s\leq t}\notin\Omega'_{t}(\varphi_0,\beta,\alpha')\big)
+3 e^{-cL^2/2t}
.
\end{equation}
and the probabilities on the right-hand side are bounded by~\eqref{eq_bound_prob_Omega'}. 
Our choice of $L$ makes $3e^{-cL^2/2t}$ smaller than these probabilities, 
and~\eqref{eq_bound_prob_Omega'} holds for $\Omega_{t,L}(\varphi_0)$ as well.

Consider now the $k\geq 1$ terms in~\eqref{eq_splitting_on_Ak}. 
On $\Omega_{t,L}(\varphi_0)$, 
Lemma~\ref{lemm_general_bounds} similarly implies: 
\begin{align}
\sup_{s\leq t}\, (s^{n\alpha}\wedge 1)\, \|F_s\|_{-\alpha/2,\, \cA_k}
&\lesssim 
\max_{1\leq n\leq 3}\sup_{s\leq t}\, (s^{n\alpha}\wedge 1)\, \|v_s\|^{3-n}_{\alpha,\, \cA_k + B_2(0)}\|\wick{Z^n_s}-\wick{(Z^L_s)^n}\|_{-\alpha/2,\, \cA_k +B_2(0)}
\nnb
&\lesssim 
L^{3(k+1)\sigma\eta}(1+t)^{3\kappa}(\log L)^{3\kappa}
.
\end{align}
Again using Proposition~\ref{prop_FK_term} to bound $\Pg_{s,t}f$ on $\Omega_{t,L}(\varphi_0)$, 
we therefore obtain:
\begin{align}
\int_0^t \|F_s\|_{-\alpha/2,\, \cA_k} \|\Pg_{s,t} f\|_{\alpha,\, \cA_k}\, ds 
&=
\|f\|_\alpha  e^{-c'L^{2k}/t} \int_0^t s^{-3\alpha}\big(s^{3\alpha}\|F_s\|_{-\alpha/2,\cA_k}\big) \, ds 
\nnb
&\leq 
\|f\|_\alpha (1+t)^{3\kappa+1} L^{3(k+1)\sigma\eta}  (\log L)^{3\kappa}e^{-c'L^{2k}/t}
  .
\end{align}
Note that this bound does not depend on the initial condition because all dependence on $\varphi_0$ is hidden in the definition of the set $\Omega_{t,L}(\varphi_0)$. 
The last bound in particular implies:
\begin{align}
\sum_{k\geq 1}L^{(k+1)d}\int_0^t \|F_s\|_{-\alpha/2,\, \cA_k} \|\Pg_{s,t} f\|_{\alpha,\, \cA_k}\, ds 
\leq 
\|f\|_{\alpha} e^{-\delta' L^2/t}
.
\end{align}
This concludes the proof: 
for $L \geq C_0 (t^{C_0}+R_f+1)$ as in Proposition~\ref{prop_FK_term},
on the event $\Omega_{t,L}(\varphi_0)$ defined in~\eqref{eq_def_event_Omega_t}, 
which satisfies~\eqref{eq_bound_prob_Omega} by construction, 
it holds that:
\begin{align}
|(E_t,f)|
\lesssim
\|f\|_{\alpha}\, \big(1+\|\varphi_0\|_{-\alpha,\rho}^3\big)\, \exp\big[-cL^2/t\big]
,
\end{align}
with $\lesssim$ independent of $f,\varphi_0,t,L$. 
\end{proof}
\subsection{Proof of Proposition~\ref{prop_FK_term}}\label{sec_proof_finite_prop_speed}
Let $t>0$ and $L>0$ to be chosen large enough below as a function of $t$ and the test function $f$. 
Fix $\alpha,\alpha',\sigma>0$ that will be taken small enough as needed along the proof, and such that all bounds stated in Lemma~\ref{lemm_general_bounds} apply.

To estimate $\Pg_{s,t}f$ we would like to use Feynman-Kac formula:
\begin{equation}
\Pg_{s,t}f(x)
=
 \Eg\bigg[ 
 f(x+\sqrt{2}B_{t-s})\exp\Big[-\lambda\int_0^{t-s} U_{s+u}(x+\sqrt{2}B_{u})\, du\Big]\bigg]
 ,\qquad 
 x\in\R^2
 .
\end{equation}
The function $U$ has regularity $0_-$ so this formula does not directly make sense and we first consider a regularised version. 
Let $R_0>0$ and define:
\begin{equation}
U^{(R_0)}_s
:=
U_s\ast\Psi_{R_0},
\qquad 
s\in(0,t] 
,
\end{equation}
where $\Psi$ is the mollifier from Section~\ref{sec:norms-def}. 
Introduce the corresponding regularised semigroup
$(\Pg^{(R_0)}_{s,t})_{s\leq t}$ solving, 
for each test function $g$:
\begin{equation}
\partial_s \Pg^{(R_0)}_{s,t}g
=
(-\Delta+\lambda U^{(R_0)}_{s})\Pg^{(R_0)}_{s,t}g,
\qquad 
s \in (0,t]
.
\label{eq_def_Pg_R0}
\end{equation}
We then have a Feynman--Kac formula for $\Pg^{(R_0)}_{s,t}f$ as stated next. 
\begin{lemma}
Let $s\in(0,t]$ and $x\in\R^2$.  
Then, for $\sigma$ small enough depending only on $\alpha$:
\begin{equation}\label{eq:Feynman-Kac}
\Pg^{(R_0)}_{s,t}f(x)
=
 \Eg\bigg[ 
 f(x+\sqrt{2}B_{t-s})\exp\Big[-\lambda\int_0^{t-s} U^{(R_0)}_{s+u}(x+\sqrt{2}B_{u})\, du\Big]\bigg]
 ,\qquad 
 x\in\R^2
 .
\end{equation}
\end{lemma} 
\begin{proof}
  We check the assumptions of (the proof of)~\cite[Theorem 7.6]{KaratzasShreeve}.
  We note that while this theorem is stated under the assumption that the potential $U^{(R_0)}$ (which corresponds to $-k$ in the notation of ~\cite[Theorem 7.6]{KaratzasShreeve}) is bounded, the proof goes through without change under the assumption that its growth at infinity is sub-quadratic locally uniformly in the time variable, ensuring that the expectation on the right hand side of \eqref{eq:Feynman-Kac} is finite. The bounds in
 Lemma~\ref{lemm_general_bounds} and the multiplicative inequality (Proposition~\ref{prop:besov-mult}) provide this necessary control. Indeed, they imply that $|U^{(R_0)}_u(x)|$ has at most polynomial growth in $x$,
\begin{align}
\sup_{u\in[s,t]}\sup_{|x|\leq A}|U^{(R_0)}_u(x)|
&\leq 
R_0^{-\alpha}\sup_{|x|\leq A} \rho^{-2\eta}(x) \sup_{u\in(0,t]}\|U_s\|_{-\alpha,\rho^{2\eta}}
\nnb
&\lesssim 
R_0^{-\alpha} u^{-2\alpha} A^{2\eta\sigma} (\log L)^{2\kappa}\, (1+t)^{2\kappa} 
,
\end{align}
and the power is arbitrarily small if $\sigma$ is small.
\end{proof} 
Note that the multiplicative inequality \eqref{e:besov-mult} implies $\lim_{R_0\to 0}\|U_s^{(R_0)}-U_s\|_{-\alpha,\, \cA_k} =0$ for each $k\geq 1$ and $s>0$. 
As a result, $\Pg^{(R_0)}_{s,t}f$ has a well-defined limit:
\begin{equation}
\lim_{R_0\to 0} \|\Pg^{(R_0)}_{s,t} f -\Pg_{s,t}f\|_{\alpha,\, \cA_k}
=
0
,\qquad 
k\geq 1
.
\label{eq_convergence_Pg}
\end{equation}
It is therefore enough to prove Proposition~\ref{prop_FK_term} for $\Pg^{(R_0)}_{s,t} f$ with bounds uniform in the regularisation parameter $R_0$. 

\begin{notation}
In the rest of this section we only work with the regularised semigroup $\Pg_{s,t}^{(R_0)}$. 
To lighten the notation we drop the dependence on $R_0$ in $U^{(R_0)},\Pg_{s,t}^{(R_0)}$ when no confusion may arise.
\end{notation}
As $U^{(R_0)}$ is only of regularity $0_-$ when $R_0\to0$, 
we cannot hope to get regularisation-independent $\bbL^\infty$ bounds on $\Pg^{(R_0)}_{s,t}f$ (much less $C^\alpha$) by a direct pathwise estimate.
Instead, following~\cite{KoenigPerkowskivanZuiljen_PAM}, 
we use the Girsanov formula to turn $U^{(R_0)}$ into a more regular function (using the so-called partial Girsanov transform of~\cite{MR3592748}). 
Introduce $\psi=(\psi_u)_{u\in[s,t]}$ such that (recall $A=-\Delta+1$):
\begin{equation}
(\partial_u -A)\psi_u
= 
-\lambda U_{u}
,\qquad
\psi_t
=
0
,
\label{eq_def_psi}
\end{equation}
where above and in the following $U$ is implicitly regularised. 
Then $u\mapsto\psi_{t-u}$ solves the forward equation with source term $\lambda U_{t-u}$ and initial condition $0$:
\begin{equation}
\psi_{t-u}(x)
=
\lambda\int_0^u e^{-(u-v)A}U_{t-v}(x)\, dv
\quad\text{or} \quad
\psi_{u}(x)
=
\lambda\int_0^{t-u} e^{-(t-u-v)A}U_{t-u-v}(x)\, dv
.
\label{eq_formula_psi}
\end{equation}
Thus $\psi$ is of regularity $2_-$ in the limit of vanishing regularisation.
Write for short $\tilde B^x_\cdot:=x+\sqrt{2}B_\cdot$.  
The Ito formula applied to $\psi_{s}(\tilde B^x_s)$ gives:
\begin{align}
\psi_t(\tilde B^x_{t-s})
&=
\psi_s(x) + \int_0^{t-s} (\partial_u+\Delta)\psi_{u+s}(\tilde B^x_{u})\, du + \int_0^{t-s} \nabla \psi_{u+s}(\tilde B^x_{u})\, d B_{u}
\nonumber\\
&=
\psi_s(x) -\lambda\int_0^{t-s} U_{u+s}(\tilde B^x_{u})\, du + \int_0^{t-s} \psi_{u+s}(\tilde B^x_{u})\, du + \int_0^{t-s} \nabla \psi_{u+s}(\tilde B^x_{u})\, d B_{u}
.
\end{align}
Thus:
\begin{align}
\Pg_{s,t}f(x)
&=
\Eg\bigg[ f(\tilde B^x_{t-s}) \exp\Big[\psi_t(\tilde B^x_{t-s}) - \psi_s(x) - \int_0^{t-s} \psi_{s+u}(\tilde B^x_{u})\, du 
\nnb
&\hspace{6cm}- \int_0^{t-s} \nabla\psi_{u+s}(\tilde B^x_{u})\, dB_u\Big]\bigg]
.
\end{align}
The Girsanov formula then implies that $\sqrt{2}W_u := \tilde B^x_u-x+\int_0^u\nabla \psi_{s+v}(\tilde B^x_v)\, dv$ ($u\leq t-s$) is $\sqrt{2}$ times a Brownian motion starting at $0$ under the measure:
\begin{equation}
\mathrm{d}\Qg_{t-s} ((x_u)_{u\leq t-s})
=
\exp\bigg[- \int_0^{t-s} (\nabla\psi_{u+s}(x_u),\, dx_u)
-\int_0^{t-s} \big|\nabla \psi_{u+s}(x_u)\big|^2\, du
\bigg]
\, \mathrm{d}\Pg((x_u)_{u\leq t-s})
.
\label{eq_def_law_aux_diffusion}
\end{equation}
It follows that the coordinate process $(X^x_u)_{u\leq t-s}$ under $\Qg_{t-s}$ solves:
\begin{equation} \label{e:X-SDE}
  \mathrm{d}X^x_u = -\nabla \psi_{u+s}(X^x_u)\, \mathrm{d}u + \sqrt{2}\, \mathrm{d} W_u, \qquad X^x_0=x
\end{equation}
and:
\begin{align}
\Pg_{s,t}f(x)
=
  \Eg_{\Qg_{t-s}}\bigg[ f(X^x_{t-s}) e^{v_{s,t}(X^x)} 
\bigg]
.
\label{eq_E_t_interm}
\end{align}
where we define the shorthand:
\begin{equation}
  v_{s,t}(X^x_\cdot) 
:= 
\psi_t(X^x_{t-s}) - \psi_s(x) - \int_0^{t-s} \psi_{s+u}(X^x_u)\, du + \int_0^{t-s} \big|\nabla \psi_{u+s}(X^x_u)\big|^2\, du 
.
\label{eq_def_vst}
\end{equation}
The functions $\psi$, $\nabla\psi$ satisfy the following estimates.
\begin{lemma}\label{lemm_bound_psi_gradpsi}
Let $K>0$. 
Recall that the weight $\rho$ satisfies $\rho(x)\lesssim K^\sigma$ if $|x|\leq K$. 
Recall also that $\kappa>0$ is the constant appearing in Lemma~\ref{lemm_general_bounds}. 
Then, on the event of Lemma~\ref{lemm_general_bounds},
\begin{align}
\sup_{s\leq t}\|\psi_s\|_{\alpha, \, [-K,K]^2}
&\lesssim 
(\log L)^{2\kappa}(1+t)^{2\kappa} K^{2\sigma\eta} 
,\label{eq_bound_psi}\\
\sup_{s\leq t} \|\nabla\psi_s\|_{\alpha,\, [-K,K]^2}
&\lesssim 
(\log L)^{2\kappa}(1+t)^{2\kappa} K^{2\sigma\eta}
.
\label{eq_bound_grad_psi}
\end{align}
\end{lemma}

\begin{proof}
In this proof we write the regularisation explicitly to avoid confusion, that is we differentiate between $U^{(R_0)}$ (appearing in the definition of $\psi=\psi^{(R_0)}$) and its limit $U$. 
Recall that $\eta=\frac{1+\alpha}{1-3\alpha}$ is the parameter appearing in Lemma~\ref{lemm_general_bounds}. 
Recall that, by definition of $U$ in~\eqref{eq_diff_Y_YL} and of the event $\Omega'_{t,L}(\varphi_0,\alpha,\alpha')$ of Lemma~\ref{lemm_general_bounds}, 
the multiplicative inequality (Proposition~\ref{prop:besov-mult}) implies that, 
on this event:
\begin{equation}
\sup_{s\leq t} (s^{2\alpha}\wedge 1) \|U_s\|_{-\alpha,\rho^{2\eta}}
\lesssim 
(\log L)^{2\kappa}(1+t)^{2\kappa}
. 
\end{equation}
Recall also the elementary bound $\|U^{(R_0)}_s\|_{-\alpha,\rho^{2\eta}}\leq \|U_s\|_{-\alpha,\rho^{2\eta}}$. 
Let $K>0$. 
Using the smoothing effect of the heat kernel (see Proposition~\ref{prop:Besov-heat}) and the bounds on the field of Lemma~\ref{lemm_general_bounds}, 
\begin{align}
\|\psi_s\|_{\alpha, \, [-K,K]^2}
&\leq 
\int_0^{t-s} \|e^{-(t-s-u)A} U^{(R_0)}_u\|_{\alpha, \, [-K,K]^2}\ du 
\lesssim 
\int_0^{t-s} K^{2\sigma\eta }\| e^{-uA} U^{(R_0)}_{u}\|_{\alpha, \rho^{2\eta}}\ du 
\nnb
&\lesssim 
K^{2\sigma \eta}\int_0^{t-s} u^{-2\alpha/2} u^{-2\alpha} e^{-u}u^{2\alpha} \|U_u\|_{-\alpha, \rho^{2\eta}}\ du 
\nonumber\\
&\lesssim 
K^{2\sigma\eta}(\log L)^{2\kappa} (1+t)^{2\kappa} 
,\label{eq_bound_psi-pf}
\\
 \|\nabla\psi_s\|_{\alpha, \, [-K,K]^2}
&\leq 
\int_0^{t-s} \|\nabla e^{-uA} U^{(R_0)}_u\|_{\alpha,\,  [-K,K]^2}\ du 
\lesssim 
K^{2\sigma\eta} \int_0^{t-s} \|e^{-uA} U^{(R_0)}_u\|_{1+\alpha, \rho^{2\eta}}\ du 
\nonumber\\
& \lesssim
K^{2\sigma\eta}\int_0^{t-s} u^{-(1+2\alpha)/2}u^{-2\alpha} e^{-u} u^{2\alpha}\|U_u\|_{-\alpha,\rho^{2\eta}}\ du
\nnb
&\lesssim 
K^{2\sigma\eta} (\log L)^{2\kappa}(1+t)^{2\kappa}
.
\label{eq_bound_grad_psi-pf}
\end{align}
\end{proof}

We now estimate $\|\Pg_{s,t}f\|_{\alpha,\, \cA_k}$, starting with $\|\Pg_{s,t}f\|_{\mathbb L^{\infty}(\cA_k)}$. 
To simplify notations, it will be convenient to introduce
a constant $c_0>0$ and, for $\ell\in\N$, the quantity:
\begin{equation} \label{e:CtL-def}
  C_{t,L}(\ell)
  := c_0L^{4(\ell+1)\sigma\eta}(\log L)^{4\kappa}(1+t)^{4\kappa+1}.
\end{equation}
The constant $c_0>0$ is defined such that, for each $\ell\in\N$,
\begin{align}
&t\sup_{u\leq t}\max\Big\{\|\nabla \psi_u\|^2_{\alpha,Q_\ell},\|\psi_u\|_{\alpha,Q_\ell} \Big\}
\leq 
C_{t,L}(\ell)
,
\qquad
Q_\ell 
:=
[-\frac12 L^{\ell+1}-3,\, \frac12L^{\ell+1}+3]^2
.
\label{eq_def_CtL}
\end{align}
which is possible by the bounds~\eqref{eq_bound_psi}--\eqref{eq_bound_grad_psi}.
We also assume $C_{t,L}(\ell) \geq 1$.
To use the bounds on $\psi,\nabla\psi$, 
we need to know where in $\R^2$ the process $X^x_\cdot$ takes values. 
To do so, couple all $(X^x_\cdot)_{x\in\R^2}$ with the same noise and write $\Eg_{\Qg_{t-s}}$ for the corresponding expectation. 
For $\ell\in\N$, introduce the event: 
\begin{equation}
I^x_{\ell} (s,t)
:=
\Big\{ \sup_{u\leq t-s}\big|X^x_{u}\big|\in \big[\frac12 L^\ell-4, \frac12 L^{\ell+1} -4\big)\Big\}
.
\label{eq_def_I_ell}
\end{equation}
The $-4$ ensures that $\{x\in \cA_k: X^x_\cdot \in I^x_\ell(s,t)\}=\emptyset$ for $\ell<k$ (see the definitions~\eqref{e:Ak-def}--\eqref{e:Ak-def2} of $\cA_k$). 
For each $k\geq 0$, using  
\eqref{eq_def_CtL} and that $I^x_{\ell}(s,t)$ is incompatible with $x\in \cA_k$ if $\ell<k$, 
we decompose:
\begin{align}
\|\Pg_{s,t}f\|_{\bbL^\infty(\cA_k)}
&\leq 
\|f\|_\infty \sum_{\ell\geq k}\sup_{x\in \cA_k}\Eg_{\Qg_{t-s}}\Big[ {\bf 1}_{I^x_\ell(s,t)}{\bf 1}_{X^x_{t-s}\in S(f)}e^{v_{s,t}(X^x_\cdot)}\Big]
\nnb
&\leq
\|f\|_\infty
\sum_{\ell\geq k}e^{4C_{t,L}(\ell)}\sup_{x\in \cA_k}\Qg_{t-s}\big(I^x_\ell(s,t), X^x_{t-s}\in S(f)\big)
.
\label{eq_chi_kPstf_last}
\end{align}
It remains to estimate this last probability. 
This is the content of the next lemma.

\begin{lemma}[Bounds on the diffusion $X$]\label{lemm_LD_X}
Let $k\in\N$, $\ell\geq k$, $t>0$, $s\leq t$ and recall the definition~\eqref{eq_def_I_ell} of the event $I^x_\ell(s,t)$. 
Let $L\geq 32$ and let $R_f$ denote the radius of the smallest ball containing $S(f)$. 
If $\ell\geq 1$, assume that:
\begin{equation}
\frac{L^\ell(1-L^{-1})}{2} - C_{t,L}(\ell)-R_f-4
\geq 
\frac14 L^{\ell}
.
\label{eq_assumption_L_LemmaLD}
\end{equation}
Then, on the event of Lemma~\ref{lemm_general_bounds}:
\begin{align}
\sup_{s\leq t}\sup_{x\in \cA_k}\Qg_{t-s}\big(I^x_\ell(s,t), X^x_{t-s}\in S(f)\big)
&\leq 
2\exp\Big[-\frac{L^{2\ell}}{16t}\Big]
\nnb
\sup_{s\leq t}\sup_{x\in \cA_k}\Qg_{t-s}\big(I^x_\ell(s,t)\big)
&\leq 
{\bf 1}_{\ell\in\{k,k+1\}}+{\bf 1}_{\ell>k+1}2\exp\Big[-\frac{L^{2\ell}}{16t}\Big]
.
\end{align}
\end{lemma}
\begin{proof}
If $k=\ell=0$ there is nothing to prove. 
Assume instead $k=0$, $\ell\geq 1$ or $\ell\geq k\geq 1$.  
Then Ito's formula gives:
\begin{equation}
X^x_u
=
x- \int_0^u \nabla\psi_{s+v}(X^x_v)\, dv +\sqrt{2}W_u
,
\end{equation}
with $W_\cdot$ a standard Brownian motion. 
The uniform bound~\eqref{eq_bound_grad_psi} on $|\nabla\psi_u|$ implies (recall that $C_{t,L}(\ell)$ is defined in~\eqref{eq_def_CtL}):
\begin{equation}
{\bf 1}_{I^x_\ell(s,t)}\Big|\int_0^u \nabla\psi_{s+v}(X^x_v)\, dv\Big|
\leq 
C_{t,L}(\ell)
.
\end{equation}
As $\ell\geq 1$, on the event $I_\ell(s,t)\cap\{X^x_{t-s}\in S(f)\}$ it must be that $X^x_{t-s}$ is a distance at least $L^{\ell}/2-4-R_f$ from its position at an intermediate time.   
Since $W_\cdot$ does not depend on $x$ and~\eqref{eq_assumption_L_LemmaLD} holds, we find:
\begin{align}
&\sup_{s\leq t}\sup_{x\in \R^2}\Qg_{t-s}\big(I^x_\ell(s,t), X^x_{t-s}\in S(f)\big)
\nnb
&\qquad \leq 
\sup_{s\leq t}\Qg_{t-s}\Big(\sup_{u\leq t-s}\big|\sqrt{2}W_u\big|
\geq L^\ell/2-4-R_f- C_{t,L}(\ell)\Big)
\nnb
&\qquad \leq 
2\exp\Big[ -\frac{\big(L^\ell-4-R_f- C_{t,\ell}(L)\big)^2}{t}\Big]
\leq 
2e^{-L^2/(16t)}
,
\end{align}
where the last line is Lemma~\ref{lemm_bound_psi_gradpsi}. 
Similarly, on the event $I^x_{s,t}(\ell)$ with $x\in \cA_k$ it must be that:
\begin{equation}
\sup_{u\leq t-s}|X^x_u-x|
\geq 
L^{\ell}/2-4-\sqrt{2}(L^{k+1}/2+1)
,
\end{equation}
where the $\sqrt{2}$ comes from the different norms used to define $I^x_{\ell}(s,t)$, $\cA_k$.  
If $\ell\in\{k,k+1\}$ we have nothing to prove, otherwise the above is positive as we assume $L\geq  32$, 
and Assumption~\eqref{eq_assumption_L_LemmaLD} on $L$ gives:
\begin{equation}
L^{\ell}/2-4-\sqrt{2}\big(L^{k+1}/2 - 1)- C_{t,L}(\ell)
\geq 
\frac{L^{\ell}(1-L^{-1})}{2}- C_{t,L}(\ell)
-
4
\geq 
L^{\ell}/4
.
\end{equation}
Thus, when $\ell>k+1$:
\begin{equation}
\sup_{s\leq t}\sup_{x\in \cA_k}\Qg_{t-s}\big(I^x_\ell(s,t)\big)
\leq 
2e^{-L^{2\ell}/(16t)}
. \qedhere
\end{equation}
\end{proof}

To apply the lemma, we will impose its assumption in the rest of this section.
More precisely, we assume the following stronger constraint on $L$: 
for each $\ell\geq 1$, 
\begin{equation}
\frac{L^\ell(1-L^{-1})}{2} - C_{t,L}(\ell)^{\frac{1}{1-\alpha}}-R_f-4
\geq 
\frac14 L^\ell 
.
\label{eq_assumption_L_proof_decay_kgeq1}
\end{equation}
Since $C_{t,L}(\ell)$ is given by \eqref{e:CtL-def},
this is true if $\sigma,\alpha$ are small enough and $L$ satisfies the main assumption:
\begin{equation}
L
\geq 
C_0 (t^{C_0}+R_f+1)
,
\label{eq_def_L_proof_finite_propspeed}
\end{equation}
for a constant $C_0>0$ large enough uniformly on $\alpha,\sigma$ (taking an exponent strictly larger than $8\kappa+2$ e.g. works if $1-\alpha-8\sigma\eta\geq 1/2$). 
Moreover, for any $\theta>0$, 
if $\alpha,\sigma$ are small enough and $C_0=C_0(\theta)$ large enough, one can have:
\begin{equation}
4C_{t,L}(0)
\leq 
L^\theta
.
\label{eq_upperbound_C_tL0}
\end{equation}

To prove Proposition~\ref{prop_FK_term}, we need to bound $\|\Pg_{s,t}f\|_{\bbL^\infty(\cA_k)}$
and the H\"older seminorm $[\Pg_{s,t}]_{\alpha,\cA_k}$.
We first apply Lemma~\ref{lemm_LD_X} to conclude the bound of $\|\Pg_{s,t}f\|_{\bbL^\infty(\cA_k)}$.

\begin{proof}[Proof of $\bbL^\infty$ bound]
Equation~\eqref{eq_chi_kPstf_last} and Lemma~\ref{lemm_LD_X} give:
\begin{align}
\|\Pg_{s,t}f\|_{\bbL^\infty(\cA_k)}
\lesssim 
\|f\|_\infty\sum_{\ell\geq k} \exp\Big[4C_{t,L}(\ell) -\frac{L^{2\ell}}{16}\Big]
.
\end{align}
Recall that $4C_{t,L}(0)\leq L^\theta$ from~\eqref{eq_upperbound_C_tL0}. 
For $\ell\geq 1$,  
Assumption~\eqref{eq_assumption_L_LemmaLD} on $L$ implies:
\begin{equation}
\forall \ell\geq 1,\qquad 
C_{t,L}(\ell)
\leq 
\frac{L^{\ell}}{4}
\quad 
\Rightarrow\quad 
4C_{t,L}(\ell) - \frac{L^{2\ell}}{16t}
\leq 
- \frac{L^{2\ell}}{32t}
,
\label{eq_log_leq_power}
\end{equation}
where the implication relies on the lower bound~\eqref{eq_def_L_proof_finite_propspeed} on $L$ with a large enough $C_0$. 
In particular we can take $C_0$ to ensure $L\geq 32t$, so that $L^\ell-L^{2\ell}/(16t)\leq -L^{2\ell}/(32t)$. 
This concludes the bound on $\|\Pg_{s,t}f\|_{\bbL^\infty(\cA_k)}$:
\begin{align}
\|\Pg_{s,t}f\|_{\bbL^\infty(\cA_k)}
&\lesssim
{\bf 1}_{k=0}\|f\|_{\bbL^\infty(\R^2)} e^{4 C_{t,L}(0)} + {\bf 1}_{k\geq 1}\|f\|_{\bbL^\infty(\R^2)} e^{- L^{2k}/(32t)}
\nnb
&\lesssim
{\bf 1}_{k=0}\|f\|_{\bbL^\infty(\R^2)}
e^{L^\theta} + {\bf 1}_{k\geq 1}\|f\|_{\bbL^\infty(\R^2)}e^{-L^{2k}/t}
.
\end{align}
\end{proof}

We now estimate the H\"older seminorm $[\Pg_{s,t}f]_{\alpha,\, \cA_k}$.
This will require a bound on the expected variation of $X^{x}_\cdot$. 
This is not standard since $X^x_\cdot$ does not have Lipschitz drift. 
Using results of~\cite{MR2593276}, we prove the following estimate in Appendix~\ref{app_gradientX}.

\begin{lemma}\label{lemm_gradientX}
For each $\alpha\in(0,1)$, 
there is a constant $C(\alpha)>0$ such that, for each $\ell,\ell'\in\N$:
\begin{align}
&\sup_{s\leq t}\sup_{\substack{x,y\in\R^2 \\ x\neq y}}\frac{1}{|x-y|}
\sup_{u\leq t-s}\Eg_{\Qg_{t-s}}\Big[{\bf 1}_{J^{x,y}_{s,t}(\ell,\ell')}|X^x_{u}-X^y_u|^2\Big]^{1/2}
\nnb
&\hspace{5cm}\lesssim
\exp\Big[C(\alpha)\max\big\{C_{t,L}(\ell)^{\frac{1}{1-\alpha}},C_{t,L}(\ell')^{\frac{1}{1-\alpha}}\big\}\Big]
, 
\end{align}
where $J^{x,y}_{\ell,\ell'}(s,t) := I^x_\ell(s,t) \cap I^y_{\ell'}(s,t)$
and  $I^x_{\ell}(s,t)$ was defined in \eqref{eq_def_I_ell}.
\end{lemma}

\begin{proof}[Proof of $C^\alpha$ bound]
Let $x,y\in \cA_k$ with $x\neq y$. Then
\begin{equation}
\Pg_{s,t}f(x) - \Pg_{s,t}f(y)
=
\Eg_{\Qg_{t-s}}\big[f(X^x_{t-s})e^{v_{s,t}(X^x_\cdot)} - f(X^y_{t-s})e^{v_{s,t}(X^y_\cdot)}\big]
.
\end{equation}
To control $[\Pg_{s,t}f]_{\alpha,\, \cA_k}$ we will bound $\|f(X^\cdot)\|_{\alpha,\, \cA_k}$ and $\|e^{v_{s,t}}\|_{\alpha,\, \cA_k}$. 
We can bound $e^{v_{s,t}}$ on the set $I^x_\ell(s,t)$ for $\ell\in\N$,   
but the resulting bound is only
summable only we can argue that the final value $X^x_{t-s}$ of the diffusion is sufficiently far from $\sup_{u\leq t-s}|X^x_{t-s}|$. 
For $\ell,\ell'\in\N$, with
\begin{equation}
J^{x,y}_{\ell,\ell'}(s,t) 
:=
I^x_\ell(s,t) \cap I^y_{\ell'}(s,t)
,
\label{eq_def_J_ell}
\end{equation}
we therefore split $[\Pg_{s,t}f]_{\alpha,\, \cA_k}$ as:
\begin{align}
\sup_{\substack{x\neq y\in \cA_k \\ |x-y|\leq 1}}
&\frac{|\Pg_{s,t}f(x)-\Pg_{s,t}f(y)|}{|x-y|^\alpha}
\nnb
&\qquad\leq  
\sup_{\substack{x,y\in \cA_k \\ x\neq y}}\frac{1}{|x-y|^\alpha}\sum_{\ell,\ell'\geq k}
\Eg_{\Qg_{t-s}}\Big[{\bf 1}_{J^{x,y}_{\ell,\ell'}(s,t)} \big[f(X^x_{t-s})e^{v_{s,t}(X^x_\cdot)} - f(X^y_{t-s})e^{v_{s,t}(X^y_\cdot)}\big]\Big]
\nnb
&\qquad= 
\sup_{\substack{x,y\in \cA_k \\ x\neq y}}
\frac{1}{|x-y|^\alpha}\Big\{\sum_{\ell,\ell'\geq k}T^{x,y}_1(\ell,\ell') + \sum_{\ell,\ell'\geq k}T^{x,y}_2(\ell,\ell') \Big\}
,
\label{eq_splitting_Hoelder_Pst}
\end{align}
where:
\begin{align}
T^{x,y}_1(\ell,\ell')
&:=
\Eg_{\Qg_{t-s}}\Big[{\bf 1}_{J^{x,y}_{\ell,\ell'}(s,t)} {\bf 1}_{X^x_{t-s},X^y_{t-s}\in S(f)}\big[f(X^x_{t-s})e^{v_{s,t}(X^x_\cdot)} - f(X^y_{t-s})e^{v_{s,t}(X^y_\cdot)}\big]\Big]
\nnb
T^{x,y}_2(\ell,\ell') 
&:=
2\Eg_{\Qg_{t-s}}\Big[{\bf 1}_{J^{x,y}_{\ell,\ell'}(s,t)} {\bf 1}_{X^x_{t-s}\in S(f),X^y_{t-s}\notin S(f)}f(X^x_{t-s})e^{v_{s,t}(X^x_\cdot)}\Big]
.
\end{align}
The term $T^{x,y}_1(\ell,\ell')$ will have good summability in $\ell,\ell',k$ when $x,y\in \cA_k$ thanks to the event $X^x_{t-s},X^y_{t-s}\in S(f)$ which ensures that the diffusion has travelled a distance of order $\max\{L^\ell,L^{\ell'}\}$ on $[0,t-s]$, 
together with the bounds of Lemma~\ref{lemm_LD_X}. 
The summability for $T^{x,y}_2(\ell,\ell')$ is worse since we do not have $X^y_{t-s}\in S(f)$, 
i.e., $X^y_{\cdot}$ may not have come close to the origin than $\cA_k$ on $[0,t-s]$.  
This is mitigated by the fact that then $f(X^y_{t-s})=0$ cancels the $e^{v_{s,t}(X^y_\cdot)}$ term which would otherwise make the sum on $\ell'$ divergent. 

Let us first bound $T^{x,y}_1(\ell,\ell')$. 
Recall the following elementary identities, valid for real-valued functions $g,h$ and any $a,b\in\R$:
\begin{align}
\big|e^{g(a)}-e^{g(b)}\big|
&\leq 
|g(a)-g(b)| e^{\|g\|_{\bbL^\infty(\R^2)}}
\nnb
|h(a)e^{g(a)}-h(b)e^{g(b)}|
&\leq 
|h(a)-h(b)|e^{\|g\|_{\bbL^\infty(\R^2)}} 
+\|h\|_\infty e^{\|g\|_{\bbL^\infty(\R^2)}}|g(a)-g(b)|
.
\end{align}
In addition, Equation~\eqref{eq_def_vst} defining $v_{s,t}$ implies:
\begin{align}
|v_{s,t}(X^x_\cdot)-v_{s,t}(X^y_{\cdot})|
&\leq 
|\psi_t(X^x_{t-s})-\psi_t(X^y_{t-s})| + |\psi_s(x)-\psi_s(y)|
\nnb
&\quad +\int_0^{t-s}\Big[ |\psi_{s+u}(X^x_u)-\psi_{s+u}(X^y_u)| + \big||\nabla \psi_{s+u}(X^x_u)|^2-|\nabla\psi_{s+u}(X^y_u)|^2\big|\Big]\, du
.
\end{align}
We therefore find (recall definition~\eqref{eq_def_CtL} of $C_{t,L}(\cdot)$):
\begin{align}
\sup_{\substack{x,y\in \cA_k \\ x\neq y}}\frac{|T^{x,y}_1(\ell,\ell')|}{|x-y|^\alpha}
&\leq 
4\|f\|_{{\alpha}}\max\Big\{C_{t,L}(\ell)e^{6C_{t,L}(\ell)}, C_{t,L}(\ell')e^{6C_{t,L}(\ell')}\Big\}
\nnb
&\qquad \times 
\sup_{\substack{x,y\in \cA_k\\ x\neq y}}\frac{\max\{t-s,1\}}{|x-y|^\alpha}\sup_{u\leq t-s}\Eg_{\Qg_{t-s}}\Big[{\bf 1}_{J^{x,y}_{\ell,\ell'}(s,t)} {\bf 1}_{X^x_{t-s},X^y_{t-s}\in S(f)}|X^x_{u}-X^y_u|^{\alpha}\Big]
.
\end{align}
H\"older's inequality then gives:
\begin{align}
\sup_{\substack{x,y\in \cA_k \\ x\neq y}}\frac{|T^{x,y}_1(\ell,\ell')|}{|x-y|^\alpha}
&\lesssim 
\|f\|_{{\alpha}}\max\Big\{C_{t,L}(\ell)e^{6C_{t,L}(\ell)}, C_{t,L}(\ell')e^{6C_{t,L}(\ell')}\Big\}
\nnb
&\qquad 
\times \sup_{x\in \cA_k}\Qg_{t-s}\big( I^x_{s,t}(\ell), X^x_{t-s}\in S(f)\big)^{1/3}
\sup_{y\in \cA_k}\Qg_{t-s}\big( I^y_{s,t}(\ell'), X^y_{t-s}\in S(f)\big)^{1/3}
\nnb
&\qquad 
\times \sup_{\substack{x,y\in \cA_k\\ x\neq y}}\frac{\max\{t-s,1\}}{|x-y|^\alpha}\sup_{u\leq t-s}
\Eg_{\Qg_{t-s}}\Big[{\bf 1}_{J^{x,y}_{s,t}(\ell,\ell')}|X^x_{u}-X^y_u|^{3\alpha}\Big]^{1/3}
\label{eq_bound_T_1}
.
\end{align}
Consider now $T_2^{x,y}(\ell,\ell')$. 
A similar argument gives:
\begin{align}
\sup_{\substack{x,y\in \cA_k \\ x\neq y}}\frac{|T^{x,y}_2(\ell,\ell')|}{|x-y|^\alpha}
&\lesssim 
\|f\|_{\alpha} e^{6C_{t,L}(\ell)}
\sup_{x\in \cA_k}\Qg_{t-s}\big( I^x_{s,t}(\ell), X^x_{t-s}\in S(f)\big)^{1/2}
\nnb
&\qquad
\times \sup_{\substack{x,y\in \cA_k\\ x\neq y}}\frac{\max\{t-s,1\}}{|x-y|^\alpha}\sup_{u\leq t-s}
\Eg_{\Qg_{t-s}}\Big[{\bf 1}_{J^{x,y}_{s,t}(\ell,\ell')}|X^x_{u}-X^y_u|^{2\alpha}\Big]^{1/2}
\label{eq_bound_T_2}
.
\end{align}
The probabilities appearing on the right-hand side of~\eqref{eq_bound_T_1}--\eqref{eq_bound_T_2} are estimated by Lemma~\ref{lemm_LD_X},
and the expected variation of $X^{x}_\cdot$ is estimated by Lemma~\ref{lemm_gradientX}.
Indeed, these give with $C(\alpha)>0$,
\begin{align}
&\sup_{\substack{x,y\in \cA_k \\ x\neq y}}\frac{|\Pg_{s,t}f(x)-\Pg_{s,t}f(y)|}{|x-y|^\alpha}
\nnb
&\qquad\lesssim 
\|f\|_{\alpha}\bigg[\Big(\sum_{\ell\geq k} C_{t,L}(\ell)e^{C(\alpha)C_{t,L}(\ell)^{\frac{1}{1-\alpha}}} e^{-L^{2\ell}/16t}\Big)^2
\nnb
&\hspace{2cm}
+ \sum_{\ell\geq k} e^{C(\alpha)C_{t,L}(\ell)^{\frac{1}{1-\alpha}}} e^{-L^{2\ell}/16t}\sum_{\ell'\geq k} e^{C(\alpha)C_{t,L}(\ell')^{\frac{1}{1-\alpha}}}\Big[{\bf 1}_{\ell'\in\{k,k+1\}} + {\bf 1}_{\ell'>k+1}e^{-L^{2\ell'}/16t}\Big]\bigg]
.
\end{align}
Since~\eqref{e:CtL-def} and \eqref{eq_assumption_L_proof_decay_kgeq1}
imply $C_{t,L}(p)^{\frac{1}{1-\alpha}}\leq \frac14 L^p$ for $p\geq 1$ and since $C_{t,L}(0)\leq C_{t,L}(1)$, 
the right-hand side above is bounded for some $c,c'>0$ by:
\begin{equation}
c\|f\|_{\alpha}\Big[C_{t,L}(1)^2e^{2C(\alpha)C_{t,L}(1)^{\frac{1}{1-\alpha}}}{\bf 1}_{k=0} + {\bf 1}_{k\geq 1}e^{-c'L^{2k}/t}\Big]
\lesssim
{\bf 1}_{k=0}\|f\|_{\alpha}e^{L^\theta} + {\bf 1}_{k\geq 1}\|f\|_{\alpha}e^{-c'L^{2k}/t}
.
\end{equation}
This concludes the proof of Proposition~\ref{prop_FK_term}.
\end{proof}

\subsection{Proof of Lemma~\ref{lemm_general_bounds}}\label{sec_proof_lem_general_bounds}
Proposition~\ref{prop_gaussian_bounds} establishes bounds on the Wick powers of $Z_\cdot$, $Z^L_\cdot$. 
These imply, for some $\beta>0$, each $r\in(0,2],A>0$ and some $\epsilon_r>0$:
\begin{equation}
\bbP\Big( \sup_{s\leq t}\, (s^{n\alpha}\wedge 1)\|\wick{(Z_s)^n}\|_{-n\alpha,\rho^n}
>A^{n/r}\Big)
\lesssim 
\exp\big[\epsilon_r \|\varphi_0\|_{-\alpha,\rho}^{r}\big] e^{-A^{r/2}}(1+t)^\beta
,
\end{equation}
where the $\lesssim$ hides a constant independent of $t,\varphi_0,L$. 
Taking $r=1$ and choosing $A^{r/2}=(1+t)^3(\log L)^3$, 
the bounds of Lemma~\ref{lemm_general_bounds} on Wick powers of $Z_\cdot$, $Z^L_\cdot$ hold on an event with probability at most $C\exp[C \|\varphi_0\|_{-\alpha,\rho}] e^{-t^2\log L}$ for $C$ independent of $t,\varphi_0,L$ as claimed. 

From Theorem~\ref{thm:apriori_bounds} we obtain a similar claim for the first bound in~\eqref{eq_bounds_YZ} involving $v^L,v$. 
Indeed,  
let $\eta=\frac{1+\alpha'}{1-3\alpha}$ with $\alpha'>\alpha>0$ small enough. 
For any $A>0$, 
some $C_0$ depending only on $\alpha,\rho$ and a different $\epsilon_r>0$, 
\begin{align}
\bbP\Big( \sup_{s\leq t}\|v^L_s\|_{\alpha',\rho^\eta}>A^{\eta/r}\Big)
&\leq 
\bbP\Big( 1+\max_{1\leq n\leq 3}\sup_{s\leq t}\, (s^{n\alpha}\wedge 1)\|\wick{(Z^L_s)^n}\|_{-n\alpha,\rho^{n}}
\geq 
C_0A^{n/r}\Big)
\nnb
&\lesssim
 (1+t)^\beta\, \exp\big[\epsilon_r \|\varphi_0\|_{-\alpha,\rho}^{r}\big]\, e^{-C_0A^{r/2}}
,
\end{align}
and the same holds for $v$. 
Choosing $r=1$ and $A^{\eta/r}=(1+t)^{5\eta}(\log L)^{5\eta}$, 
the first bound in~\eqref{eq_bounds_YZ} is seen to hold except with probability $\lesssim\exp\big[\epsilon_1\|\varphi_0\|_{-\alpha,\rho}\big]\, e^{-t^2\log L}$, 
where the $\lesssim$ is independent of $t,\varphi_0,L$.  
This concludes the proof of Lemma~\ref{lemm_general_bounds}.

\section{Finite time bound on the relative entropy}
\label{sec:entropy}

The goal of this section is to prove Theorem~\ref{thm:ent-claim} (restated in Proposition~\ref{prop:ent-claim} below).
Throughout this section, we consider the finite volume dynamics $\varphi^L_\cdot$ ($L>0$),
and assume that the initial condition $\varphi_0=\varphi_0^L$ is given in terms of $\varphi_0 \in C^{-\alpha}(\rho)$ defined on the full plane by \eqref{e:phi0L}.
As the infinite volume dynamics will not play a role in this section,
the dependence on $L$ is omitted from the notation.
Thus we denote by $m_t$ the law of $\varphi_t$ such that $m_0 = \delta_{\varphi_0}$ and by $m_\infty$ the $\varphi^4$ measure $\nu_L$ on $\T_L^d$.
The relative entropy of a probability measure $\mu$ with respect to another measure $\nu$ is denoted by $\bbH(\mu|\nu)$,
see \eqref{e:H-ent-def}.

\begin{proposition} \label{prop:ent-claim}
  Let $d=2$, and let $\alpha,\sigma>0$ be sufficiently small.
  Then there is a constant $\beta>0$ such that for any deterministic initial condition $\varphi_0 \in C^{-\alpha}(\rho)$:
  \begin{equation} \label{e:ent-claim_sec5}
    \bbH(m_1|m_\infty) \lesssim  (1+\|\varphi_0\|_{-\alpha,\rho}^{8}) L^{2\beta},
  \end{equation}
  with the implied constant depending on $\lambda,\mu,\alpha,\sigma$.
\end{proposition}

The corresponding statement of course also holds in dimension $d=1$.
We expect an analogous statement to hold also in dimension $d=3$, but we do not have a proof at the moment.
The power $L^{2\beta}$ instead of what would be the optimal estimate $L^d$ results from pathwise estimates involving
the parameter $\sigma>0$ in the weight.

\subsection{Strategy}

We will estimate the entropy on the left-hand side of \eqref{e:ent-claim_sec5} in terms of the following decomposition
and show that each term satisfies the claimed bound:
\begin{align} \label{e:ent-decomp}
  \bbH(m_t|m_\infty)
  &= \int \log\Big(\frac{dm_t}{dm_\infty}\Big)\, dm_t
    \nnb
  &= 
  \int \log\Big(\frac{dm_t}{dm^0_t}\Big)\, dm_t
  +\int \log\Big(\frac{dm^0_t}{dm^0_\infty}\Big)\, dm_t
  +\int \log\Big(\frac{dm^0_\infty}{dm_\infty}\Big)\, dm_t
  .
\end{align}
The superscript $0$ indexes the Gaussian version of the measures.
Thus $m_t^0(d\varphi)$ is the Gaussian law of
the density at time $t$ of the Ornstein--Uhlenbeck dynamics (see Section~\ref{sec_gaussian_estimates})
\begin{equation}
  (\partial_t - A)Z_t = \sqrt{2}\dot W,
\end{equation}
with the same initial condition $Z_0=\varphi_0$, i.e., $m_0^0 = m_t = \delta_{\varphi_0}$
(where we again recall that all these objects are defined on $\T^2_L$ but the $L$ is dropped from the notation). 
The Gaussian invariant measure is denoted $m_\infty^0$ and formally reads:
\begin{equation}
  m_\infty^0(d\varphi) \propto e^{-\frac12 (\varphi,A\varphi)} \, d\varphi.
\end{equation}
We will see that the integrands in the right-hand side of~\eqref{e:ent-decomp} are well defined densities and provide an explicit bound on each integral, 
thereby justifying the decomposition~\eqref{e:ent-decomp}. 
For $m^0_t\ll m^0_\infty$ this can be checked by direct computations, 
while $m^0_\infty\ll m_\infty$ follows from the Nelson estimate.  
The estimate of the first term $\bbH(m_t|m^0_t)$ is more difficult and will take up the rest of this subsection.

Note that the strategy of comparing with the Gaussian case necessarily produces a divergent bound in dimensions $d>2$, 
but in dimension $2$ it will be enough.

To compute $\bbH(m_t|m^0_t)$, 
we first express it in the next lemmas in terms of a relative entropy on path-space. 
Since for any $t>0$ the law ${\bf Q}_t$ of $(\varphi_s)_{s\in[0,t]}$ and ${\bf Q}_t^0$ of the Ornstein-Uhlenbeck process $(Z_s)_{s\in[0,t]}$ are mutually singular
(as shown in \cite{MR1661767}),
the direct bound $\bbH(m_t|m^0_t)\leq \bbH({\bf Q}_t|{\bf Q}_t^0)$ would not be useful. 
Following the method of~\cite{MR4476105}, this is circumvented by rewriting the law of $\varphi^4_t$ for fixed $t>0$
as the solution of a time-shifted equation with a more regular drift. 

To this end, write $B$ for the drift in the equation defining $\varphi_\cdot$: 
for $t>0$, 
\begin{equation}
B_t
=
B(\varphi_t,\wick{\varphi^3_t})
,\qquad 
B: (\phi,\psi) \in C^{-\alpha}(\rho)\times C^{-3\alpha}(\rho)
\longmapsto 
\lambda \psi+ (\mu-1)\phi
,
\label{eq_def_driftB}
\end{equation}
where we recall that Wick powers of $\varphi$ are defined in~\eqref{eq_def_Wick_varphi}. 
We also define the following modified drift which appears in the time-shifted dynamics:
\begin{equation}
\tilde B_{s,t}(x)
:=
2{\bf 1}_{[t/2,t]}(s)\big(e^{-(t-s) A}B_{2s-t}\big)(x)
,\qquad 
x\in\T^2_L
.
\end{equation}
Since $2s-t \leq t$ for $s\in [t/2,t]$, the process $s\mapsto \tilde B_{s,t}$ is a preditable process of the noise.
 
\begin{lemma}[Time-shifted dynamics]\label{lemm_def_tildephi}
Let $t>0$ and let $(\varphi_s)_{s\in [0,t]}$ and $(\tilde\varphi_s)_{s\in[0,t]}$ respectively be mild solutions on $[0,t]$ of:
\begin{equation}
(\partial_s + A)\varphi_s 
=
-B_s
+\sqrt{2}\dot W_s
\label{eq_time_unchanged_SDE}
\end{equation}
and
\begin{equation}
(\partial_s + A)\tilde\varphi_s 
=
-\tilde B_{s,t} 
+\sqrt{2}\dot W_s
\label{eq_time_changed_SDE}
.
\end{equation}
Assume $\tilde\varphi_0=\varphi_0$. 
Then $\varphi_t=\tilde\varphi_t$.
\end{lemma}
\begin{proof}
By assumption,
\begin{equation}
\varphi_t 
=
e^{-tA}\varphi_0 - \int_0^t e^{-(t-s)A}B_s\, ds + \sqrt{2}\int_0^t e^{-(t-s)A}\, dW_s
.
\end{equation}
Changing variables from $s \in [0,t]$ to $2s-t$ with $s\in [t/2,t]$,  the drift term can be rewritten as:
\begin{align}
- \int_0^t e^{-(t-s)A}B_s\, ds
&=
- \int_0^t e^{-\frac{1}{2}(t-s)A}\Big(e^{-\frac12(t-s)A}B_s\Big)\, ds
\nnb
&=
- 2\int_{t/2}^t e^{-(t-s)A}\Big(e^{-(t-s)A}B_{2s-t}\Big)\, ds
\nnb
&= 
-\int_0^t e^{-(t-s)A}\Big(2{\bf 1}_{[t/2,t]}(s) \, e^{-(t-s)A}B_{2s-t}\Big)\, ds
.
\end{align}
Thus $\varphi_t=\tilde\varphi_t$, 
with $\tilde\varphi_\cdot$ as in the lemma.
\end{proof}
We next use the dynamics $(\tilde\varphi_s)_{s\in[0,t]}$ to express the relative entropy $\bbH(m_t|m^0_t)$.
While $\bbH({\bf Q}_t|{\bf Q}_t^0)$ is infinite, 
the path space relative entropy between the law of $(\tilde\varphi_s)_{s\in[0,t]}$ and the corresponding Ornstein--Uhlenbeck process if finite 
thanks to the time shift.
\begin{lemma} \label{lem:ent-Giransov}
Let ${\bf Q}^0_t$ denote the law of the solution $( Z_s)_{s\in[0,t]}$ of the stochastic heat equation on $[0,t]$ with initial condition $\varphi_0$:
\begin{equation}
  (\partial_s+A) Z_s 
  =
   \sqrt{2}\dot W_s,\qquad 
   Z_0 = \varphi_0
  .
  \end{equation}  
 Let also $\tilde{\bf Q}_t$ denote the law of $(\tilde\varphi_s)_{s\in[0,t]}$. 
 Then
  \begin{align} \label{e:ent-Girsanov}
    \bbH(m_t|m^0_t)
    &\leq 
      \bbH(\tilde {\bf Q}_t|{\bf Q}_t^0)
      \nnb
    &\leq  
    \E\qa{ \frac14 \int_0^t\big\|\tilde B_{s,t}\big\|^2_{\bbL^2(\T^2_L)} \, ds}
    =
    \E\qa{ \frac12\int_{0}^t\big\|e^{-\frac12(t-s)A}B_s\big\|^2_{\bbL^2(\T^2_L)} \, ds}
    .
  \end{align}
\end{lemma}
\begin{proof}
The first inequality in~\eqref{e:ent-Girsanov} comes from the fact that $\varphi_t$ and $\tilde\varphi_t$ have the same law
and thus that $m_t$ and $m_t^0$ are marginals of $\tilde {\bf Q}_t$ and ${\bf Q}^0_t$ respectively. 
The last equality in~\eqref{e:ent-Girsanov} is just a change of variable.
The middle equality is a consequence of the Girsanov formula, as follows.

We will apply the version of the Girsanov theorem stated in \cite[Theorem 10.14]{MR3236753}.
There $U$ is a (large) Hilbert space in which the process $W$ takes values
and $\cov(W_t)=tQ$ where $Q: U \to U$ a symmetric, 
positive definite, trace class operator. The reproducing kernel of the process is the space $U_0=Q^{1/2}(U)$
with inner product $(\cdot,\cdot)_0$ and norm $\|\cdot\|_0$.
The cylindrical Wiener process on $\bbL^2(\T^2_L)$ we are interested in is by construction so that $U_0=\bbL^2(\T_L^2)$,
see \cite[Section 4.1.2]{MR3236753}.
Denote by ${\bf P}_t$ the probability measure on a filtered probability space on which $(W_s)_{s\leq t}$ is defined. 
The Girsanov theorem states that if $\psi: [0,t] \to U_0=\bbL^2(\T_L^2)$ is a predictable process then
\begin{equation}
  \hat W_s = W_s -\int_0^s\psi_u\, du, \qquad s\in [0,t],
\end{equation}
is a $Q$-Wiener process with respect to the measure 
\begin{equation}
  {\rm d}\hat {\bf P}_t= \exp\pa{\int_0^t (\psi_s, dW_s)_{0}- \frac12 \int_0^t \|\psi_s\|^2_{0}  \, ds}
  {\rm d}{\bf P}_t
  ,
\end{equation}
provided the exponential has expectation $1$ with respect to ${\bf P}_t$. The last condition is implied by
Novikov's condition \cite[Proposition~10.17]{MR3236753}:
\begin{equation} \label{e:Girsanov-condition}
  \E\qa{\exp\pa{\frac12 \int_0^t \|\psi_s\|_{0}^2 \, ds}} < \infty.
\end{equation}
Under this assumption, it follows that
\begin{align} \label{e:Girsanov-entropy}
  \bbH(\hat {\bf  P}_t | {\bf P}_t)
  &= \hat \E\qa{\int_0^t (\psi_s, dW_s)_{0}- \frac12 \int_0^t \|\psi_s\|_{0}^2 \, ds}
    \nnb
      &= \hat \E\qa{\int_0^t (\psi_s, d\tilde W_s)_{0}+ \frac12 \int_0^t \|\psi_s\|_{0}^2\, ds}
  = \hat \E\qa{\frac12 \int_0^t \|\psi_s\|_{0}^2 \, ds }.
\end{align}
In general, even if \eqref{e:Girsanov-condition} does not hold, it is true that
\begin{equation} \label{e:Girsanov-entropy-general}
  \bbH(\hat {\bf  P}_t | {\bf P}_t)
  \leq \hat \E\qa{\frac12 \int_0^t \|\psi_s\|_{0}^2\, ds }.
\end{equation}
This follows by a localization argument.
Define the stopping times $\tau_N = \inf\{t\geq 0: \int_0^t \|\psi_s\|_0^2 \, ds \geq N\}$,
and notice that $\tau_N \to \infty$ almost surely if the right-hand side  of \eqref{e:Girsanov-entropy-general} is finite (and otherwise the assertion is trivial).
Then lower semicontinuity of the entropy with respect to weak convergence in the first inequality below, the equality
\eqref{e:Girsanov-entropy} applied to the bounded process $\psi_{s\wedge \tau_N}$,
and monotone convergence in the last equality below give
\begin{equation}
  \bbH(\hat {\bf  P}_t | {\bf P}_t)
  \leq \liminf_{N\to\infty}
  \bbH(\hat {\bf  P}_t^{(N)} | {\bf P}_t)
  = \liminf_{N\to\infty} \hat \E\qa{\frac12 \int_0^t \|\psi_{s\wedge \tau_N}\|_{0}^2 \,ds }
  = \hat \E\qa{\frac12 \int_0^t \|\psi_{s}\|_{0}^2 \,ds},
\end{equation}
where $\hat {\bf P}_t^{(N)}$ denotes the law corresponding to the shift $\psi_{s\wedge \tau_N}$.

We will apply this with $\psi_s =  -\frac{1}{\sqrt{2}}\tilde B_{s,t}$. Then if
\begin{equation}
  d\varphi_s
  = -A\varphi_s \, ds + \sqrt{2} dW_s
    = -A\varphi_s\, ds - \tilde B_{s,t} \, ds + \sqrt{2}d\hat W_s
  ,
\end{equation}
so that by the first equality of the last display $(\varphi_s)_{s\leq t}$ has law ${\bf Q}_t^0$  under ${\bf P}_t$,
by the second equality and the fact that $\hat W$ is a cylindrical Wiener process under $\hat {\bf P}_t$ as discussed above,
$(\varphi_s)_{s\leq t}$ has law $\tilde {\bf Q}_t$ under $\hat{\bf P}_t$. Therefore
\begin{equation}
  \bbH(\tilde {\bf Q}_t | {\bf Q}_t^0) \leq \bbH(\hat {\bf  P}_t | {\bf P}_t)
  \leq
  \E\qa{\frac14 \int_0^t \|\tilde B_{s,t}\|_{\bbL^2(\T^2_L)}^2 \, ds},
\end{equation}
as claimed.
\end{proof}

\subsection{Proof of Proposition~\ref{prop:ent-claim}}

We estimate the three terms in the entropy decomposition \eqref{e:ent-decomp}. 

\paragraph{First term}
Let $\sigma>0$, $\sigma'>\sigma$ to be chosen below and write $\rho=(1+|\cdot|^2)^{-\sigma/2}$ and idem for $\rho'$. 
The right-hand side of \eqref{e:ent-Girsanov} was nearly already bounded in the proof of Lemma~\ref{lem:ent-Giransov}.
Indeed, by definition of $(\tilde\varphi_s)_{s\in[0,t]}$ in Lemma~\ref{lemm_def_tildephi} and
the smoothing property of the heat kernel from Proposition~\ref{prop:Besov-heat}:
\begin{align}
&\E\Big[\int_0^t\big\|e^{-\frac{1}{2}(t-s)A}B_s\big\|^2_{\bbL^2(\T^2_L)}\, ds\Big]\nnb
&\lesssim 
L^2\E\Big[\int_{0}^t\Big[ \|e^{-\frac12(t-s)A}\wick{\varphi^3_{s}}\|^2_{\bbL^\infty(\T^2_L)}+ \|e^{-\frac{1}{2}(t-s)A}\varphi_{s}\|^2_{\bbL^\infty(\T^2_L)}\Big]\, ds\Big]
\nnb
&\lesssim
L^{2+2\sigma'}\E\Big[\int_{0}^t\Big[ \|e^{-\frac12(t-s)A}\wick{\varphi^3_{s}}\|^2_{\rho'}+ \|e^{-\frac{1}{2}(t-s)A}\varphi_{s}\|^2_{\rho'}\Big]\, ds\Big]
\nnb
&\lesssim
  L^{2+2\sigma'} \E\Big[ \sup_{s\leq t} \Big[  (s^{6\alpha} \wedge 1)\|\wick{\varphi^3_{s}}\|^2_{-3\alpha,\rho'}+ (s^{2\alpha} \wedge 1)\|\varphi_{s}\|^2_{-\alpha,\rho'}\Big]\Big]\int_{0}^t e^{-(t-s)} (t-s)^{-\alpha} s^{-6\alpha}\, ds
   .
\end{align}
The apriori estimates of Theorem~\ref{thm:apriori_bounds} in the form of Corollary~\ref{cor:Wick-power-moment} complete the bound: 
for $t\leq 1$, 
\begin{align}
\E\Big[\int_0^t\big\|e^{-\frac{1}{2}(t-s)A}B_s\big\|^2_{\bbL^2(\T^2_L)}\, ds\Big]
&\lesssim 
  L^{2+2\sigma'}
  (1+\|\varphi_0\|_{-\alpha,\rho}^{6+\kappa})
  \leq
    L^{2+2\sigma'}
  (1+\|\varphi_0\|_{-\alpha,\rho}^{8})
,
\label{eq_bound_first_term_entropy}
\end{align}
where the last equation comes from the fact that $\kappa>0$ can be made arbitrarily small (for $\alpha>0$ small enough), 
and where $\sigma'>\sigma$ can be chosen appropriately depending on $\rho$ (in particular it can be made arbitrarily close to $\sigma$ if $\alpha$ is small).

\paragraph{Second term}
The second term involves the relative density of the Ornstein--Uhlenbeck process at time $t>0$ and its invariant measure.
The invariant measure $m_\infty^0$ is a Gaussian measure with covariance $A^{-1}$ where $A= -\Delta+1$ and mean $0$,
and the law $m_t^0$ of the Ornstein--Uhlenbeck process at time $t>0$ is a Gaussian measure with covariance and mean given by \eqref{e:OU-cov}--\eqref{e:OU-mean}:
\begin{equation}
  2\int_0^t e^{-2sA}\, ds = A^{-1} (1-e^{-2tA})
  \quad \text{and} \quad    e^{-tA} \varphi_0.
\end{equation}
In the finite-dimensional analogue, it is a straightforward computation that the relative density of $m_t^0$ with respect to $m_\infty^0$ is given by
\begin{align}
  \frac{dm_t^0}{dm_\infty^0}(\varphi)
  &=
    \exp\qa{-\frac{1}{2}\big(e^{-tA} \varphi, A(1-e^{-2tA})^{-1}e^{-tA} \varphi\big) - \frac12 \log\det(1-e^{-2tA})}
    \nnb
  &\qquad\qquad \times \exp\qa{-\frac12 \big(e^{-At}\varphi_0,A(1-e^{-2tA})^{-1}e^{-At}\varphi_0\big) + \big(\varphi, A(1-e^{-2tA})^{-1}e^{-At}\varphi_0)\big)}
    ,
\end{align}
where we used that
\begin{gather}
  \frac{A}{1-e^{-2tA}}-A
  = \frac{Ae^{-2tA}}{1-e^{-2tA}}
  .
\end{gather}
Since the covariances $A^{-1}$ and $A^{-1}(1-e^{-2tA})$ can be diagonalised simultaneously (in the Fourier basis)
and since $e^{-2tA}$ is trace class,  the same formula holds in our infinite dimensional situation by truncating Fourier modes and taking limits,
with the determinant interpreted in the limit as a Fredholm determinant:
\begin{equation}
  \log \det(1-e^{-2tA})= \tr \log (1-e^{-2tA}). 
\end{equation}
Therefore
\begin{align} \label{e:secondterm-decomp}
  \int \log\Big(\frac{dm^0_t}{dm^0_\infty}\Big)\, dm_t
  &=
    -\frac12 \log\det(1-e^{-2tA})
    \nnb
    &\quad -\frac12 \big(e^{-At}\varphi_0,A(1-e^{-2tA})^{-1}e^{-At}\varphi_0\big)
  \nnb
  &\quad
    -\frac12 \int \big(e^{-tA}\varphi, A(1-e^{-2tA})^{-1}e^{-tA}\varphi\big)   \, dm_t
    \nnb
    &\quad
    + \int \big(\varphi, A(1-e^{-2tA})^{-1}e^{-At}\varphi_0\big)   \, dm_t
.
\end{align}
The Fredholm determinant term can be estimated as
\begin{equation}
  -\tr \log (1-e^{-2tA})
  =
 -\sum_{p\in\frac{\pi}{L}\Z^2}
 \log\big(1-e^{-2t(|p|^2+1)}\big)
  \lesssim
   \frac{L^2}{t}e^{-2t}
   \leq 
   \frac{L^2}{t}
   .
 \end{equation}
The second and third terms on the right-hand side of \eqref{e:secondterm-decomp} are negative.
The integrand in the fourth term on the right-hand side of \eqref{e:secondterm-decomp} is bounded by
  \begin{align}
    \big(\varphi, A(1-e^{-2tA})^{-1}e^{-tA}\varphi_0\big)
    &\leq  \|e^{-\frac14 tA}\varphi\|_{\bbL^2}\|e^{-\frac14 tA}\varphi_0\|_{\bbL^2}  \|e^{-\frac12 tA}A(1-e^{-2tA})^{-1}\|
      \nnb
    &\lesssim \frac{L^{2+2\sigma}}{t^{1+\alpha}} \|\varphi\|_{-\alpha,\rho}\|\varphi_0\|_{-\alpha,\rho}
      ,
\end{align}
where we wrote $\|\cdot\|$ for the operator norm on $\bbL^2=\bbL^2(\T^2_L)$ and used that (as quadratic forms on $\bbL^2$):
\begin{equation}
  e^{-\frac12 tA}A(1-e^{-2tA})^{-1}
  \leq \frac{1}{t} \sup_{\lambda>0} \frac{\lambda}{e^{\frac12 \lambda}-e^{-\frac32 \lambda}}
  \leq \frac{1}{t},
\end{equation}
and (by the definition of the weight and Proposition~\ref{prop:Besov-heat}), for any test function $\psi$,
\begin{gather}
  \|\psi\|_{\bbL^2}^2 \leq (2L)^2\|\psi\|_{\bbL^\infty}^2 \lesssim L^{2+2\sigma'} \|\psi\|_{\bbL^\infty(\rho')}^2\\
  \|e^{-\frac14 tA}\psi\|_{\bbL^\infty(\rho')} \lesssim t^{-\alpha/2} \|\psi\|_{-\alpha,\rho'}
  .
\end{gather}
Using the apriori bounds of Corollary~\ref{cor:Wick-power-moment} 
yields $\int \|\varphi\|_{-\alpha,\rho'} \, dm_t \lesssim 1+\|\varphi_0\|_{-\alpha,\rho}^{1+\kappa}$ for $t\leq 1$,
some $\kappa>0$ than can be chosen arbitarily small, and $\rho'$ chosen as in~\eqref{eq_bound_first_term_entropy}.
As a result:
\begin{equation}
  \int     \big(\varphi, A(1-e^{-2tA})^{-1}e^{-tA}\varphi_0\big) \, dm_t
  \lesssim \frac{L^{2+2\sigma'}}{t^{1+\alpha}} (1+\|\varphi_0\|_{-\alpha,\rho}^8)
  .
\end{equation}
Together these bounds yield, for $t=1$,
\begin{equation}
\int \log\Big(\frac{dm^0_t}{dm^0_\infty}\Big)\, dm_t
\lesssim 
L^2
+L^{2+2\sigma'} (1+\|\varphi_0\|_{-\alpha,\rho}^{8})
\lesssim
L^{2\beta} (1+\|\varphi_0\|_{-\alpha,\rho}^8).
\end{equation}

\paragraph{Third term}

The third term in the relative entropy decomposition is the relative entropy of the Gaussian invariant measure with respect to the $\varphi^4$ invariant measure. Since
\begin{equation}
  \frac{dm_\infty}{dm_\infty^0}(\varphi) = \frac{e^{-V}}{\E_{m_{\infty}^0}[e^{-V}]},
\end{equation}
where, by \eqref{e:phi42measure},
\begin{equation}
  V =   \int_{\T_L^2}\pa{ \frac{\lambda}{4}\wick{\varphi^4}  + \frac{\mu}{2}  \wick{\varphi^2}} \,dx,
\end{equation}
one has
\begin{equation}
\int \log\Big(\frac{dm^0_\infty}{dm_\infty}\Big)\, dm_t
=
\log \E_{m_\infty^0}[e^{-V}]
+
\E_{m_t}[V]
.
\end{equation}
The first term on the last right-hand side is $O(L^2)$ 
by Nelson's estimate, see \cite[Lemma~V.6 and Theorem~V.7]{MR0489552}. 
The second term is bounded using the apriori bounds of Corollary~\ref{cor:Wick-power-moment}.
Indeed, if $\Lambda=[-L,L)^2$ then
$\dnorm{1_\Lambda}_{\beta,\rho^{-1}} \lesssim L^{d+\sigma}$
by \eqref{e:1A-B11-bd},
where $\dnorm{\cdot}_{\beta,\rho^{-1}}$ is defined in  \eqref{e:dnorm-def} (a version of the $B^\beta_{1,1}$ norm).
Therefore, for $1>\beta>\alpha>0$, by the duality pairing \eqref{e:Besov-duality},
\begin{equation}
  (\wick{\varphi^n_t}, 1_\Lambda) \lesssim \|\wick{\varphi^n_t}\|_{-\alpha,\rho'} \dnorm{1_\Lambda}_{\beta,(\rho')^{-1}}
  \lesssim L^{d+\sigma'} \|\wick{\varphi^n_t}\|_{-\alpha,\rho'} .
\end{equation}
Applying the apriori bounds of Corollary~\ref{cor:Wick-power-moment} gives, for $t=1$ and sufficiently small $\alpha$:
\begin{equation}
  \E_{m_t}[V]\lesssim
  L^{2+\sigma'} \E\qbb{ \|\wick{\varphi^4_t}\|_{-\alpha,\rho'}+ \|\wick{\varphi^2_t}\|_{-\alpha,\rho'} }
  \lesssim 
  L^{2+2\sigma' } (1+\|\varphi_0\|_{-\alpha,\rho}^{8}),
\end{equation}
where again $\sigma'>\sigma$ is chosen depending on $\alpha,\sigma$.

\appendix
\section{Properties of local Besov--H\"older norms}
\label{app:norms}

This appendix provides proofs of several properties of the local Besov--H\"older norms that we need.
We focus on the standard ``elliptic'' versions of these norms defined in Section~\ref{sec:norms-def}.
The properties of the parabolic versions used in Appendix~\ref{app:apriori} are completely analogous.

\subsection{Weighted $\bbL^p$ spaces}
For $p\in [1,\infty]$, we denote the weighted $\bbL^p$ norm by
\begin{equation}
  \|f\|_{\bbL^p(\rho)} = \|\rho f\|_{\bbL^p(\R^d)},
\end{equation}
and for $p=\infty$ we also write
\begin{equation}
  \|f\| = \|f\|_{\bbL^\infty(\R^d)}, \qquad
  \|f\|_\rho = \|f\|_{\bbL^\infty(\rho)}.
\end{equation}
These definitions do not agree with \cite{MR3693966},
where $\|f\|_{\bbL^p(\rho)}=\|\rho^{1/p}f\|_{\bbL^p}$,
but they do agree with \cite{TriebelIII}, for example.
Throughout we only consider the polynomial weight $\rho$ defined in \eqref{e:weight} and observe that it
satisfies the inequalities
\begin{equation} \label{e:weight-ineq}
  \rho(x)/\rho(z) \leq C\rho(x-z)^{-1}, \qquad
  \rho(x-z)^{-1} \leq C\rho(x)^{-1}\rho(z)^{-1},
  \qquad \rho_R(x)^{-1} \leq \rho(x)^{-1}
\end{equation}
where $\rho_R(x)=\rho(Rx)$, uniformly in $x,z\in \R^d$ and $R\leq 1$.
In particular,
\begin{align} \label{e:Lp-weight}
  \|f*g\|_{\bbL^\infty(\rho)}
  &= \sup_x \rho(x) \absa{\int g(x-y)f(y) \, dy}
    \nnb
  &\leq \qa{\sup_x  \int \frac{\rho(x)}{\rho(y)} g(x-y) \, dy} \|f\|_{\bbL^\infty(\rho)}
    \nnb
  &\leq C \qa{\sup_x  \int \rho(x-y)^{-1} g(x-y) \, dy} \|f\|_{\bbL^\infty(\rho)}
    = C \|g\|_{\bbL^1(\rho^{-1})} \|f\|_{\bbL^\infty(\rho)}
\end{align}
and
\begin{align} \label{e:L1-weight}
    \|g_R\|_{\bbL^1(\rho^{-1})}
  &= \int R^{-d} |g(x/R)| \rho(x)^{-1} \, dx
    \nnb
  &= \int  |g(x)| \rho(Rx)^{-1} \, dx
    = \|g\|_{\bbL^1(\rho_R^{-1})}
    \leq \|g\|_{\bbL^1(\rho^{-1})},
  \end{align}
 where we recall the notation (and the $\rho_R$ convention is only used for weights)
\begin{equation}
g_R(x) = R^{-d} g\Big(\frac{x}{R}\Big), \qquad \rho_R(x) = \rho(Rx).
\end{equation}
Also, if $g$ has support in $B_S(0)$ then
\begin{equation}
  |f*g(x)|
  \leq \|g\|_{\bbL^1} \|f\|_{\bbL^\infty(B_S(x))}
  .
\end{equation}

\subsection{Properties of local Besov--H\"older spaces}

From Section~\ref{sec:norms-def},
we recall our definitions of the local and weighted Besov--H\"older norms
(which are versions of the $B^\alpha_{\infty,\infty}$ norms): for $\alpha<0$,
\begin{align}
    \label{e:norm-minus-C-app}
  \|f\|_{\alpha,C} &= \sup_{\substack{R\leq 1, x\in C:\\B_R(x) \subset C}} |\Psi_R*f(x)| R^{-\alpha},
  \\
  \label{e:norm-minus-rho-app}
  \|f\|_{\alpha,\rho} &= \sup_{R\leq 1} \|\Psi_R*f\|_{\rho} R^{-\alpha}.
\end{align}
We will denote by $\dsemnorm{\cdot}_{\beta,\rho^{-1}}$ a version of the $B^\beta_{1,1}$ seminorm with $\beta \in (0,1)$ and weight $\rho(x)^{-1}$:
\begin{align}
\label{e:dsemnorm-def}
  \dsemnorm{g}_{\beta,\rho^{-1}}
  &= \int \rho(x)^{-1} \int_{B_1(0)} \frac{|g(x)-g(x+y)|}{|y|^{\beta}} \, \frac{dy}{|y|^d} \, dx
  \nnb
  &\sim
  \int_0^1 \qbb{ \int \rho(x)^{-1} \mint_{B_R(0)} \frac{|g(x)-g(x+y)|}{|y|^{\beta}} \, dy \, dx} \frac{dR}{R}
\end{align}
and a corresponding version of the $B^\beta_{1,1}$ norm by:
\begin{equation} \label{e:dnorm-def}
  \dnorm{g}_{\beta,\rho^{-1}} = \|g\|_{\bbL^1(\rho^{-1})} + \dsemnorm{g}_{\beta,\rho^{-1}}
  .
\end{equation}
The weight is always assumed to satisfy \eqref{e:weight-ineq},
and we write $\|f\|_\alpha$ and $\dnorm{g}_\beta$ if $\rho(x)=1$ for all $x$,
and analogously for the seminorms.

The main estimate from which essentially all the remaining ones
in this section are derived is the following convolution estimate.

\begin{proposition}
 \label{prop:norm-tilde}
 Let $\varphi: \R^d \to \R$ be a bounded function with support in $B_1(0)$ and $\int \varphi \, dx = 1$.
 Then for any $1>\beta>\alpha>0$,
 there is a constant $C(\alpha,\beta,\varphi)$ such that for any sufficiently integrable function $g: \R^d \to \R$ with support in $B_S(0)$
 with $S\in (0,\infty]$ and $R \in (0,1]$,
 \begin{equation} \label{e:normbd-local}
   R^\alpha |g_R* f(x)|  
   \leq C(\alpha,\beta,\varphi)
   \qbb{ \|g\|_{\bbL^1} +  \dsemnorm{g}_{\beta}}
  \times \sup_{r \leq R}   r^{\alpha}  \|\varphi_r*f\|_{B_{R(S+1)}(x)} 
  .
\end{equation}
Moreover, for any weight $\rho$ satisfying  \eqref{e:weight-ineq},
\begin{equation} \label{e:normbd-weight}
  R^{\alpha}
  \|g_R*f\|_{\rho}
  \leq C(\alpha,\beta,\rho,\varphi)
  \qbb{\|g\|_{\bbL^1(\rho^{-1}_R)} +  \dsemnorm{g}_{\beta,\rho^{-1}_R}  }
  \times \sup_{r \leq R}r^{\alpha} \|\varphi_r*f\|_{\rho} 
  .
\end{equation}
\end{proposition}

\begin{remark} \label{rk:completion}
  In particular,
  \begin{equation}
    \sup_{R\leq 1}
    R^\alpha \|g_R*f\|_\rho \lesssim \dnorm{g}_{\beta,\rho^{-1}} \|f\|_{\alpha,\rho}.
  \end{equation}
This implies that elements of $C^{-\alpha}(\rho)$, which we defined as the completion of $C_c^\infty(\R^d)$ with respect to $\|\cdot\|_{-\alpha,\rho}$,
can indeed be identified with Schwartz distributions. Indeed, taking $R=1$,
\begin{equation} \label{e:Besov-duality}
	\int g(x)f(x)\, dx \lesssim  \dnorm{g}_{\beta,\rho^{-1}} \|f\|_{-\alpha,\rho}.
\end{equation}
Using that $\dnorm{g}_{\beta,\rho^{-1}}= \|g\|_{\bbL^1(\rho^{-1})} + \dsemnorm{g}_{\beta,\rho^{-1}}$ can be controlled by Schwartz seminorms,
equivalence classes of Cauchy sequences with respect to $\|\cdot\|_{-\alpha,\rho}$ can be identified with Schwartz distributions.
\end{remark}

\begin{proof}
  \emph{Step 1:}
  Denote $\varphi_0= \delta_0$  and $\Delta_r * g = \varphi_0 * g - \varphi_r*g = g-\varphi_r * g$.
  Then there is a constant $C$ (depending on $\rho$ and $\varphi$) such that for any $r,R \in (0,1]$ and $\beta\in (0,1)$:
  \begin{equation} \label{e:Deltarbd}
   \|\Delta_r * g\|_{\bbL^1(\rho^{-1})} \leq C r^\beta \dsemnorm{g}_{\beta,\rho^{-1}}^r, \qquad
   \dsemnorm{\Delta_r * g}_{\beta,\rho^{-1}}^R \leq C \dsemnorm{g}_{\beta,\rho^{-1}}^R,
  \end{equation} 
  where we define
  \begin{equation}
   \dsemnorm{g}_{\beta,\rho^{-1}}^R =
   \int \rho(x)^{-1} \int \frac{|g(x)-g(x+y)|}{|y|^\beta} |\varphi_R(y)| \, dy\, dx.
  \end{equation}
  Indeed,  since $\varphi_r$ has support in $B_r(0)$,
  \begin{align}
    \int \rho(x)^{-1} |\Delta_r * g(x)|\, dx
    &\leq \iint  \rho(x)^{-1}|g(x)-g(x+y)| |\varphi_r(y)| \, dy \, dx
      \nnb
    &\leq r^\beta \iint \rho(x)^{-1} \frac{|g(x)-g(x+y)|}{|y|^\beta} |\varphi_r(y)| \, dy \, dx
    = r^\beta  \dsemnorm{g}_{\beta,\rho^{-1}}^r.
  \end{align}
  For the second inequality, it suffices to show
  that $\dsemnorm{\varphi_r* g}_{\beta,\rho^{-1}}^R \leq C \dsemnorm{g}_{\beta,\rho^{-1}}^R$:
  \begin{align}
    &\int\rho(x)^{-1}  \int_{B_R(0)} \frac{|\varphi_r*g(x)-\varphi_r*g(x+y)|}{|y|^\beta} |\varphi_R(y)| \, dy\, dx
      \nnb
    &=
      \int \rho(x)^{-1} \int_{B_R(0)} \frac{|\int (g(x+z)-g(x+y+z)) \varphi_r(z)\, dz| }{|y|^\beta} |\varphi_R(y)| \, dy\, dx
      \nnb
    &\leq
      \int \rho(x)^{-1} \int_{B_R(0)} \int_{B_r(0)} \frac{|g(x+z)-g(x+y+z)|}{|y|^\beta} |\varphi_R(y)| |\varphi_r(z)|\, dz\,dy\, dx 
      \nnb
    & =\int \int_{B_r(0)} \rho(x-z)^{-1}  \int_{B_R(0)} \frac{|g(x)-g(x+y)|}{|y|^\beta} |\varphi_R(y)| |\varphi_r(z)|\,dy\, dz\, dx 
      \nnb
     & \leq C \int \rho(x)^{-1}\int_{B_r(0)} \rho(z)^{-1} |\varphi_r(z)| \int_{B_R(0)} \frac{|g(x)-g(x+y)|}{|y|^\beta} |\varphi_R(y)|\, dy\,dz\, dx 
      \nnb
    &  \lesssim \int \rho(x)^{-1} \int_{B_R(0)}  \frac{|g(x)-g(x+y)|}{|y|^\beta} |\varphi_R(y)| \, dy\, dx
      ,
  \end{align}
  where we used \eqref{e:weight-ineq} for the weight. It follows that
  $\dsemnorm{\Delta_r * g}_{\beta,\rho^{-1}}^R \leq \dsemnorm{g}_{\beta,\rho^{-1}}^R  + \dsemnorm{\varphi_r* g}_{\beta,\rho^{-1}}^R \lesssim \dsemnorm{g}_{\beta,\rho^{-1}}^R$
  with the implicit constant only depending on $\rho$ and $\varphi$.
  
  \medskip
  \noindent\emph{Step 2:}
  For $\theta\in (0,\frac12]$, define the distributions
  \begin{equation}
    \Phi^{k} = \Delta_\theta * \cdots *\Delta_{\theta^k}, \qquad \Phi^{0} = \delta_0.
  \end{equation}
  Then  $\Phi^k$ has support in a ball of radius $\sum_{n=1}^k \theta^n = \theta(1-\theta^{k})/(1-\theta) \leq 1$ and
  \begin{equation}
    \Phi^0 =  \Phi^0 * \varphi_\theta + \Phi^1
    = \Phi^0 * \varphi_\theta + \Phi^1 * \varphi_{\theta^2} + \Phi^2
    = \dots
    = \sum_{k=1}^n \Phi^{k-1}* \varphi_{\theta^{k}} + \Phi^n
    .
  \end{equation}
  Choosing $0<\theta\ll 1$, the bounds \eqref{e:Deltarbd} imply that $g * \Phi^n \to 0$ in $\bbL^1$ for all smooth $g$ and:
  \begin{equation}
    \|g * \Phi^k\|_{\bbL^1(\rho^{-1})}
    \leq C\theta^{\beta k} \dsemnorm{g*\Phi^{k-1}}_{\beta,\rho^{-1}}^{\theta^k}
      \leq (C\theta^\beta)^{k} \dsemnorm{g}_{\beta,\rho^{-1}}^{\theta^k},
      \qquad k\geq 1
    ,
  \end{equation}
  and trivially $\|g * \Phi^0\|_{\bbL^1(\rho^{-1})} = \|g\|_{\bbL^1(\rho^{-1})}$.
  Therefore, for any $\beta>\alpha$,
  \begin{align} \label{e:gRf-local}
    |g_R* f(x)|
    &\leq \sum_{k \geq 1}
      |g_R * \Phi^{k-1}_R * \varphi_{\theta^k R} * f(x)|
      \nnb
    &\leq \sum_{k\geq 1}
      \int |g_R * \Phi^{k-1}_R(y)||\varphi_{\theta^k R} * f(x-y)| \, dy
      \nnb
    &\leq \sum_{k \geq 1}
      \|g_R*\Phi^{k-1}_R\|_{\bbL^1}
      \|\varphi_{\theta^k R} * f\|_{\bbL^\infty(B_{SR+R}(x))}
            \nnb
   &= \sum_{k\geq 1}
      \|g*\Phi^{k-1}\|_{\bbL^1}
      \|\varphi_{\theta^k R} * f\|_{\bbL^\infty(B_{SR+R}(x))}
            \nnb
    &\lesssim
      R^{-\alpha}
      \qbb{
      \|g\|_{\bbL^1} +
      \sum_{k\geq 1} (C\theta^{\beta})^k \theta^{-\alpha k} \dsemnorm{g}_{\beta}^{\theta^{k-1}}}
      \qbb{ \sup_{r\leq R} r^\alpha \|\varphi_r * f\|_{\bbL^\infty(B_{SR+R}(x))} }
            \nnb
    &\lesssim
      R^{-\alpha} 
      \qbb{ \|g\|_{\bbL^1} + \dsemnorm{g}_{\beta}}
      \qbb{ \sup_{r\leq R} r^\alpha \|\varphi_{r} * f\|_{\bbL^\infty(B_{R(S+1)}(x))}}
    .
  \end{align}
  The last inequality  in \eqref{e:gRf-local} is explained at the end of proof,  after observing that
  the weighted version of \eqref{e:gRf-local} follows in the same way using \eqref{e:Lp-weight}--\eqref{e:L1-weight}:
  \begin{align} \label{e:gRf-weighted}
    \|g_R * f\|_{\bbL^\infty(\rho)}
    &\leq \sum_{k\geq 1}
      \|g_R * \Phi^{k-1}_R * \varphi_{\theta^k R} * f\|_{\bbL^\infty(\rho)}
      \nnb
      &\leq C\sum_{k\geq 1}\|g_R*\Phi^{k-1}_R\|_{\bbL^1(\rho^{-1})}
      \|\varphi_{\theta^k R} * f\|_{\bbL^\infty(\rho)}
        \nnb
    &= C \sum_{k\geq 1}\|g*\Phi^{k-1}\|_{\bbL^1(\rho_R^{-1})}
      \|\varphi_{\theta^k R} * f\|_{\bbL^\infty(\rho)}
            \nnb
    &\lesssim R^{-\alpha} \qbb{ \|g\|_{\bbL^1(\rho^{-1}_R)} + \sum_{k\geq 1} (C\theta^{\beta})^k \theta^{-\alpha k} \dsemnorm{g}_{\beta,\rho_R^{-1}}^{\theta^{k-1}}}
      \qa{ \sup_{r\leq R} r^\alpha \|\varphi_{r} * f\|_{\bbL^\infty(\rho)} }
            \nnb
    &\lesssim  R^{-\alpha} 
    \qbb{ \|g\|_{\bbL^1(\rho_R^{-1})} + \dsemnorm{g}_{\beta,\rho_R^{-1}} }
      \qbb{ \sup_{r\leq R} r^\alpha \|\varphi_{r} * f\|_{\bbL^\infty(\rho)} }
      .
  \end{align}
  Finally, in the last inequalities of \eqref{e:gRf-local} and \eqref{e:gRf-weighted}, we used that
  for any $\beta>\alpha$ and $\theta \ll 1$,
  \begin{equation}
    \sum_{k\geq 1} (C\theta^{\beta})^k \theta^{-\alpha k} \dsemnorm{g}_{\beta,\rho^{-1}}^{\theta^{k-1}}
    \lesssim \iint \rho(x)^{-1}\frac{|g(x)-g(x+y)|}{|y|^{\beta}} \, \frac{dy}{|y|^d} \, dx =\dsemnorm{g}_{\beta,\rho^{-1}}.
  \end{equation}
  Indeed, with $R_k = \theta^k$ and $C \leq \theta^{\alpha-\beta}$ for $\theta$ sufficiently small since $\beta>\alpha$,
  \begin{equation}
    \sum_{k\geq 1} (C\theta^{\beta})^k \theta^{-\alpha k} \dsemnorm{g}_{\beta,\rho^{-1}}^{\theta^{k-1}}
    \lesssim \sum_{k\geq 1} \int \rho(x)^{-1}\mint_{B_{\theta^{k-1}}(0)} \frac{|g(x)-g(x+y)|}{|y|^{\beta}} \, dy \, dx.
  \end{equation}
  This is bounded by
  \begin{align}
    &\lesssim \int_0^1 \frac{dR}{R} R^{-d} \int \rho(x)^{-1}  \int_{B_R(0)} \frac{|g(x)-g(x+y)|}{|y|^{\beta}} \, dy \, dx
      \nnb
    & =\int \rho(x)^{-1} \int_{B_1(0)} \frac{|g(x)-g(x+y)|}{|y|^{\beta}} \qa{ \int_{|y|}^1 \frac{dR}{R} R^{-d}}  \, dy \, dx
    \nnb
    &\lesssim \int \rho(x)^{-1} \int_{B_1(0)}\frac{|g(x)-g(x+y)|}{|y|^{\beta}} \, \frac{dy}{|y|^d} \, dx.
  \end{align}
  This completes the proof.
\end{proof}

\begin{proposition}[Reconstruction / Commutator estimate] \label{prop:besov-mult}
  Let $\alpha \in (0,1)$ and $0<\alpha< \beta$. Then for any $R\leq 1$,
  \begin{equation} \label{e:besov-commutator}
    | (\Psi_R*(uv))(x) - u(x) (\Psi_R*v)(x) | \lesssim  R^{\beta -\alpha} [u]_{\beta,B_{2R}(x)} \|v\|_{-\alpha,B_{2R}(x)}
    .
  \end{equation}
  In particular,
  \begin{equation} \label{e:besov-mult}
    \|uv\|_{-\alpha,C} \lesssim
    \|u\|_{\beta,C+B_{2}(x)} \|v\|_{-\alpha,C+B_{2}(x)},
  \end{equation}
  and for the weighted  norms, if $\rho' \lesssim 1$,
  \begin{equation}  \label{e:besov-mult-rho}
    \|uv\|_{-\alpha,\rho\rho'}
    \lesssim \|v\|_{-\alpha,\rho} \|u\|_{\beta,\rho'}.
  \end{equation}
\end{proposition}

\begin{remark}
  As an application, we mention that if $A \subset \R^d$ is a cube (or other sufficiently regular set) then,
  by Proposition~\ref{prop:norm-tilde} and Proposition~\ref{prop:besov-mult}:
  \begin{align}
    \int_{A} f(x)g(x) \, dx
    &\lesssim \dnorm{1_A}_{\beta} \norm{fg}_{-\alpha, A+B_1(0)}
      \nnb
    &\lesssim \dnorm{1_A}_{\beta} \dnorm{g}_{\beta,A+B_3(0)}\norm{f}_{-\alpha, A+B_3(0)}
    \lesssim |A| \dnorm{g}_{\beta,A+B_3(0)}\norm{f}_{-\alpha, A+B_3(0)}
      \label{eq_besovmult_cubes}
  \end{align}
  where we used in the last inequality that, for a cube $A$ and $\beta \in (0,1)$:
  \begin{equation}
    \dnorm{1_A}_\beta \lesssim  |A|+|\partial A| \lesssim |A|.
  \end{equation}
  More generally, we record for use in Section~\ref{sec:entropy} that   if $A=[-L,L]^d$ then
  \begin{equation} \label{e:1A-B11-bd}
    \dnorm{1_A}_{\beta,\rho^{-1}} \lesssim  L^{d+\sigma}.
  \end{equation}
  Indeed, clearly $\|1_A\|_{\bbL^1(\rho^{-1})} \lesssim L^{d+\sigma}$, and 
\begin{align}
  \dsemnorm{1_A}_{\beta,\rho^{-1}}
  &= \iint_{\R^2\times B_1(0)} \rho(x)^{-1} \frac{|1_A(x)-1_A(x+y)|}{|y|^{\beta}} \, \frac{dy}{|y|^d} \, dx
    \nnb
  &\lesssim \sup_{x\in\partial A}\rho(x)^{-1} \int_{B_1(0)} \frac{|\partial A| |y|}{|y|^\beta} \frac{dy}{|y|^d}
  \lesssim |\partial A|\sup_{x\in\partial A}\rho(x)^{-1}  \int_0^1 r^{-\beta} \, dr
    \lesssim L^{d-1+\sigma}.
\end{align}

\end{remark}

\begin{proof}[Proof of Proposition~\ref{prop:besov-mult}]
  The left-hand side of \eqref{e:besov-commutator} is the absolute value of
  \begin{equation}
    \Psi_R * (uv)(x) - u(x) (\Psi_R * v)(x) = g_R^{x,R}* v(x),
  \end{equation}
  where
  \begin{equation}
    g^{x,R}(z) =\Psi(z)\qB{ u(x-Rz) - u(x) }.
  \end{equation}
  Since $\Psi$ has support in $B_1(0)$, so does $g^{x,R}$.
  By Proposition~\ref{prop:norm-tilde}, 
  \begin{equation}
    R^{\alpha}| \Psi_R* (uv)(x) - u(x) (\Psi_R*v)(x) |  \lesssim \|v\|_{-\alpha,B_{2R}(x)} \dnorm{g^{x,R}}_{\beta}.
  \end{equation}
  The following bound (with $\rho=1$)
  therefore completes the proof of \eqref{e:besov-commutator}:
  \begin{align}
    \dsemnorm{g^{x,R}}_{\beta,\rho^{-1}}
    &\lesssim \sup_{r\leq 1} \int \rho(z)^{-1} \mint_{B_r(0)} \frac{|g^{x,R}(z)-g^{x,R}(z+w)|}{|w|^\beta} \, dw\, dz
      \nnb
    &\leq
      \sup_{r\leq 1} \int \rho(z)^{-1} \mint_{B_r(0)} \frac{|\Psi(z)-\Psi(z+w)|}{|w|^\beta}|u(x-Rz)-u(x)| \,dw\,dz
      \nnb
      &\quad+\sup_{r\leq 1}\int \rho(z)^{-1}\mint_{B_r(0)}  |\Psi(z+w)|\frac{|u(x-Rz-Rw)-u(x-Rz)|}{|w|^\beta} \, dw\, dz
      \nnb
    &\lesssim
      R^\beta [u]_{\beta,B_{2R}(x)},
  \end{align}
  where we used that, for $|w|\leq 1$,
  \begin{equation}
    \frac{|\Psi(z)-\Psi(z+w)|}{|w|^\beta} \lesssim  1_{|z|\leq 2}, \qquad
    |\Psi(z+w)| \lesssim  1_{|z|\leq 2},
  \end{equation}
  and similarly
  \begin{align}
    \dsemnorm{g^{x,R}}_{\bbL^1(\rho^{-1})}
    &= \int \rho(z)^{-1} |g^{x,R}(z)|\, dz
      \nnb
    &= \int \rho(z)^{-1} |\Psi(z)||u(x-Rz)-u(x)|\, dz
    \lesssim R^\beta [u]_{\beta,B_R(x)} .
  \end{align}
  The multiplication estimate \eqref{e:besov-mult} is a direct consequence.

  The weighted version is analogous.   By Proposition~\ref{prop:norm-tilde} and the above estimates for $g^{x,R}$,
  \begin{equation}
    R^{\alpha}| \Psi_R* (uv)(x) - u(x) (\Psi_R*v)(x) |
    \lesssim \|v\|_{-\alpha,\rho} \dnorm{g^{x,R}}_{\beta,\rho^{-1}}
    \lesssim R^\beta \|v\|_{-\alpha,\rho} [u]_{\beta,B_{2R}(x)}.
  \end{equation}
  Using that $[u]_{\beta,\rho} \sim \sup_{x} \rho(x)[u]_{\beta,B_{2R}(x)}$, therefore
  \begin{equation}
    R^\alpha\| \Psi_R* (uv) - u(\Psi_R*v)\|_{\rho}
    \lesssim R^{\beta} \|v\|_{-\alpha,\rho} [u]_{\beta,\rho}.
  \end{equation}
  In particular, \eqref{e:besov-mult-rho} follows.
\end{proof}

Write $e^{\Delta t}f$ for $p_t*f$ where $p_t$ denotes the heat kernel on $\R^d$: 
\begin{equation} \label{e:ptdef-app}
  p_t(x) = (4\pi t)^{-d/2}e^{-|x|^2/4t}. 
\end{equation}

\begin{proposition} \label{prop:Besov-heat}
  For $\alpha,\beta \in (0,1)$ and $t\leq 1$,
  \begin{equation} \label{e:Besov-heat1}
    [e^{\Delta t} v]_{\beta,\rho} \leq C(\alpha,\beta,\rho) t^{-\frac{\alpha+\beta}{2}}\|v\|_{-\alpha,\rho}
  \end{equation}
  and
  \begin{equation} \label{e:Besov-heat2}
    \|e^{\Delta t} v\|_\rho \leq C(\alpha,\rho) t^{-\frac{\alpha}{2}} \|v\|_{-\alpha,\rho}.
  \end{equation}
\end{proposition}

\begin{proof}
  As preliminary warning, note that $p_t= (p_1)_{\sqrt{t}}$ with the notation $f_R(x) = R^{-d}f(x/R)$.
  We begin with the estimate \eqref{e:Besov-heat2} which follows immediately from Proposition~\ref{prop:norm-tilde}
  with $g=p_1$:
  \begin{equation}
    t^{\alpha/2}\|p_t*v\|_\rho \leq C(\alpha,\rho) \|v\|_{-\alpha,\rho}
  \end{equation}
  since $\|p_1\|_{\bbL^1(\rho^{-1})} \lesssim 1$ and 
  \begin{equation}
    \dsemnorm{p_1}_{\beta,\rho^{-1}} \lesssim
    \sup_{R\leq 1} \int \rho(x)^{-1}\int_{B_R(0)} \frac{|p_1(x+y)-p_1(x)|}{|y|} \, dy \, dx
    \lesssim  \int \rho(x)^{-1}e^{-c|x|} \, dx
    \lesssim 1.
  \end{equation}
  To show \eqref{e:Besov-heat1}, we summarise the following elementary estimates for the heat kernel
  (which can all be seen from  \eqref{e:ptdef-app}).
  For any $\beta \in [0,1]$ and $|x-y|\leq \frac12|x-z|$,
  \begin{equation}
    |p_t(x-z)-p_t(y-z)|
    \lesssim \min\ha{1,\frac{|x-y|}{\sqrt{t}}} \frac{e^{-c|x-z|/\sqrt{t}}}{t^{d/2}}
    \lesssim \frac{|x-y|^\beta}{t^{\beta/2}} \frac{e^{-c|x-z|/\sqrt{t}}}{t^{d/2}}.
  \end{equation}
  On the other hand, for $|x-y|\geq \frac12 |x-z|$,
  \begin{equation}
    |p_t(x-z)-p_t(y-z)| \leq p_t(x-z)+p_t(y-z) \lesssim \frac{e^{-c|x-z|/\sqrt{t}}+e^{-c|y-z|/\sqrt{t}}}{t^{d/2}}
  \end{equation}
  so that
  \begin{equation}
    |p_t(x-z)-p_t(y-z)| \lesssim \frac{|x-y|^\beta}{t^{\beta/2}}\qa{ (\frac{\sqrt{t}}{|x-z|})^\beta \frac{e^{-c|x-z|/\sqrt{t}}}{t^{d/2}}+(\frac{\sqrt{t}}{|y-z|})^\beta\frac{e^{-c|y-z|/\sqrt{t}}}{t^{d/2}}}
    .
  \end{equation}
  In summary,
  with $g(x)=  (1+|x|^{-\beta}) e^{-c|x|}$,
  \begin{equation}
    |p_t(x-z)-p_t(y-z)| \lesssim \frac{|x-y|^\beta}{t^{\beta/2}} \qa{g_{\sqrt{t}}(x-z)+g_{\sqrt{t}}(y-z)}
    .
  \end{equation}
  Thus
  \begin{equation}
    g^{x,y}(z) =\frac{p_1(x-z)-p_1(y-z)}{|x-y|^\beta}  
  \end{equation} 
  satisfies
  \begin{equation}
    |g^{x,y}(z)| \lesssim g(x-z)+g(y-z)
  \end{equation}
  and analogous estimates hold for derivatives:
  \begin{equation}
    |\nabla g^{x,y}(z)|
    \lesssim g(x-z)+g(y-z)
  \end{equation}
  In particular, for $\beta \in (0,1]$, uniformly in $|x-y|\leq 1$,
  \begin{align}
    \rho(x) \dsemnorm{g^{x,y}}_{\beta,\rho^{-1}}
    &\lesssim \sup_{R\leq 1} \int \rho(x)\rho(z)^{-1} \mint_{B_R(0)} \frac{|g^{x,y}(z+w)-g^{x,y}(z)|}{|w|^\beta}  \, dw \, dz
      \nnb
    &\leq C \sup_{R\leq 1} \int \rho(x-z)^{-1} \mint_{B_R(0)} \frac{|g^{x,y}(z+w)-g^{x,y}(z)|}{|w|^\beta} \, dw \, dz
      \nnb
    &\lesssim \int \rho(x-z)^{-1} g(x-z) \, dz \lesssim 1,
  \end{align}
  and
  \begin{align}
    \rho(x) \dsemnorm{g^{x,y}}_{\bbL^1(\rho^{-1})}
    &= \int \rho(x)\rho(z)^{-1} \int |g^{x,y}(z)| \, dz
     \nnb
    &\lesssim \int \rho(x-z)^{-1} g(x-z) \, dz \lesssim 1.
  \end{align}
  By Proposition~\ref{prop:norm-tilde} with $g=g^{x,y}$ and $\varphi=\Psi$, therefore
  \begin{align}
    t^{\beta/2}t^{\alpha/2}[e^{\Delta t} v]_{\beta,\rho}
	  &= t^{\alpha/2} \sup_{\substack{x\neq y \\ |x-y|\leq 1}} \rho(x) \absa{\int g_{\sqrt{t}}^{x/\sqrt{t},y/\sqrt{t}}(z)v(z) \, dz}
    \nnb
	  &= t^{\alpha/2} \sup_{\substack{x\neq y \\ |x-y|\leq 1}} \rho(x) \absa{g_{\sqrt{t}}^{x/\sqrt{t},y/\sqrt{t}}*v(0)}
    \nnb
	  &\lesssim \sup_{\substack{x\neq y \\ |x-y|\leq 1/\sqrt{t}}} \rho(x) \dnorm{g^{x,y}}_{\beta,\rho^{-1}} \sup_{r \leq \sqrt{t}} r^{\alpha} \|\Psi_{r} * v\|_\rho
    \lesssim  \|v\|_{-\alpha,\rho}
    ,
  \end{align}
  as claimed.
\end{proof}

The following proposition is a localised version of the Besov embedding of $B_{p,p}^{\alpha+d/p}$ into $B_{\infty,\infty}^\alpha$.
  
\begin{proposition} \label{prop:Besov-embedding}
Let $\alpha>0$ and $p\geq 1$.
There is then $C>0$ independent of $p$ such that:  
\begin{equation}
  R^{\alpha p}|\Psi_R * f(x)|^p
  \leq
    C^p\int_0^R  t^{p\alpha-d} \|\Psi_t*f\|_{\bbL^p(B_{3R}(x))}^p\, \frac{dt}{t}.
\end{equation}
In particular, for any weight $\rho$ as above and a different $C$,
\begin{equation}
  \|f\|_{-\alpha,\rho}^p
  \leq C^p \int_0^1 R^{p\alpha-d}\|\Psi_R*f\|_{\bbL^p(\rho)}^p  \, \frac{dR}{R}.
\end{equation}
\end{proposition}

\begin{proof}
  By H\"older's inequality,
  \begin{align}
    |f * \Psi_{R}(x)|
    \leq \|f\|_{\bbL^p(B_R(x))}\|\Psi_R\|_{\bbL^q}
    &= R^{-d/p} \|f\|_{\bbL^p(B_R(x))}\|\Psi\|_{\bbL^q(B_1(0))}
    \nnb
    &\leq R^{-d/p} \|f\|_{\bbL^p(B_R(x))}\|\Psi\|_{\bbL^\infty}
    .
 \end{align}
  Let $\beta>\alpha>0$ and define
  \begin{equation}
    \tilde \Psi(x)
    = \int_0^{1/2} t^{\beta-d}(\Psi*\Psi)(\frac{x}{t}) \, \frac{dt}{t}
    = \int_0^{1/2} t^{\beta} (\Psi_{t}*\Psi_{t})(x) \, \frac{dt}{t}.
  \end{equation}
  The second equality is $(\Psi*\Psi)_t(x) = (\Psi_t*\Psi_t)(x)$.
  Since $\beta>0$ and $\Psi$ is integrable and supported in $B_1(0)$,
  it follows that $\tilde \Psi$ is integrable and is supported in $B_1(0)$. 
  Moreover,
  \begin{align}
    \tilde \Psi_R(x)
    &= \int_0^{1/2} t^{\beta} (\Psi_t*\Psi_t)_R(x) \, \frac{dt}{t}
      \nnb
      &= \int_0^{1/2} t^{\beta} (\Psi_{tR}*\Psi_{tR})(x) \, \frac{dt}{t}
    = R^{-\beta}  \int_0^{R/2}t^{\beta} (\Psi_{t}*\Psi_{t})(x) \, \frac{dt}{t}.
  \end{align}
 The above and H\"older's inequality then give, for $p>1$ and $1/q=1-1/p$,
    \begin{align}
      R^{\alpha}|\tilde\Psi_R * f(x)|
      &\leq R^{\alpha-\beta} \int_0^{R/2}t^{\beta} |\Psi_t*\Psi_t*f(x)| \frac{dt}{t}
      \nnb
    &\leq  R^{\alpha-\beta} \|\Psi\|_{\bbL^\infty}\int_0^{R/2}t^{\beta}  t^{-d/p} \|\Psi_t*f\|_{\bbL^p(B_t(x))} \frac{dt}{t}
      \nnb
    &\leq R^{\alpha-\beta} \|\Psi\|_{\bbL^\infty}\pa{\int_0^{R/2}  t^{p\alpha-d} \|\Psi_t*f\|_{\bbL^p(B_t(x))}^p \frac{dt}{t}}^{1/p}\pa{\int_0^{R/2} t^{q(\beta-\alpha)} \frac{dt}{t}}^{1/q}
      \nnb
    &= \|\Psi\|_{\bbL^\infty}\pa{\int_0^{R/2}  t^{p\alpha-d} \|\Psi_t*f\|_{\bbL^p(B_t(x))}^p \frac{dt}{t}}^{1/p}\pa{\int_0^{1/2} t^{q(\beta-\alpha)} \frac{dt}{t}}^{1/q}
      \nnb
    &\lesssim \pa{\int_0^{R/2}  t^{p\alpha-d} \|\Psi_t*f\|_{\bbL^p(B_t(x))}^p \frac{dt}{t}}^{1/p}
    ,
    \end{align}
    where the proportionality constant does not depend on $p$. 
    This gives, for some $C>0$:
    \begin{equation}
      R^{\alpha p}|\tilde\Psi_R * f(x)|^p 
      \leq 
       C^p\int_0^{R/2}  t^{p\alpha-d} \|\Psi_t*f\|_{\bbL^p(B_t(x))}^p \frac{dt}{t}.
    \end{equation}
    Applying Proposition~\ref{prop:norm-tilde} with $\varphi = \tilde \Psi$ and $g=\Psi$, which are both supported in $B_1(0)$, gives, 
    for a different $C>0$ which is allowed to change:
    \begin{equation}
      R^{\alpha p}|\Psi_R * f(x)|^p \leq 
      C^p\sup_{r\leq R} r^{\alpha p}\|\tilde \Psi_r*f\|_{B_{2R}(x)}^p
      \leq 
      C^p\int_0^{R/2}  t^{p\alpha-d} \|\Psi_t*f\|_{\bbL^p(B_{3R}(x))}^p \frac{dt}{t}.
    \end{equation}
  This is the first claim and the second claim easily follows.
\end{proof}

\begin{proposition}[Compact embedding] \label{prop:arzela}
  Let $\alpha'>\alpha\geq 0$ and $\rho'(x)/\rho(x) \to 0$ as $|x|\to\infty$. Then the embedding $C^{-\alpha}(\rho) \subset C^{-\alpha'}(\rho')$
  is compact.
\end{proposition}

\begin{proof}
    Let $(u_j)_j \subset C^{-\alpha}(\rho)$ be a sequence in the unit ball of $C^{-\alpha}(\rho)$.
    It follows from the Banach--Alaoglu theorem that there is a weak-$*$ convergent subsequence, i.e., $(u_j,f) \to (u,f)$
    for every $f$ in $C_c^\infty(\R^d)$. 
    We will show that $u_j \to u$ in $C^{-\alpha'}(\rho')$ after possibly passing to another subsequence.
    
    For any $v \in C^{-\alpha}(\rho)$,
    define $\tilde v(s,x) = v*\Psi_{e^{-s}}(x)$ on $[0,\infty) \times \R^d$. Then
    $\|v\|_{-\alpha,\rho} = \|\tilde v\|_{\tilde \rho}$ where $\tilde \rho(s,x) = e^{-s\alpha}\rho(x)$
    and the norm on the right-hand side is the weighted $\bbL^\infty$ norm in the variables $(s,x) \in [0,\infty) \times \R^d$. Thus tildes denote joint functions of $x$ and the scale parameter $s$.
    The $(\tilde u_j)$ are in the unit ball of $C^0(\tilde \rho)$:
    \begin{equation}
      \|\tilde u_j\|_{\tilde \rho} = \|u_j\|_{-\alpha,\rho} \leq 1.
    \end{equation}
    Let $\tilde \rho'(s,x) = e^{-s\alpha'}\rho'(x)$ and $\tilde\rho^\delta(s,x) = e^{-s (\alpha+\delta)}\rho(x)$
    with $\alpha < \alpha+\delta<\alpha'$.
    Then $\tilde \rho'(s,x)/\tilde \rho^\delta(s,x) \to 0$ as $|x|+s \to \infty$
    and the embedding $C^{\delta}(\tilde \rho^\delta) \subset C^0(\tilde \rho')$ is compact by the Arzela--Ascoli theorem.
    It thus suffices to show that
    \begin{equation} \label{e:tildeuj}
      \|\tilde u_j\|_{\delta,\tilde \rho^\delta} \lesssim 1,
    \end{equation}
    where the left-hand side is a weighted H\"older norm on $[0,\infty)\times \R^d$.
    Assuming~\eqref{e:tildeuj}, 
    the Arzela--Ascoli theorem implies that   $\|\tilde u_j-\tilde u\|_{\tilde \rho'} \to 0$ along some subsequence still denoted by $(u_j)_j$:
    \begin{equation}
      \sup_{R\leq 1} \|u_j*\Psi_R - \tilde u(R,\cdot)\|_{\rho'} R^{\alpha'} \to 0.
    \end{equation}
    By the weak-$*$ convergence $u_j \to u$ and uniqueness of limits it holds that $\tilde u(R,\cdot) = u*\Psi_R$ for each $R$, 
    hence $u_j \to u \in C^{-\alpha'}(\rho')$ as claimed.
    
    To show \eqref{e:tildeuj}, let $h\in(0,1)$, $z\in B_1(0)\setminus\{0\}$, $x\in \R^d$, and define
    for $y\in \R^d$:
    \begin{equation}
      g^{x,h}(y) = \frac{\Psi_{e^{-h}}(x+y)-\Psi_{1}(x+y)}{h^\delta},
      \qquad
      g^{x,z}(y) =  \frac{\Psi_{1}(x+y+z)-\Psi_{1}(x+y)}{|z|^\delta}.
    \end{equation}
    Since $\Psi$ is smooth and compactly supported, for any $\delta,\beta\in(0,1)$, uniformly in $h,z,x$,
    \begin{equation}
      \rho(x) \dnorm{g^{x,h}}_{\beta,\rho^{-1}} \lesssim 1,
      \qquad
      \rho(x) \dnorm{g^{x,z}}_{\beta,\rho^{-1}} \lesssim 1.
    \end{equation}
    Therefore Proposition~\ref{prop:norm-tilde} implies, taking $\beta>\alpha$
    and using that $\rho_R(x)=\rho(Rx)$ by definition:
    \begin{align}
      \frac{\tilde u(s,x+z)-\tilde u(s,x)}{|z|^\delta}
      &= e^{\delta s}(g_{e^{-s}}^{e^{s}x,e^{s}z}*u)(0)
        \nnb
      &    \lesssim e^{(\alpha+\delta) s} \dnorm{g^{e^sx,e^sz}}_{\beta,\rho_{e^{-s}}^{-1}} \|u\|_{-\alpha,\rho}        \nnb
      &    \lesssim e^{(\alpha+\delta) s} \rho_{e^{-s}}(e^sx)^{-1} \|u\|_{-\alpha,\rho}
        =
        \tilde\rho^\delta(s,x)^{-1}
        \|u\|_{-\alpha,\rho}
      \\
            \frac{\tilde u(s+h,x)-\tilde u(s,x)}{h^\delta}
      &= (g^{e^{-s}x,h}_{e^{-s}} * u)(0)
        \nnb
        &\lesssim e^{\alpha s} \dnorm{g^{e^{s}x,h}}_{\beta,\rho_{e^{-s}}^{-1}} \|u\|_{-\alpha,\rho}
          \lesssim \tilde\rho(s,x)^{-1} \|u\|_{-\alpha,\rho}.
    \end{align}
    In detail, the first equality above is
    \begin{align}
      \frac{\tilde u(s,x+z)-\tilde u(s,x)}{|z|^\delta}
      &= \frac{1}{|z|^\delta} \Big( \Psi_{e^{-s}} \ast u (x+z) - \Psi_{e^{-s}} \ast u(x) \Big)
        \nnb
      &= \frac{1}{|z|^\delta}  \int \Big[ \Psi ( e^{s}(x+z-y) ) - \Psi ( e^{s}(x-y))  \Big]    u (y)  \, e^{sd}\, dy
        \nnb
      &= e^{s\delta}\frac{1}{(e^{s}|z|)^\delta}  \int \Big[ \Psi ( e^{s}(x+z-y) ) - \Psi ( e^{s}(x-y))  \Big] \, u (y)   \,  e^{sd} \, dy
        \nnb
      &= e^{s\delta} \int g^{e^{s}x,e^{s}z}(-e^{s}y)   u (y)\, e^{sd}  \, dy
      = e^{s\delta} (g^{e^{s}x,e^{s}z}_{e^{-s}}*   u) (0),
    \end{align}
    and the second equality is
    \begin{align}
      \frac{\tilde u(s+h,x)-\tilde u(s,x)}{h^\delta}
      &= \frac{1}{h^\delta} \Big( \Psi_{e^{-(s+h)}} \ast u (x) - \Psi_{e^{-s}} \ast u(x) \Big)
        \nnb
      &= \frac{1}{h^\delta} \int  \pB{ \Psi_{e^{-h}}(e^{s}x-e^{s}y) - \Psi_{1}(e^{s}x-e^{s}y) } u(y) \, e^{sd}\,  dy
        \nnb
      &= \int  g^{e^{s}x,h}(-e^{s}y) u(y) \, e^{sd}\,dy
      =(g^{e^{s}x,h}_{e^{-s}} * u)(0) .
    \end{align}
    This shows $\|\tilde u\|_{\delta,\tilde\rho^\delta} \lesssim 1$. 
\end{proof}

\begin{proposition}[Continuity theorem]\label{prop:kolmogorov}
  Let $n\geq 1$ and consider a random distribution-valued process $(X_t)_{t\geq 0}$. 
  Assume that there are $\epsilon,C>0$ and, for each $\lambda\in(0,1)$, 
  constants $\epsilon_\lambda,C_\lambda>0$ such that, 
  for any $t,s\geq 0$, $x\in\R^2$ and $R\in(0,1]$:
  \begin{align}
    \E\qB{\exp\Big[\epsilon |\log R|^{-1}|(\Psi_R\ast X_t)(x)|^{2/n}\Big]} &\leq C
    \label{eq_assump_exp_moment_Kolmogorov_line1}\\
\E\qbb{\exp\bigg[\epsilon_\lambda |\log R|^{-1}R^{2\lambda/n}\frac{|(\Psi_R \ast X_t)(x) - (\Psi_R \ast X_s)(x)|^{2/n}}{|s-t|^{\lambda/n}}\bigg]}
&\leq  C_\lambda
,
    \label{eq_assump_exp_moment_Kolmogorov}
  \end{align}
  Then, for any $\alpha,\sigma >0$ and $\rho(x)=(1+|x|^2)^{-\frac{\sigma}{2}}$,
  there is a modification of $X$ with values in $C(\R_+,C^{-\alpha}(\rho))$ and $\epsilon',\kappa>0$ independent of $\alpha$ such that:
  \begin{equation}
  \E\bigg[ \exp\bigg[\epsilon'\sup_{s\in[0,t]}\|X_s\|^{2/n}_{-\alpha,\rho}\bigg]\bigg]
  \lesssim 
(1+t)^{\kappa}
  .
  \label{eq_statement_Kolmogorov}
  \end{equation}
\end{proposition}

\begin{proof}
To bound the exponential moment, 
we bound all moments using the expansion:
\begin{equation}
\E[e^{\epsilon Y}]
=
1+\sum_{p\geq 1}\frac{\epsilon^p}{p!}\E[Y^p]
,\qquad 
Y\geq 0
.
\end{equation}
We will prove the existence of $c>0$ such that, for each large enough $p$ and $t\geq 0$:
\begin{equation}
 \E\Big[ \sup_{s\in[0,t]}\|X_s\|^{p}_{-\alpha,\rho}\Big]
 \leq 
 c^{pn/2} p^{pn/2} \log(1+t)^{pn/2}\log(1+t)
 .
 \label{eq_to_prove_estimate_moment}
\end{equation}
Together with the elementary identity $p^p/p!\geq e^p$ ($p\geq 1$) and the observation $\E[Y]\leq \E[Y^{n/2}]^{2/n}$ if $n\geq 2$ and $Y\geq 0$, 
this will imply the claim.

By Proposition~\ref{prop:Besov-embedding}, 
there is $C>0$ such that, for each $p\geq 1$:
\begin{align}
  \E\Big[\sup_{s\in[0,t]} \|X_s\|_{-\alpha,\rho}^p \Big]
  &\leq 
  C^p\, \E\bigg[\sup_{s\in[0,t]} \int dx \, \rho(x)^p \int_0^1 R^{\alpha p-d}|(\Psi_R*X_s)(x)|^p\, \frac{dR}{R}\bigg]
    \nnb
  &\leq
  C^p\int dx \, \rho(x)^p \int_0^1 R^{\alpha p-d}\, \E\Big[\sup_{s\in[0,t]} |(\Psi_R*X_s)(x)|^p\Big]\frac{dR}{R}.
\end{align}
It is thus enough to prove~\eqref{eq_to_prove_estimate_moment} for each $p$ large enough for $\rho^p$ to be integrable, 
with the left-hand side there replaced with: 
\begin{equation}
\sup_{x\in\R^2}\int_0^1 R^{\alpha p-d}\, \E\Big[\sup_{s\in[0,t]}|(\Psi_R*X_s)(x)|^p\Big]\, \frac{dR}{R}
.
\end{equation}
Such a $p\geq 2$ is henceforth fixed. 
Let $t\geq 0$, $x\in\R^2$ and $R\in(0,1]$. 
Split $[0,t]$ into intervals of length at most $2R^{-2}$ as follows:
\begin{equation}
[0,t]
=
\bigcup_{k=0}^{\lfloor R^{-2}t\rfloor-2}[kR^2,(k+1)R^2]
\cup 
\big[(\lfloor R^{-2}t\rfloor-1)R^2,t\big]
.
\end{equation}
Write $I_R := \{kR^2:k\leq \lfloor R^{-2}t\rfloor-1\}$. 
Then:
\begin{align}
\E\Big[\sup_{s\in[0,t]}|(\Psi_R\ast X_s)(x)|^p\Big]
&\leq
2^{p-1}\E\Big[\sup_{u\in I_R}\sup_{s\in[u,u+2R^2]}\big|(\Psi_R\ast X_s)(x)-(\Psi_R*X_u)(x)\big|^p\Big]
\nnb
&\qquad+
2^{p-1}\E\Big[\sup_{u\in I_R}|(\Psi_R\ast X_u)(x)|^p\Big]
\label{eq_splitting_Kolmogorov}
.
\end{align}
In the rest of the proof, $\lesssim$ symbols hide multiplicative factors of the form $c^p$ for some $c>0$ that may vary from line to line.

We first consider the first term on the right-hand side (involving the difference).
Recall the Garsia--Rodemich--Rumsey inequality~\cite[Theorem~2.1.3]{StroockVaradhan_book}: 
if $f$ is a continuous function and $\gamma,\psi$ are strictly increasing functions vanishing at $0$ with $\lim_{s\to\infty}\psi(s)=+\infty$, 
then:
\begin{equation}
\sup_{\substack{u,s\in[0,t]\\ |u-s|\leq 2R^2}}
|f(s)-f(u)|
\leq 
8 \int_0^{2R^2}\psi^{-1}(4B/r^2)\, \gamma(dr)
,
\end{equation}
where
\begin{equation}
B:=
\int_0^t\int_0^t \psi\Big(\frac{|f(s)-f(u)|}{\gamma(|s-u|)}\Big)\, ds\, du
.
\end{equation}
Define, 
for $\lambda\in(0,1)$ to be chosen later as a function of $n,p$ and $\epsilon_\lambda>0$ as in~\eqref{eq_assump_exp_moment_Kolmogorov}:
\begin{equation}
\delta_R
:=
\epsilon_\lambda |\log R|^{-1} R^{2\lambda/n}
.
\end{equation}
Taking
\begin{equation}
  f(s)= (\psi_R\ast X_s)(x), \qquad \gamma(s)=\sqrt{s}, \qquad \psi(x)=e^{\delta_R x^{2\lambda/n}}-1
\end{equation}
so that $\psi^{-1}(x)=\delta_R^{-n/2\lambda}\log(1+x)^{n/2\lambda}$ yields: 
\begin{align}
&\E\Big[\sup_{u\in I_R}\sup_{s\in[u,u+2R^2]}\big|(\Psi_R\ast X_s)(x)-(\Psi_R\ast X_u)(x)\big|^p\Big]
\nnb
&\quad 
\lesssim 
\E\Bigg[\bigg[\delta_R^{-n/2\lambda}\int_0^{2R^2}\frac{dr}{\sqrt{r}}\log^{n/2\lambda}\bigg(1 + \frac{4B}{r^2}\bigg)\bigg]^p\Bigg]
\nnb
&\quad 
\lesssim
\delta_R^{-np/2\lambda}R^{p-1}
\\
&\qquad  \times \int_0^{2R^2}\frac{dr}{\sqrt{r}}\, \E\Bigg[\log^{np/2\lambda}\bigg(1 + \frac{4}{r^2}\int_0^t\int_0^t \exp\Big[\frac{\delta_R|(\Psi_R\ast X_s)(x)- (\Psi_R\ast X_t)(x)|^{2\lambda/n}}{|t-s|^{\lambda/n}}\Big]\, ds\, dt\bigg)\Bigg]
\nonumber
,
\end{align}
where in the last line we used the convexity of $x\geq 0\mapsto x^p$ on the integral on $r$, 
noting that $\int_0^{2R^2}r^{-1/2}\, dr=2\sqrt{2}R$. 

For $p\geq 2$, $\lambda<1$ it holds that $1\leq e^{np/2\lambda-1}$. 
Using this fact, the concavity of $x\mapsto \log(e^{np/2\lambda-1}+x)^{np/2\lambda}$ and 
Jensen's inequality applied to the expectation gives:
\begin{align}
&\E\Big[\sup_{u\in I_R}\sup_{s\in[u,u+2R^2]}\big|(\Psi_R*X_s)(x)-(\Psi_R*X_u)(x)\big|^p\Big]
\nnb
&
\lesssim |\log R|^{np/2\lambda}R^{-1}
\nnb
&\quad
\times \int_0^{2R^2}\frac{dr}{\sqrt{r}} \log^{np/2\lambda}\Bigg(e^{np/2\lambda-1} + \frac{4}{r^2}\int_0^t\int_0^t \E\bigg[\exp\Big[\frac{\delta_R|(\Psi_R*X_s)(x)- (\Psi_R*X_t)(x)|^{2\lambda/n}}{|t-s|^{\lambda/n}}\Big]\bigg] ds\, dt\Bigg)
\nnb
&
\hspace{3cm}\lesssim
|\log R|^{np/2\lambda}R^{- 1} \Big(\frac{np}{2\lambda}-1\Big)^{np/2\lambda}\int_0^{2R^2}\frac{dr}{\sqrt{r} } \Big[ 1+ \log^{np/2\lambda}\Big(1 + \frac{t^2}{r^2}\Big)\Big]
  \nnb
  &\hspace{3cm}\lesssim 
|\log R|^{np/\lambda}\Big(\frac{np}{2\lambda}-1\Big)^{np/2\lambda}\log(1+t)^{np/2\lambda}
.
\label{eq_to_bound_expectation_sup_interm}
\end{align}
The last quantity does not depend on $x$. 
For it to grow at most as $p^{np/2}$ when $p$ is large, 
we take $\lambda$ satisfying:
\begin{equation}
\Big(\frac{np}{2\lambda}-1\Big)^{np/2\lambda}
=
\Big(\frac{np}{2\lambda}\Big)^{np/2}
\quad
\Rightarrow\quad
\lambda = 1-\frac{2(1+o_{p\to\infty}(1))}{np\log(np/2)}
.
\end{equation}
For such a $\lambda$ we can in particular write that, for all $p$ large enough as a function of $n$ only:
\begin{equation}
\log(1+t)^{np/2\lambda}
\leq 
\log(1+t) \log(1+t)^{np/2}
.
\end{equation}
Integrating the right-hand side of~\eqref{eq_to_bound_expectation_sup_interm} in $R$ against $R^{\alpha p -d-1}\, dR$ therefore yields a bound of the form~\eqref{eq_to_prove_estimate_moment} on the difference term in the right-hand side of~\eqref{eq_splitting_Kolmogorov},  
whenever $p$ is large enough depending only on $\alpha,n,\rho$. 

Consider now the second term on the right-hand side of~\eqref{eq_splitting_Kolmogorov}. 
Using $\E[Y]=\int_0^\infty\P(Y>a)\, da$ for a nonnegative random variable $Y$, 
it reads, for $A=A(R,p,t)>0$ to be chosen later:
\begin{align}
  \E\Big[\sup_{u\in I_R}|(\Psi_R*X_u)(x)|^p\Big]
&\leq 
A
+\int_{A}^\infty 
\P\Big(\sup_{u\in I_R}|(\Psi_R*X_u)(x)|^p>a\Big)\, da
\nnb
&\leq 
A
+\lfloor R^{-2} t\rfloor \sup_{u\in [0,t]}\int_{A}^\infty 
\P\Big(|(\Psi_R*X_u)(x)|^p>a\Big)\, da
.
\end{align}
The exponential moment bound~\eqref{eq_assump_exp_moment_Kolmogorov_line1} allows one to bound the integral term by:
\begin{align} 
C\lfloor R^{-2} t\rfloor \int_{A}^\infty 
\exp\Big[ -\epsilon a^{\frac{2}{np}}/|\log R|\Big]\, da
&
\lesssim 
\lfloor R^{-2}t\rfloor |\log R|^{np/2}\int_{A/|\log R|^{np/2}}^{\infty} e^{-\epsilon b^{\frac{2}{np}}}\, db
.
\end{align}
An integration by parts gives, for any $\alpha<1$ and $B>0$ with $\alpha \epsilon B^\alpha\geq 2$:
\begin{equation}
\int_{B}^\infty e^{-\epsilon b^{\alpha}}\, db
\leq 
\frac{2B^{1-\alpha}}{\epsilon\alpha }e^{-\epsilon B^\alpha}
.
\end{equation}
In our case $\alpha=(2/np)$ satisfies $\alpha>1$ as soon as $p>2$, 
while 
$B=A/|\log R|^{np/2}$. 
Choosing $A=\epsilon^{-np/2}|\log R|^{np/2}\log(\lfloor R^{-2}t\rfloor)^{np/2}$ concludes the proof:
\begin{equation}
\E\Big[\sup_{u\in I_R}|(\Psi_R*X_u)(x)|^p\Big]
\lesssim 
|\log R|^{np}\log(1+t)^{np/2}
.
\end{equation}
\end{proof}

\section{Proof of Gaussian estimates}
\label{app_gaussian_estimates}

In this appendix we include the proofs of the estimates on the Ornstein--Uhlenbeck process and
its Wick powers that we need.
Throughout this section we work in dimension $d=2$. 

\subsection{Standard Wick powers}
First recall that the Wick powers with $t$- and $L$-independent counterterms for the centred field $\tilde Z$ were defined in~\eqref{e:Wick-convention} for $n\in\{2,3\}$. 
More generally, 
Wick powers of $\tilde Z_t$ are given by:
\begin{equation}\label{e:Wick-convention-app}
\wick{\tilde Z^n_t} 
=
\lim_{\epsilon\to 0}
P_n(a_\epsilon,\eta_\epsilon* \tilde Z_t)
,\qquad 
 a_\epsilon
 =
 \frac{1}{2\pi}\log\frac{1}{\epsilon}
 ,
\end{equation}
and the same definition applies to the non-centred field $Z_t$. 
Above, 
\begin{equation}
  P_n(a,X) = e^{-\frac{a}{2} \partial_X^2} X^n.
\end{equation}
For the analysis of these Wick powers, 
it is convenient to first define the standard (homogeneous) Wick powers (which correspond to $t$- and $L$-dependent counterterms) 
as elements of the $n$-th homogeneous Wiener chaos.
This can for example be done using iterated stochastic integrals (as in~\cite[Section~5]{MR3693966}).
To distinguish the Wick powers \eqref{e:Wick-convention-app} from the standard 
Wick powers in the homogeneous Wiener chaos, the latter will be denoted by $\wwick{\tilde Z^n}$.
Using that the covariance of $Z_t$ is given by \eqref{e:OU-cov},
the homogeneous Wick powers for the centred field $\tilde Z_t$ satisfy, for each $f\in\bbL^2(\R^2)$:
\begin{equation}  \label{e:wwick-var}
  \E\qa{(\wwick{(\tilde Z_t)^n},f)^2} = n! \int_{(\R^d)^2} f(x) f(y) \Big(\int_{0}^{2t} e^{-u} p_u(x-y) \, du\Big)^n \, dx\, dy.
\end{equation}
The definition of the Wick power $\wwick{(\tilde Z_t^L)^n}$ is the same 
with $p_t$ replaced by $p^L_t$, 
which is viewed as a function on $\R^d$ through the formula:
\begin{equation}
 p_t^L(x) = {\bf 1}_{x\in[-L,L)^2}\sum_{a\in\Z^2}p_t(x-2aL)
 .
 \label{eq_OU_kernel_periodic_app}
\end{equation}
It is well known that  
$\wwick{(\tilde Z_t^L)^n}$ can be obtained as in \eqref{e:Wick-convention-app} except
that $a_\epsilon$ must be replaced by
\begin{equation}
  a_\epsilon(t,L)= \var((\eta_\epsilon * \tilde Z_t^L)(0)) = \int_0^t e^{-s} \eta_\epsilon(x)\eta_\epsilon(y) p_s^L(x-y)\, dx\, dy \,ds 
\end{equation}
see, for example, \cite[Theorem V.3.]{MR0489552} for a variant of this statement.
Thus
the $t$- and $L$-independent Wick powers $\wick{\tilde Z^n}$ 
are obtained from the homogeneous ones by
\begin{align}
  \wick{\tilde Z^n_t}
  &=  \wwick{\tilde Z^n_t} + \lim_{\epsilon\to 0} \qB{ P_n(a_\epsilon(t,L)+(a_\epsilon - a_\epsilon(t,L)), \eta_\epsilon*Z_t)- P_n(a_\epsilon(t,L), \eta_\epsilon*Z_t)}
    \nnb
  &=  \wwick{\tilde Z^n_t} + \sum_{m=0}^{n-2} Q_{n,m}(f(t,L)) \wwick{\tilde Z^m_t}
  \label{eq_wick_as_fct_wwick}
\end{align}
where
\begin{equation}
  f(t,L)
  =\lim_{\epsilon\to 0} \qB{ a_\epsilon - a_\epsilon(t,L)}
\end{equation}
and  we used that, for some polynomials $Q_{n,m}$,
\begin{align}
  P_n(a+b,X)-P_n(a,X) = (e^{-\frac{b}{2} \partial_X^2}-1) e^{-\frac{a}{2}\partial_X^2} X^n
  &= (e^{-\frac{b}{2} \partial_X^2}-1) P_n(a,X)
    \nnb
    &= \sum_{m=0}^{n-2} Q_{n,m}(b) P_m(a,X)
\end{align}
In particular,
\begin{equation}
  \wick{\tilde Z_t^2} = \wwick{\tilde Z_t^2}+f(t,L), \qquad
  \wick{\tilde Z_t^3} = \wwick{\tilde Z_t^3}+3f(t,L)\tilde Z_t.
\end{equation}

\begin{lemma}\label{lemma_size_ftL}
  The limit $f(t,L)$ exists and
  \begin{equation}
\sup_{L\geq 1}f(t,L)
\lesssim
1
,\qquad
t\geq 1
,\qquad
f(t,L)\underset{t\to 0}\sim \frac{1}{4\pi}\log(t)
,
\label{eq_size_ftL}
\end{equation}
where the behaviour when $t\to0$ is valid uniformly in $L \geq 1$.   
\end{lemma}

\begin{proof}
For some $C>0$ independent of $t,L$, 
\begin{align}
  -f(t,L)
  &= \lim_{\epsilon\to 0} \qa{ \int_0^t  e^{-s}\iint  \eta_\epsilon(x)\eta_\epsilon(y)p_{s}^L(x-y) \, dx\, dy\, ds - \frac{1}{2\pi} \log \frac{1}{\epsilon}}
  \nnb
  &= -\int_0^\infty e^{-s} \pa{ p_s(0) - p_s^L(0)1_{s\leq t}} \, ds +C 
  \nnb
  &=  \sum_{n\in \Z^2: n \neq 0}  \int_0^t \frac{e^{-|2Ln|^2/4s}e^{-s}}{4\pi s} \, ds - \int_t^\infty \frac{e^{-s}}{4\pi s} \, ds +C.
  \label{eq_def_diff_counterterms}
\end{align}
In the second line, we used that
\begin{equation}
  \int_{\epsilon^2}^\infty e^{-s}p_s(0) \, ds = \int_{\epsilon^2}^\infty e^{-s} \frac{ds}{4\pi s} =  \frac{1}{2\pi} \log \frac{1}{\epsilon} - \frac{\gamma}{\pi}
\end{equation}
where $\gamma$ is the Euler--Mascheroni constant, and hence the first line is
\begin{equation}
  \lim_{\epsilon \to 0}\int_0^\infty e^{-s} \iint \eta_\epsilon(x)\eta_\epsilon(y) \pa{p_s^L(x-y)1_{s\leq t}-p_s(0)1_{s\geq \epsilon^2}} \, ds  - \frac{\gamma}{\pi}
  .
\end{equation}
For any $\delta>0$ dominated convergence implies that
\begin{align}
  &\lim_{\epsilon \to 0}\int_\delta^\infty e^{-s} \iint \eta_\epsilon(x)\eta_\epsilon(y) \pa{p_s^L(x-y)1_{s\leq t}-p_s(0)1_{s\geq \epsilon^2}} \, ds
    \nnb
  &= -\int_\delta^\infty e^{-s} \pa{ p_s(0) - p_s^L(0)1_{s\leq t}} \, ds .
\end{align}
On other hand, for the $0<s<\delta$ contribution, the difference between $p_s$ and $p_s^L$ can be neglected as $\delta\to 0$, and
\begin{align}
  &\int_0^\delta e^{-s} \iint  \eta_\epsilon(x)\eta_\epsilon(y) (p_s(x-y)-p_s(0)1_{s\geq \epsilon^2})\,dx\,dy\, ds
    \nnb
  &=\int_0^\delta e^{-s}  \iint \eta(x)\eta(y) (p_s(\epsilon(x-y))-p_s(0)1_{s\geq \epsilon^2})\,dx\,dy\, ds
    \nnb
  &=\int_0^\delta  e^{-s} \iint  \eta(x)\eta(y) (e^{-\epsilon^2 |x-y|^2/4s}-1_{s\geq \epsilon^2})\,dx\,dy\, \frac{ds}{4\pi s} 
    \nnb
  &=\int_0^{\delta/\epsilon^2}e^{-s/\epsilon^2}\iint  \eta(x)\eta(y) (e^{-|x-y|^2/4s}-1_{s\geq 1})\,dx\,dy\, \frac{ds}{4\pi s}
.
\end{align}
Using that
\begin{equation}
  \iint \eta(x)\eta(y) |e^{-|x-y|^2/4s}-1_{s\geq 1}|\,dx\,dy
  =O(s)1_{s\leq 1} + O(\frac{1}{s})1_{s\geq 1}
\end{equation}
is integrable with respect to $ds/s$ it follows from dominated convergence that the following limit exists (and is finite):
\begin{align}
  &\lim_{\delta\to 0}\lim_{\epsilon\to 0}\int_0^\delta e^{-s} \eta_\epsilon(x)\eta_\epsilon(y) (p_s(x-y)-p_s(0)1_{s\geq \epsilon^2})\,dx\,dy\, ds
    \nnb
  &=\int_0^{\infty} \eta(x)\eta(y) (e^{-|x-y|^2/4s}-1_{s\geq 1})\,dx\,dy\, \frac{ds}{4\pi s}  
=O(1),
\end{align}
completing the estimate. 
\end{proof}

\subsection{Bounds on the field}
The following estimates are essentially contained in the proof of Theorem 5.1 in\cite{MR3693966}.
\begin{proposition}\label{prop_gaussian_bounds0IC}
Let $n\geq 1$ be an integer and let $\alpha,\sigma>0$. 
There is $c>0$ such that,
for each $p\geq 1$, 
\begin{align}
\sup_{L\in[3,\infty]} \E\Big[\sup_{s\in[0,t]} \| \wwick{(\tilde Z^L_s)^n}\|^p_{-\alpha,\rho}\Big]
&\leq
\Big(c p\, \log(1+ t)\Big)^{np/2}
.
\label{eq_gaussian_moment_bound_0IC}
\end{align}
As a result, 
there is $\kappa>0$ such that, 
for any $r\in[0,2]$ and some $\epsilon_r>0$:
\begin{equation}
\sup_{L\in[3,\infty]}\E\bigg[\exp\Big[\epsilon_r\sup_{s\in[0,t]}\|\wwick{(\tilde{Z}^L_s)^n}\|^{r/n}_{-\alpha,\rho}\Big]\bigg]
\lesssim
(1+t)^{\kappa}
.
\label{eq_gaussian_exp_moment_bound_0IC}
\end{equation}
\end{proposition}
\begin{proof}
Let us first prove~\eqref{eq_gaussian_exp_moment_bound_0IC} assuming the moment bounds~\eqref{eq_gaussian_moment_bound_0IC}. 
Let $r\in[0,2]$. 
Recall the elementary identity $p^p/p!\geq e^p$ ($p\geq 1$) and the expansion:
\begin{equation}
\E[e^{\epsilon X^{r/n}}]
=
1+\sum_{p\geq 1}\frac{\epsilon^p}{p!}\E[X^{rp/n}]
,\qquad 
X\geq 0,\epsilon>0
.
\label{eq_expansion}
\end{equation}
Taking $X=\sup_{s\in[0,t]}\|\wwick{(\tilde{Z}^L_s)^n}\|_{-\alpha,\rho}$ and using $\E[X^{rp/n}]\leq 1+\E[X^p]^{r/n}$ then yields~\eqref{eq_gaussian_exp_moment_bound_0IC}. 

To prove~\eqref{eq_gaussian_moment_bound_0IC}, 
we check that the assumptions of Proposition~\ref{prop:kolmogorov} are satisfied. 
Computations closely follow those of~\cite[Theorem 5.1]{MR3693966}. 
Let $R\in(0,1)$, $t_1,t_2\geq 0$ with $|t_1-t_2|\leq 1$ and $x_1,x_2\in\R^2$ with $|x_1-x_2|\leq 1$. 
Write for short $\Psi^x_R:= \Psi_R(x-\cdot)$. 
From~\eqref{eq_expansion}, 
we see that it is enough to prove the following. 
For each $\lambda\in(0,1)$, 
there are constants $c,c_\lambda>0$ such that, 
uniformly on $t\geq0$, $x\in\R^2$, $L\geq 3$ and $p\geq 1$:
\begin{align}
\E\Big[\, \big| \big(\wwick{(\tilde Z_t^L)^n},\Psi^x_R\big)\big|^{2p/n}\, \Big]
&\leq
c^p (p-1)^{p}|\log R|^{p}
,
\label{eq_gaussian_bound_without_sup}
\\
\E\Big[ \, \Big|\big(\wwick{(\tilde Z_{t_1}^L)^n},\Psi^{x}_R\big)-\big(\wwick{(\tilde Z_{t_2}^L)^n},\Psi^{x}_R\big)\Big|^{2p/n}\, \Big]
&\leq 
c_\lambda^p (p-1)^{p}R^{-\lambda p}|t_1-t_2|^{\lambda p/2}|\log R|^{p}
.
\label{eq_gaussian_time_diff_bound_without_sup}
\end{align}
By hypercontractivity (see e.g.~\cite[(1.71)]{Nualart_book}) and since $\E[X]\leq \E[X^q]^{1/q}$ for any $q\geq 1$ and $X\geq 0$, 
it is enough to establish the above bounds when $p=n$. 
For $x,y\in\R^2$, define:
\begin{equation}
|x|_L
:=
\begin{cases}
\inf\{|x+y|:y\in 2L\Z^2\} &\text{if } L<\infty,\\
|x|&\text{if } L=\infty.\\
\end{cases}
\end{equation}
Consider first~\eqref{eq_gaussian_bound_without_sup}. 
It reads:
\begin{equation}
\E\big[ \big(\wwick{(\tilde Z_t^L)^n},\Psi^x_R\big)^2\big]
=
n!\int_{(\R^2)^2} \Psi_R^x(y)\Psi_R^x(z)\,\big(\mathcal K^L(t,t,y-z)\big)^n dy\, dz
,
\label{eq_gaussian_bound_without_sup_for_square}
\end{equation}
with (recall that $p^L$ is defined in~\eqref{eq_OU_kernel}--\eqref{eq_OU_kernel_periodic}):
\begin{equation}
\mathcal K^L(t_1,t_2,x)
=
\int_{|t_1-t_2|}^{t_1+t_2} e^{-u}p^L_u(x)\, du
.
\label{eq_def_kernelK}
\end{equation}
Standard estimates on the heat kernel give:
\begin{equation}
\mathcal K^L(t,t,x)
\lesssim 
1+\log_+(1/|x|_L)
,
\end{equation}
with $\log_+(x):= \max \{\log(x),0\}$. 
Indeed, 
\begin{align}
8\pi\mathcal K^L(t,t,x)
&\leq
\int_0^{|x_L|^2} e^{-\frac{|x|_L^2}{4u}}\, \frac{du}{u} 
+
\int_{|x_L|^2}^{|x_L|^2\vee 2t}e^{-u} e^{-\frac{|x|_L^2}{4u}} \, \frac{du}{u} 
\nnb
&\leq
\int_0^1 e^{-\frac{1}{4u}}\,\frac{du}{u} + e^{-1/4}\int_{|x_L|^2}^\infty e^{-u}\, \frac{du}{u}
\nnb
&\lesssim 
1+\log_+(1/|x_L|)
.
\end{align}
As a result,~\eqref{eq_gaussian_bound_without_sup_for_square} is up to a multiplicative constant bounded by:
\begin{align}
&\int_{(\R^2)^2}\Psi^x_R(y)\Psi^x_R(z)\Big[1+\log_+\Big(\frac{1}{|y-z|_L}\Big)^n\Big]\, dy\, dz
\nnb
&\qquad \lesssim
1+ \int_{(\R^2)^2}\Psi^x_R(y)\Psi^x_R(z) \log_+\Big(\frac{1}{|y-z|_L}\Big)^n\, dy\, dz
.
\end{align}
Recall that $\Psi_R=R^{2}\Psi(R^{-1}\cdot)$ has compact support in $B_R(0)$, 
so that the integral vanishes unless $|y-z|\leq 2R$. 
In particular on the domain of integration one must have $|y-z|_L=|y-z|$ as $L\geq 3$. 
Changing variables, we find:
\begin{align}
\E&\big[ \big(\wwick{(\tilde Z_t^L)^n},\Psi^x_R\big)^2\big]
\lesssim 
1+\int_{\R^2}\, dy\, \Psi\Big(\frac{x}{R}-y\Big) \int_{\R^2}\Psi\Big(\frac{x}{R}-z\Big)\log_+\Big(\frac{1}{R|y-z|}\Big)^n\, dz
\nnb
&\quad\leq
|\log R|^n+\|\Psi\|_\infty\int_{\R^2}\, dy\, \Psi\Big(\frac{x}{R}-y\Big) \int_{B(y,2\wedge 1/R)}\log\Big(\frac{1}{|y-z|}\Big)^n\, dz
\nnb
&\quad\lesssim
|\log R|^n+\int_{B(0,2\wedge 1/R)}\log\Big(\frac{1}{|z|}\Big)^n\, dz
\nnb
&\quad\lesssim
|\log R|^n
.
\label{eq_bound_ZPhi_T}
\end{align}
Consider next~\eqref{eq_gaussian_time_diff_bound_without_sup}. 
A direct computation gives (recall that $\cK^L$ is defined in~\eqref{eq_def_kernelK}):
\begin{align}
\E\Big[ \Big(\big(\wwick{(\tilde Z_{t_1}^L)^n},\Psi^{x}_R\big)-\big(\wwick{(\tilde Z_{t_2}^L)^n},\Psi^{x}_R\big)\Big)^2\Big]
=
n!
  \int_{(\R^2)^2}	\Psi^{x}_R(y)\Psi^x_R(z) \cQ^n_L(t_1,t_2,y-z)\, dy\, dz
, 
\end{align}
where:
\begin{equation}
\cQ^n_L(t_1,t_2,y-z)
:=
(\cK^L(t_1,t_1,y-z))^n +(\cK^L(t_2,t_2,y-z))^n - 2(\cK^L(t_1,t_2,y-z))^n
.
\end{equation}
Let $\lambda\in(0,1)$. 
Elementary computations give, 
for each $|y-z|\leq 2R\leq 2$ and $L\geq 3$: 
\begin{align}
\cK^L(t_1,t_1,y-z)
-\cK^L(t_1,t_2,y-z)
&=
\int_0^{|t_1-t_2|}e^{-u}p^L_u(y-z)\, du + \int_{t_1+t_2}^{2t_1}e^{-u}p^L_u(y-z)\, du
\nnb
&\lesssim 
\frac{|t_1-t_2|^\lambda}{|y-z|^{2\lambda}}
,
\end{align}
where the proportionality constant depends on $\lambda$. 
This implies, 
bounding $\log_+(\sqrt{t}/|y-z|_L)\lesssim \log_+(|y-z|^{-1})$ uniformly in $t\geq 1$ for $|y-z|\leq 2R$:
\begin{equation}
|\cQ^n_L(t_1,t_2,y-z)|
\lesssim
\frac{|t_1-t_2|^\lambda}{|y-z|^{2\lambda}}
\big[1+\log_+(|y-z|^{-1})^{n-1}\big]
.
\end{equation}
The claim~\eqref{eq_gaussian_time_diff_bound_without_sup} follows by computation similar to~\eqref{eq_bound_ZPhi_T}. 
\end{proof}
In view of the relationship~\eqref{eq_link_wick_powers_w_w/o_IC} between Wick powers of $Z$ and $\tilde Z$ and the multiplicative inequality of Proposition~\ref{prop:besov-mult}, 
Proposition~\eqref{prop_gaussian_bounds0IC} implies the following bounds on the fields $Z,Z^L$.
\begin{proposition}\label{prop_gaussian_bounds_app}
  Let $n\geq 1$ be an integer, 
  let $\alpha,\sigma>0$ and recall that $\rho(x)=(1+|x|^2)^{-\sigma/2}$ and $Z^\infty:=Z$. 
There is $\kappa>0$ such that, for each $r\in[0,2]$ and some $\epsilon_r,\epsilon'_r>0$:
\begin{align}
\sup_{L\in[3,\infty]} \E\bigg[\exp\Big[\, \epsilon_r\Big(\, \sup_{s\in[0,t]}  (s^{n\alpha}\wedge 1)\| \wick{(Z^L_s)^n}\|_{-n\alpha,\rho^n}\Big)^{r/n}\, \Big]\bigg]
&\lesssim
(1+ t)^{\kappa}\exp\Big[\epsilon'_r\big(1+\|\varphi_0\|_{-\alpha,\rho}^{r}\big)\Big]
.
\label{eq_gaussian_bound_non0IC_rho_norm_app}
\end{align}
\end{proposition}
\begin{proof}
Let $L\in[3,\infty]$, $s\in(0,t]$ and $\alpha'>\alpha$.  
Equation~\eqref{eq_link_wick_powers_w_w/o_IC}, the multiplicative inequality of Proposition~\ref{prop:besov-mult} and the bound $\|\cdot\|_{\beta,\rho^n}\lesssim \|\cdot\|_{\beta',\rho^n}$ for any $\beta\leq \beta'$ imply:
\begin{align}
(s^{n\alpha}\wedge 1)&\, \|\wick{(Z^L_s)^n}\|_{-n\alpha,\rho^n}
\leq 
(s^{n\alpha}\wedge 1)\sum_{\ell=0}^n \binom{n}{\ell}\|\wick{(\tilde{Z}^L_s)^\ell} (e^{-sA}\varphi_0)^{n-\ell}\|_{-n\alpha,\rho^n}
\nnb
&\lesssim
(s^{n\alpha}\wedge 1)\max_{1\leq \ell\leq n}\|\wick{(\tilde{Z}^L_s)^\ell}\|_{-\ell\alpha,\rho^\ell}\ \|(e^{-sA}\varphi_0)^{n-\ell}\|_{\ell\alpha',\rho^{n-\ell}} + (s^{n\alpha}\wedge 1)\|(e^{-sA}\varphi_0)^n\|_{-n\alpha,\rho^n}
.
\end{align}
Note the following elementary bound valid for $\beta\geq 0$, any test function $f$ and any integer $p>1$:
\begin{equation}
\|f^p\|_{\beta,\rho^n} 
=
\|f^p\|_{\rho^n} + \sup_{x\in\R^2} \rho^n(x) \sup_{0<|z|\leq 1}\frac{|f^p(x)-f^p(x+z)|}{|z|^\beta}
\lesssim 
\|f\|_{\rho^{n/p}}^p +\|f\|_{\rho^{\frac{n-1}{p-1}}}^{p-1} \|f\|_{\beta,\rho}
.
\label{eq_norms_of_powers}
\end{equation}
Using the smoothing effect of $e^{-sA}$ (Proposition~\ref{prop:Besov-heat}), 
we can bound, for $\ell\in\{0,...,n\}$:
\begin{align}
\|(e^{-sA}\varphi_0)^{n-\ell}\|_{\ell\alpha',\rho^n}
&\lesssim 
\|(e^{-sA}\varphi_0)\|^{n-\ell}_{\rho} 
+ \|(e^{-sA}\varphi_0)\|^{n-\ell-1}_{\rho}\, \|(e^{-sA}\varphi_0)\|_{\ell\alpha',\rho} 
\nnb
&\lesssim 
\Big(s^{-\frac{(n-\ell)\alpha}{2}} + s^{-\frac{(n-\ell-1)\alpha}{2}}s^{-\frac{\ell\alpha'+\alpha}{2}}\Big)\, 
\|\varphi_0\|_{-\alpha,\rho}^{n-\ell}
\nnb
\|(e^{-sA}\varphi_0)^{n}\|_{\alpha,\rho^n}
&\lesssim
s^{-\frac{(n+1)\alpha}{2}}\|\varphi_0\|^n_{-\alpha,\rho}
.
\end{align}
Choose $\alpha'\in(\alpha,2\alpha)$ so that:
\begin{equation}
\min\Big\{n\alpha-\frac{(n-\ell)\alpha}{2} -\frac{\ell\alpha'}{2} : \ell\in\{0,...,n\}\Big\}
=:
\beta>0
.
\end{equation}
For such an $\alpha'$,
\begin{align}
(s^{n\alpha}\wedge 1)&\, \|\wick{(Z^L_s)^n}\|_{-n\alpha,\rho^n}
\lesssim 
\max_{1\leq \ell\leq n}\Big\{ (s^{\beta}\wedge 1)\, \|\wick{(\tilde{Z}^L_s)^\ell}\|_{-\ell\alpha,\rho^\ell}\, \|\varphi_0\|^{n-\ell}_{-\alpha,\rho}\Big\}
+\|\varphi_0\|^n_{-\alpha,\rho}
.
\end{align}
Recall the definition of $t,L$-independent counterterms from~\eqref{eq_def_diff_counterterms} and that $\wwick{(Z_s^L)^\ell}-\wick{(\tilde Z_s^L)^\ell}$ diverges like $\log(1/s)^{\lfloor \ell/2\rfloor}$ as $s\downarrow 0$ from~\eqref{eq_size_ftL}, 
uniformly in $L$. 
The divergence $\log(1/s)^{\lfloor \ell/2\rfloor}$ is absorbed in the $s^{\beta}\wedge 1$ prefactor for any $\ell\in\{1,...,n\}$. 
Thus, using the elementary inequality $ab\leq a^p/p+b^q/q$ ($a,b\geq 0$) in the second line with exponents $n/\ell$, $n/(n-\ell)$ for each $1\leq \ell<n$: 
\begin{align}
\sup_{s\in[0,t]}(s^{n\alpha}\wedge 1)\|\wick{(Z^L_s)^n}\|_{-n\alpha,\rho^n}
&\lesssim 
\|\varphi_0\|_{-\alpha,\rho}^{n} + \max_{1\leq \ell\leq n}\sup_{s\in[0,t]}\|\wwick{(Z^L_s)^\ell}\|_{-\ell\alpha,\rho^\ell}\, \|\varphi_0\|^{n-\ell}_{-\alpha,\rho}
\nnb
&\lesssim
\|\varphi_0\|_{-\alpha,\rho}^{n}\Big( 1+ \max_{1\leq \ell\leq n}\sup_{s\in[0,t]}\|\wwick{(Z^L_s)^\ell}\|_{-\ell\alpha,\rho^\ell}^{n/\ell} \Big)
.
\end{align}
Taking this expression to the power $r/n$ and invoking Proposition~\ref{prop_gaussian_bounds0IC} (with $\ell\alpha,\rho^\ell$ instead of $\alpha,\rho$ there) yields~\eqref{eq_gaussian_bound_non0IC_rho_norm_app}. 
\end{proof}

\subsection{Bounds on $Z-Z^L$}
We will also need bounds on $Z-Z^L$.
These are stated next and again follow from small modifications to the proof of~\cite[Theorem 5.1]{MR3693966}. 
\begin{proposition}\label{prop_gaussian_difference_app}
Let $\alpha,\sigma>0$, $\varphi\in C^{-\alpha}(\rho)$ and $n\in\N\setminus\{0\}$. 
For $a>0$, write $C_{a L}:= [- a L,a L]^2$.  
There is $c>0$ such that, for each $L\geq 12$, $t>0$ and each test function $f$ supported on $C_{\frac23 L}$:
\begin{align}
\E\Big[\, \big(\wick{Z_t^n}-\wick{(Z^L_t)^n},f\big)^2\, \Big]
&\lesssim 
(t^{-n\alpha}\vee 1)\,  \|f\|_\alpha^2 (1+\|\varphi_0\|^{2n}_{-\alpha,\rho})\, e^{-cL^2/t}
,
\label{eq_bound_Z-ZL_testfunction}
\\
\E\Big[\, \big\|\wick{Z_t^n}-\wick{(Z^L_t)^n}\big\|_{-n\alpha,\, [-\frac23 L,\frac23 L]^2 }^2\, \Big]
&\lesssim 
(t^{-n\alpha}\vee 1)\, (1+\|\varphi_0\|^{2n}_{-\alpha,\rho})\, e^{-cL^2/t}
.
\label{eq_bound_Z-ZL_app}
\end{align}
In addition, if $\sigma'>\sigma$ and $\rho'=(1+|\cdot|^2)^{-\sigma'/2}$,
\begin{equation}
\E\bigg[\, \big\|\wick{Z_t^n}-\wick{(Z^L_t)^n}\big\|_{-n\alpha,(\rho')^n}^2\, \bigg]
\lesssim
(t^{-n\alpha}\vee 1) \, (1+\|\varphi_0\|^{2n}_{-\alpha,\rho})\, \frac{1}{L^{2n(\sigma'-\sigma)}}
.
\label{eq_bound_Z-ZL_app_weighted}
\end{equation}
\end{proposition}
\begin{proof}
Recall that Wick powers of $Z^L_t$ and of the centred field $\tilde Z^L_t$,
where $t>0$ and $L\in [1,\infty]$, are related through:
\begin{equation}
\wick{(Z^L_t)^n}
=
\sum_{\ell=0}^n\binom{n}{\ell} \wick{(\tilde Z^L_t)^\ell} (e^{-tA}\varphi_0^L)^{n-\ell}
.
\end{equation}
From now on, always $L\geq 12$ and omission means that $L=\infty$.
Using the multiplicative inequality (Proposition~\ref{prop:besov-mult}),
similarly as in the proof of Proposition~\ref{prop_gaussian_bounds_app}, 
the last equation implies that it is enough to prove~\eqref{eq_bound_Z-ZL_app}--\eqref{eq_bound_Z-ZL_app_weighted} for both $\wick{(\tilde{Z}_t)^n} - \wick{(\tilde{Z}_t^L)^n}$ and $(e^{-tA}(\varphi_0-\varphi_0^L))^n$. 
As we shall see below, 
the weight $\rho'\ll \rho$ in~\eqref{eq_bound_Z-ZL_app_weighted} is introduced only to get a decay on this last term that is uniform in $\varphi_0$.  
In addition, 
in view of the relationship~\eqref{eq_wick_as_fct_wwick} between Wick powers and their homogeneous counterparts and recalling also properties of the difference $f(t,L)$ between counterterms from Lemma~\ref{lemma_size_ftL}, 
it suffices to bound $\wwick{\tilde Z_t^n}-\wwick{(\tilde Z^L_t)^n}$ rather than $\wick{\tilde Z_t^n}-\wick{(\tilde Z^L_t)^n}$.

On the other hand,~\eqref{eq_bound_Z-ZL_testfunction} 
requires us to estimate:
\begin{equation}
\E\Big[\,\big[\, \big(\wick{\tilde Z_t^\ell}, (e^{-tA}\varphi_0)^{n-\ell} f\big) -\big(\wick{(\tilde Z^L_t)^\ell},(e^{-tA}\varphi^L_0)^{n-\ell} f\big) \big]^2\, \Big]
,\qquad
0\leq \ell\leq n
.
\end{equation}
Computations are similar to those needed to prove~\eqref{eq_bound_Z-ZL_app}--\eqref{eq_bound_Z-ZL_app_weighted}, 
so we focus on proving those for 
each of $\wick{(\tilde{Z}_t)^n} - \wick{(\tilde{Z}_t^L)^n}$ and $(e^{-tA}(\varphi_0-\varphi_0^L))^n$.

\paragraph{Estimate of $\wwick{\tilde Z_t^n}-\wwick{(\tilde Z^L_t)^n}$} 
Recall from Proposition~\ref{prop:Besov-embedding} that, 
for each $p\geq 1$, $R\in(0,1]$ and $x\in\R^2$ (with $\Psi_R^x:=\Psi_R(x-\cdot)$):
\begin{align}
\sup_{R\in(0,1]}&R^{n \alpha p} \big|\big(\wwick{\tilde Z_t^n}-\wwick{(\tilde Z^L_t)^n},\Psi_R^x\big)\big|^p
\nnb
&\hspace{2cm}\lesssim
\int_0^1 R^{n\alpha p-2} 
\big\|\Psi_R\ast \big(\wwick{\tilde Z_t^n}-\wwick{(\tilde Z^L_t)^n}\big)\big\|^p_{\bbL^p(B_{3R}(x))} \, \frac{dR}{R}
.
\label{eq_supR_as_intR_Z-ZL}
\end{align}
The bound~\eqref{eq_gaussian_bound_without_sup} on $\wwick{\tilde Z^n_t}$ and $\wwick{(\tilde Z^L)^n_t}$ gives, 
as soon as $n\alpha p-2>0$:
\begin{equation}
\int_0^1 R^{n\alpha p-2} (t^{np\alpha}\wedge 1)\, \sup_{x\in\R^2}\E\Big[\big|\big(\wwick{\tilde Z_t^n}-\wwick{(\tilde Z^L_t)^n},\Psi^x_R\big)\big|^p\Big]\, \frac{dR}{R}
\lesssim 
1
.
\end{equation}
As $\rho\leq L^{-\sigma}$ outside of $C_{\frac{2}{3}L+3}$, 
we obtain:
\begin{equation}
\E\bigg[\, \sup_{x\notin C_{\frac23 L}} \rho^{2n}(x)\big\|\wwick{\tilde Z_t^n}-\wwick{(\tilde Z^L_t)^n}\big\|_{-n\alpha}^2\, \bigg]
\lesssim
(t^{-n\alpha}\vee 1) \frac{1}{L^{2n\sigma}}
.
\end{equation}
The bounds~\eqref{eq_bound_Z-ZL_app}--\eqref{eq_bound_Z-ZL_app_weighted} for $\wwick{\tilde Z_t^n}-\wwick{(\tilde Z^L_t)^n}$ therefore follow if we can prove that, 
for some $c>0$ and a large enough $p$:
\begin{equation}
\int_0^1 R^{n\alpha p-2} (t^{np\alpha}\wedge 1)\, \sup_{x\in C_{\frac23 L+3}}\E\Big[\big|\big(\wwick{\tilde Z_t^n}-\wwick{(\tilde Z^L_t)^n},\Psi^x_R\big)\big|^p\Big]\, \frac{dR}{R}
\lesssim 
e^{-cL^2/t}
.
\end{equation}
By Gaussian hypercontractivity it is enough to bound the expectation for $p=2$. 
Let us prove:
\begin{align}
&\sup_{x\in C_{\frac23 L+3}}\E\Big[ \big(\wwick{\tilde Z_t^n}-\wwick{(\tilde Z^L_t)^n},\Psi_R^x\big)^2\Big]
\lesssim
(1+|\log R\, |^n)\, 
e^{-cL^2/t}
.
\label{eq_bound_Z-ZL_app_toprove}
\end{align}
We again closely follow the proof of~\cite[Theorem 5.1]{MR3693966} where more general estimates are proven that however do not capture the exponential decay in $L^2/t$. 
Recall that $p_t$ is the heat kernel given in~\eqref{eq_OU_kernel} and $p^L_t$ its periodised version extended to $\R^2$ as in~\eqref{eq_OU_kernel_periodic_app}. 
For $x,y\in\R^2$, define:
\begin{equation}
\mathcal K^{\infty}_t(x,y)
:=
\int_0^t \int_{\R^2}e^{-2(t-r)}p_{t-r}(x-z)p_{t-r}(y-z)\, dz \, dr
=
\int_0^t e^{-2r}p_{2r}(x-y)\, dr
\end{equation}
and:
\begin{align}
\mathcal K^{L}_t(x,y)
:&=
\int_0^t e^{-2(t-r)}\int_{\R^2}p^L_{t-r}(x-z)p^L_{t-r}(y-z)\, dz \, dr
\nnb
&=
\sum_{a\in\Z^2}\int_0^te^{-2r} \int_{\R^2}p_{r}(x-z) p_r(y-z-2aL)\, dz \, dr
\label{eq_def_calK_L}
.
\end{align}
Define also:
\begin{equation}
\mathcal K^{L,\infty}_t(x,y)
=
\sum_{a\in \Z^2}\int_0^t e^{-2r}\int_{[-L,L]^2}p_{r}(x-z)p_{r}(y-z-2aL)\, dz\, dr
.
\end{equation}
Let $x_0\in\R^2$.
Direct computations using the definition~\eqref{e:wwick-var} of Wick powers and a change of variable then give:
\begin{align}
&\E\Big[\big(\wwick{\tilde Z_t^\ell}-\wwick{(\tilde Z^L_t)^\ell},\Psi^{x_0}_R\big)^2\Big]\nnb
&=
\ell!\int h(x)h(y)\big[\, (\mathcal K^L_t)^\ell + (\mathcal K^\infty_t)^\ell -2(\mathcal K^{L,\infty}_t)^\ell\, \big](x,y) \, dx\,dy
\nnb
  &\lesssim
\int_{B_{x_0}(1)^2} \big[\, |(\mathcal K^\infty_t)^\ell-(\mathcal K^{L,\infty}_t)^\ell| +|(\mathcal K^L_t)^\ell-(\mathcal K^{L,\infty}_t)^\ell|\, \big](Rx,Ry) \, dx\,dy
.
\label{eq_moment_tildeZ}
\end{align}
Let us bound the integrand in~\eqref{eq_moment_tildeZ}. 
One has:
\begin{align}
|(\mathcal K^\infty_t)^\ell-(\mathcal K^{L,\infty}_t)^\ell|
\lesssim 
|\mathcal K_t^\infty-\mathcal K^{L,\infty}_t| \big(1+(\mathcal K_t^\infty)^\ell+(\mathcal K^{L,\infty}_t)^\ell\big)
.
\end{align}
A similar bound holds for the other half of the integrand in~\eqref{eq_moment_tildeZ} and we now bound each of the above terms. 
Standard heat kernel bounds give:
\begin{equation}
\mathcal K^\infty_t(x,y)
\lesssim 
1+\log_+(|x-y|^{-1})
.
\label{eq_bound_calK_infty}
\end{equation}
Notice that the support $B_R(x_0)$ of $\Psi^{x_0}_R$ is included in $C_{\frac23 L+4}$.  
As $L\geq 12$, 
any $x,y\in C_{\frac{2}{3}L+4}$ 
satisfy $|x-y-2La|^2\geq L^2|a|^2/4$ ($a\in\Z\setminus\{0\}$). 
Using the elementary bound:
\begin{align}
\sum_{a\in \Z^2\setminus\{0\}}
e^{-a^2L^2/(8u)}
&\leq
e^{-L^2/(8u)}\Big(4+\Big[2\int_0^\infty e^{-b^2/(16u)}\, db\Big]^2\Big)
\nnb
&=
4e^{-L^2/(8u)}(1+16\pi u)
,\qquad u\in(0,t]
,
\end{align}
we find, for some $c>0$:
\begin{align}
\mathcal K^{L}_t(x,y)
&=
\mathcal K^{\infty}_t(x,y)
+\sum_{a\in \Z^2\setminus\{0\}}\int_0^t e^{-2r}p_{2r}(x-y-2La)\, dr
\nnb
&\lesssim
1
+\log_+(|x-y|^{-1})
+\int_0^t\frac{(1+u)}{u} e^{-L^2/(8u)}e^{-2u}\, du
\nnb
&\lesssim 
1
+\log_+(|x-y|^{-1})\big[1+e^{-cL^2/t}\big]
.
\label{eq_bound_calK_L_t}
\end{align}
It remains to estimate $\mathcal K^{L,\infty}_t$, 
which we do by bounding its distance to $\mathcal K^\infty_t,\mathcal K^L_t$. 
Noting that for some $c\in(0,1)$ one has $|x-z|^2\geq |x-z|^2/2 + 4cL^2$ as soon as $z\notin[-L,L]^2$ and $x\in C_{\frac{2}{3}L+4}$,
\begin{align}
\mathcal K^L_t(x,y)
-\mathcal K^{L,\infty}_t(x,y)
&=
\sum_{a\in\Z^2}\int_0^te^{-2r}\int_{\R^2\setminus[-L,L]^2}p_r(x-z)p_r(y-z-2La)\, dz\, dr
\nnb
&\lesssim
\sum_{a\in\Z^2}\int_0^te^{-cL^2/r} e^{-2r}\int_{\R^2\setminus[-L,L]^2}p_{2r}(x-z)p_r(y-z-2La)\, dz\, dr
\nnb
&\lesssim
\sum_{a\in\Z^2}\int_0^te^{-cL^2/r} e^{-2r} \int_{\R^2}p_{2r}(z)p_r(y-x-z-2La)\, dz\, dr
\nnb
&\lesssim
e^{-cL^2/t}\int_0^te^{-2r}p^L_{3r}(y-x)\, dr
\nnb
&\lesssim
e^{-cL^2/t}
\big[1+\log_+(|x-y|^{-1})\big]
.
\end{align}
Similarly,
\begin{align}
\mathcal K^\infty_t(x,y)
-\mathcal K^{L,\infty}_t(x,y)
&=
\int_0^te^{-2r}\int_{\R^2\setminus[-L,L]^2}p_{r}(z-x)p_r(y-x-z)\, dz\, dr
\nnb
&\quad
-
\sum_{a\in\Z^2\setminus\{0\}}\int_0^te^{-2r}\int_{[-L,L]^2}p_{r}(z-x)p_r(y-z-2La)\, dz\, dr
.
\end{align}
The first term is just $\mathcal K^L_t-\mathcal K^{L,\infty}_t$, 
while the second one is bounded by:
\begin{equation}
\sum_{a\in\Z^2\setminus\{0\}}\int_0^te^{-2r} p_{2r}(y-x-2La)\, dz\, dr
\lesssim 
e^{-cL^2/t}\big[1+\log_+(|x-y|^{-1})\big]
. 
\end{equation}
Since any power of $\log(|\cdot|^{-1})$ is integrable around $0$, 
recalling~\eqref{eq_moment_tildeZ} and
putting all bounds together yields the desired bound~\eqref{eq_bound_Z-ZL_app_toprove}:  ,  
\begin{align}
\E\Big[\big(\wwick{\tilde Z_t^n}-\wwick{(\tilde Z^L_t)^n},
\Psi^{x_0}_R\big)^2\Big]
&\lesssim 
(1+|\log R\, |^n)\, e^{-cL^2/t}
. 
\end{align}

\paragraph{Initial condition}

We now prove~\eqref{eq_bound_Z-ZL_app}--\eqref{eq_bound_Z-ZL_app_weighted} for $(e^{-tA}(\varphi_0-\varphi_0^L))^n$. 
One has:
\begin{align}
\big|(e^{-tA}\varphi_0)^{n}-(e^{-tA}\varphi_0^L)^{n}\big|
&=
(e^{-tA}\varphi_0-e^{-tA}\varphi_0^L)\sum_{\ell=0}^{n-1}(e^{-tA}\varphi_0)^\ell(e^{-tA}\varphi^L_0)^{n-1-\ell}
. 
\end{align}
Let $x_0\in C_{\frac23 L}$. 
Using $\rho\gtrsim L^{-\sigma}$ on $C_{\frac23 L}$, 
the smoothing effect of the heat kernel (Proposition~\ref{prop:Besov-heat}) and $\|\varphi_0^L\|_{-\alpha,\rho}\lesssim\|\varphi_0\|_{-\alpha,\rho}$:
we get:
\begin{align}
\big|(e^{-tA}\varphi_0)^{n}-(e^{-tA}\varphi_0^L)^{n}\big|(x_0)
&\lesssim
L^{(n-1)\sigma}\big|e^{-tA}\varphi_0-e^{-tA}\varphi_0^L\big| (x_0) 
\max_{0\leq\ell\leq n-1}
\|e^{-tA}\varphi_0\|^\ell_{\rho}\|e^{-tA}\varphi_0\|^{n-1-\ell}_{\rho}
\nnb
&\lesssim
L^{(n-1)\sigma}\big|e^{-tA}\varphi_0-e^{-tA}\varphi_0^L\big| (x_0) \,
(t^{-(n-1)\alpha/2}\vee 1)\|\varphi_0\|^{n-1}_{-\alpha,\rho}
.
\end{align}
Let $\chi\geq 0$ be supported on $[-9L/10,9L/10]^2$ and equal to $1$ on $[-4L/5,4L/5]^2$.  
Then, by definition of $\varphi^L_0$ (recall~\eqref{eq_def_varphi_L0}):
\begin{equation}
e^{-tA}\varphi_0(x_0) - e^{-tA}\varphi_0^L(x_0)
=
e^{-t}(\varphi_0-\varphi^L_0, p_{t}(x_0-\cdot)(1-\chi))
,\qquad 
x\in\R^2
.
\end{equation}
Recall the following elementary bounds: 
for any $\alpha'\in(0,1)$, 
there is $c>0$ such that, for any $|y|\geq 9L/10$ and $|z|\leq 1$,
\begin{align}
|p_t(x_0-y)|
&\lesssim 
e^{-cL^2/t} \frac{1}{t}e^{-c|y|^2/t}
,
\\
|p_t(x_0-y-z)-p_t(x_0-y)|
&\lesssim 
|z|^{\alpha'} t^{-\alpha'/2} \frac{1}{t}e^{-c|y|^2/t}
.
\end{align}
In particular, for any $\beta\in(\alpha,(n+1)\alpha)$, 
using the above with $\alpha'=(n+1)\alpha>\beta$ and recalling definition~\eqref{e:dnorm-def} of the norm $\dnorm{\cdot}_{\beta,(\rho')^{-1}}$:
\begin{equation}
e^{-t}\dnorm{p_{t}(x_0-\cdot)(1-\chi)}_{\beta,(\rho')^{-1}}
\lesssim 
(1+t^{-(n+1)\alpha/2}) \, e^{-t}\, e^{-cL^2/t}
\lesssim
(t^{-(n+1)\alpha/2}\vee 1)\, e^{-cL^2/t}
.
\end{equation}
Remark~\ref{rk:completion} to get the first line below 
and $\|\varphi_0^L\|_{-\alpha,\rho'},\|\varphi_0\|_{-\alpha,\rho'}\lesssim\|\varphi_0\|_{-\alpha,\rho}$ then give:
\begin{align}
\big|e^{-tA}\varphi_0(x_0) - e^{-tA}\varphi_0^L(x_0)\big|
&\lesssim 
e^{-t}\|\varphi_0-\varphi^L_0\|_{-\alpha,\rho'}
\ \dnorm{p_{t}(x_0-\cdot)(1-\chi)}_{\beta,(\rho')^{-1}}
\nnb
&\lesssim 
e^{-t}\|\varphi_0\|_{-\alpha,\rho}\  \dnorm{p_{t}(x_0-\cdot)(1-\chi)}_{\beta,(\rho')^{-1}}
\nnb
&\lesssim
(t^{-(n+1)\alpha/2}\vee 1)\|\varphi_0\|_{-\alpha,\rho} \, e^{-cL^2/t}
.
\end{align}
Taking the supremum over $x_0\in C_{\frac23 L}$ yields~\eqref{eq_bound_Z-ZL_app} for the initial condition. 

To prove~\eqref{eq_bound_Z-ZL_app_weighted}, 
notice that any $\xi \in C^{-\alpha}(\rho)$ satisfies:
\begin{align}
\sup_{x\notin C_{\frac23 L}} \rho'(x)\sup_{R\in(0,1]}R^\alpha|\xi\ast \Psi_R (x)|
&\lesssim 
L^{\sigma-\sigma'}\sup_{x\notin C_{\frac23 L}} \rho(x)\sup_{R\in(0,1]}R^\alpha|\xi\ast \Psi_R (x)|
\nnb
&\leq 
L^{\sigma-\sigma'}\|\xi\|_{-\alpha,\rho}
.
\end{align}
Together with the previous bound valid inside $C_{\frac23 L}$ this gives the claim. 
\end{proof}

\section{Proof of a priori estimates for the SPDE}
\label{app:apriori}

In this section we provide some estimates for solutions of the  $\varphi^4$ equations. The method is an adaptation of \cite{MR4164267} to the simpler two-dimensional case rather than the three-dimensional case treated there, and  to the specific needs of the present work. Throughout the section we will work with the remainder equation
\begin{equation}\label{eq:remainder_equation}
(\partial_t - \Delta) v = - \lambda v^3 +   (-3v^2 Z_1 - 3vZ_2 - Z_3)
\end{equation}
for $\lambda >0$ and under the deterministic assumption of control on the space-time distributional norms $\| Z_1 \|_{-\kappa,B_2(0)}$,
$\| Z_2 \|_{-2\kappa, B_2(0)}$, and $\| Z_3 \|_{-3\kappa, B_2(0)}$,
see  Section~\ref{sec:parabolic-norms}  below
for the definition of these space-time norms.
The main results are a control of the space-time $\bbL^\infty$ norm 
\begin{equation}
\| v \|_{B_1(0)}:= \sup_{z \in (-1,0] \times \{x \colon | x | < 1 \}} | v(z) |
\end{equation}
and a local space-time $\alpha$-H\"older seminorm
\begin{equation}
[ v ]_{\alpha, B_1(0)}:= \sup_{z, \bar{z} \in (-1,0] \times \{x \colon | x | < 1 \}, z \neq \bar{z}}  \frac{| v(z) - v(\bar{z}) |}{d(z,\bar{z})^\alpha },
\end{equation}
where the parabolic metric $d$ is defined below in \eqref{eq:definition_parabolic_distance}.
The main result is the following theorem:

\begin{theorem}\label{thm:a_priori_estimates_v}
  Let $v$ be a continuous function on $\overline{B_2(0)} =  [-4, 0 ] \times \{ x \in \R^d_x \colon |x | \leq 2 \}  $ which solves  \eqref{eq:remainder_equation} in the distributional sense on $B_2(0) =  (-4, 0 ) \times \{ x \in \R^d_x \colon |x | < 2 \}  $.
  Let $\kappa>0$ be small enough and fix $\alpha \in (0,1)$.   Then
\begin{equation}\label{eq:main_remainder_estimate}
\| v \|_{B_1(0)} + [v]_{\alpha,B_1(0)} ^{\frac{1}{1+\alpha}}   \lesssim 1+ \Big(   \max \Big\{    \|Z_1 \|_{-\kappa,B_2(0)}, \, \|Z_2 \|^{\frac{1}{2}}_{-2\kappa,B_2(0)} , \, \|Z_3 \|^{\frac{1}{3}}_{-3\kappa,B_2(0)} \Big\}   \Big)^{\frac{1}{1-\kappa}}
\end{equation}
with an implicit constant that depends on $\lambda, d,\kappa, \alpha$.
\end{theorem}

Theorem~\ref{thm:a_priori_estimates_v} is proven at the end of Section~\ref{sec_largescale_remainder_eq}. 
The theorem implies the estimates for the $\varphi^4_2$ SPDE
stated in Theorem~\ref{thm:apriori_bounds} and used in the bulk of the paper.
For convenience, we restate the proposition as the following corollary
and give its prove before giving the proof of Theorem~\ref{thm:a_priori_estimates_v}.

\begin{corollary}
\label{cor:apriori_bounds}
  Let $\alpha' \in [0,1)$ (with $0$ included), let  $\alpha >0$ be small enough, and set $\eta= \frac{1+\alpha'}{1-3\alpha}$.
    For $L\in \N\cup\{\infty\}$, let $v^L$ be the solution of the remainder equation  \eqref{e:YL-Duhamel}.
  
  \smallskip
  \noindent
  (i)  
  For initial condition $v_0=0$ (as in our standard convention \eqref{eq:initial-data-choice}),
  for each $t\geq 1$ and each ball $B=B_1(x)\subset\R^2$ in the spatial variable:
\begin{equation}
  \sup_{0 \leq s \leq t} \|v^L_s\|_{\alpha',B}
  + \sup_{\substack{0 \leq \bar s< s \leq t \\ |s-\bar s|\leq 1}} \frac{\|v^L_s-v^L_{\bar s}\|_B}{|\bar s-s|^{\alpha'/2}}
\lesssim 
1+ 
  \sup_{0<s\leq t}\max_{n=1,2,3} \Big\{
 	 \Big( 
	 	(s^{n\alpha}\wedge 1)\|\wick{(Z^L_s)^n} \|_{-n\alpha,2B} 
	\Big)^{\frac{1}{n}}
\Big\}^\eta .
\label{eq_desired_bounds-B}
\end{equation}
Furthermore, for any $\sigma>0$, 
\begin{equation}
  \sup_{0 \leq s \leq t } \|v^L_s\|_{\alpha',\rho}
  + \sup_{\substack{0 \leq \bar s< s \leq t \\ |s-\bar s|\leq 1}} \frac{\|v^L_s-v^L_{\bar s}\|_\rho}{|\bar s-s|^{\alpha'/2}}
\lesssim 
1+ 
  \sup_{0<s \leq t}\max_{n=1,2,3} \Big\{
  \Big(
  (s^{n\alpha}\wedge 1)\|\wick{(Z^L_s)^n} \|_{-n\alpha,\rho^{\frac{n}{\eta}}} 
\Big)^{\frac{1}{n}}
\Big\}^\eta 
.
\label{eq_desired_bounds-rho}
\end{equation}

  \smallskip
  \noindent
  (ii)
  For arbitrary initial condition $v_0 \in C^{-\alpha}(\rho)$ and $t\geq 1$, one also has
\begin{equation}
  \|v^L_t\|_{\alpha',\rho}
  \lesssim 
  1+ 
  \sup_{0<s \leq t}\max_{n=1,2,3} \Big\{
  \Big(
  (s^{n\alpha}\wedge 1)\|\wick{(Z^L_s)^n} \|_{-n\alpha,\rho^{\frac{n}{\eta}}} 
  \Big)^{\frac{1}{n}}
  \Big\}^\eta 
.
\label{eq_desired_bounds-rhorho-alternative-appendix}
\end{equation}
The implicit constants  are all independent of $t,L$ and $v_0, \varphi_0$.
\end{corollary}

\begin{proof}
The following argument does not change for $L<\infty$, and therefore for the rest of the proof we omit the subscripts $L$ from $v^L$ and $\wick{(Z^L)^n}$. 
Let $v$ solve the integral equation \eqref{e:Y-Duhamel} for  $t \geq 0$. 

For the proof of (i), due to the choice of zero initial data for $v$, see \eqref{eq:initial-data-choice},
we can extend both $v$ and the distributions $Z_s$, $\wick{Z_s^2}$ and $\wick{Z_s^3}$ to negative times $s<0$  by setting 
\begin{equation}
v_s= Z_s =  \wick{Z_s^2} =  \wick{Z_s^3} = 0, \qquad s <0 ,
\end{equation}
and  rewrite \eqref{e:Y-Duhamel} as 
\begin{equation} \label{e:Y-Duhamel-bis}
  v_t = 
  \int_{t_0}^t e^{-(t-s)A} \Big[-\mu v - \lambda [v_s^3+3v_s^2 Z_s + 3v_s\wick{Z_s^2}+ \wick{Z_s^3}]\Big] \, ds,
\end{equation}
valid for all $t_0 \leq 0$ and $t \geq t_0$. A standard argument (see e.g. \cite[Proposition 13]{MR3693966}) shows that the mild formulation implies that the remainder equation for $v$ also holds in the distributional sense for all times $s \in \mathbb{R}$. In  order to apply Theorem \ref{thm:a_priori_estimates_v} we first rewrite the equation as  
\begin{equation}
(\partial_t - \Delta) v = - \lambda v^3   + \Big[- 3  v^2  \lambda Z - 3v \Big( \lambda \wick{Z^2} +\frac{1+\mu}{3} \Big) -\lambda  \wick{Z^3} \Big].
\end{equation}
Thus, to get \eqref{eq_desired_bounds-B} we start with
\begin{align}
  &\sup_{s\in [0,t]} \|v_s\|_{\alpha',B}
    + \sup_{\substack{0 \leq \bar s< s \leq t \\ |s-\bar s|\leq 1}} \frac{\|v^L_s-v^L_{\bar s}\|_B}{|\bar s-s|^{\alpha'/2}} \nnb
& \leq   \sup_{s\in [0,t]}  \Big( \|v\|_{B_1(s,x)} + 2[v]_{\alpha',B_1(s,x)}  \Big) \nnb
& \lesssim   1+ \sup_{s\in [0,t]}  \Big( \|v\|_{B_1(s,x)} + [v]_{\alpha',B_1(s,x)}^{\frac{1}{1+\alpha'}}  \Big)^{1+\alpha'}
\end{align}
where in the (trivial) first step we bound the $\bbL^\infty$ bound over the space-ball at time $s$ by the $\bbL^\infty$ norm over a space-time  ball,
and similarly for the H\"older norms.
Applying \eqref{eq:main_remainder_estimate} this is
\begin{equation}
  \lesssim 1+ \sup_{s\in [0,t]} \Big(   \max \Big\{    \|Z_1 \|_{-\kappa,B_2(0)}, \, \|Z_2 \|^{\frac{1}{2}}_{-2\kappa,B_2(0)} , \, \|Z_3 \|^{\frac{1}{3}}_{-3\kappa,B_2(0)} \Big\}   \Big)^{\frac{1+\alpha'}{1-\kappa}}
\end{equation}
with 
\begin{equation}
  Z_1  = \lambda\,  Z_{s+\cdot}(x+\cdot), \qquad  Z_2 =  \lambda \, \wick{Z^2_{s+\cdot}(x+\cdot)} +\frac{1+\mu}{3}, \qquad Z_3 =  \lambda  \, \wick{Z^3_{s+\cdot}(x+\cdot)} ,
\end{equation}
defined for negative times as discussed above.
To get \eqref{eq_desired_bounds-B}, it  only remains to bound the space-time distributional norms of the $Z_i$, $i \in \{ 1,2,3 \}$. The constant term $\frac{1+\mu}{3}$ as well as the pre-factors $\lambda$ can be absorbed into the term $1$ and the implicit constant. To estimate the space-time distributional H\"older norms of  $\wick{Z^n}$ we use the factorisation $\Psi(t,x) = \Psi^{(1)}(t) \Psi^{(2)}(x)$ (see discussion around equation~\eqref{eq:self-similarity} below)  to write 
\begin{align}
\notag
\| \wick{Z^n} \|_{- n\kappa,B_2(s,x)} 
	& =  \sup_{\substack{\bar{s}, \bar{x}, R \\ B_R( \bar{s}, \bar{x}) \subset B_2(s,x) }} 
		R^{n \kappa} | \Psi_R \ast \wick{Z^n}(\bar{s}, \bar{x}) |  \\
\notag
	&=    \sup_{\substack{\bar{s}, \bar{x}, R \\ B_R( \bar{s}, \bar{x}) \subset B_2(s,x) }} 
		R^{n \kappa} \int_0^\infty R^{-2} \Psi^{(1)} \Big( \frac{\bar{s} - r}{ R^2} \Big)    R^{-n\alpha}  \| \wick{Z_r^n}  \|_{-n\alpha , B_2(x)}  dr \\
\notag
	&\leq \sup_{0 <s<t}  (s^{n\alpha}\wedge 1)\|\wick{(Z_s)^n} \|_{-n\alpha,2B}   \\
\notag
		& \qquad \times    \sup_{\substack{\bar{s}, \bar{x}, R \\ B_R( \bar{s}, \bar{x}) \subset B_2(s,x) }} R^{n \kappa}  R^{-n\alpha}  \int_0^\infty  R^{-2} \Psi^{(1)} \Big( \frac{\bar{s} - r}{ R^2} \Big) ( r^{-n\alpha } \vee 1 )      dr \\
\label{e:space-time-distribution-bound}
	&\leq \sup_{0 <s<t}  (s^{n\alpha}\wedge 1)\|\wick{(Z_s)^n} \|_{-n\alpha,2B}   
   \sup_{\substack{\bar{s}, \bar{x}, R \\ B_R( \bar{s}, \bar{x}) \subset B_2(s,x) }} R^{n \kappa}  R^{-n\alpha} R^{-2n\alpha}.
\end{align}
The latter supremum is finite for $\kappa = 3\alpha $.  Thus, \eqref{eq_desired_bounds-B}   follows for $\eta = \frac{1+\alpha'}{1 - 3 \alpha}$.  

To see \eqref{eq_desired_bounds-rho}, we observe that 
\begin{equation}
\| v_s \|_{\alpha,\rho} \lesssim \sup_{x \in \R^2} \rho(x) \| v_s \|_{\alpha,B_1(x)},
\end{equation}
so that  \eqref{eq_desired_bounds-B} yields
\begin{equation}
\sup_{s\in [0,t]} \|v_s\|_{\alpha,\rho}
\lesssim 1+   \sup_{s\leq t}  \max_{n=1,2,3}   \Big\{
  \Big(   (s^{n\alpha}\wedge 1) \sup_{x \in \R^2} \rho^{\frac{n}{\eta}}(x) \|\wick{(Z_s)^n} \|_{-n\alpha,2B}  \Big)^{\frac{1}{n}} \Big\}^\eta .
\end{equation}
The same applies to the time-H\"older norm,
which in turn implies   \eqref{eq_desired_bounds-rho}. 

The proof of (ii) is similar. 
Let $t\geq 1$.
We again use Theorem~\ref{thm:a_priori_estimates_v} to estimate a space-time H\"older norm 
of $v$  (say on the time-interval $[t/2,t]$) in terms of the space-time H\"older norms of 
\begin{equation}
\tilde{Z}_1  = \lambda \tilde{Z}, \qquad \tilde{Z}_2 =  \lambda \wick{\tilde{Z}^2} +\frac{1+\mu}{3}, \qquad \tilde{Z}_3 =  \lambda  \wick{\tilde{Z}^3} ,
\end{equation}
for times in $[t/4,t]\subset(0,t]$. These can be controlled  as in \eqref{e:space-time-distribution-bound}.
Actually, the bounds here are even slightly better because one does not have to deal with the blow-up near~$0$.
\end{proof}

\subsection{Parabolic H\"older norms}
\label{sec:parabolic-norms}

We will use parabolic space-time H\"older norms.
These are defined in the same way as the ``elliptic'' norms (for spatial variables) in Section~\ref{sec:norms-def},
except for inclusion of a time variable which is scaled parabolically,
and they have completely analogous properties to those given in Appendix~\ref{app:norms}.

\paragraph{Parabolic scaling and test function}

For space-time points $z =(t,x),\;\bar{z} = (\bar{t}, \bar{x})  \in \R_t \times \R_x^d$ we define the
parabolic metric  by
\begin{equation}\label{eq:definition_parabolic_distance}
d(z, \bar{z}) = \sqrt{|t - \bar{t}| } + |x - \bar{x}|,
\end{equation}
where $|x - \bar{x}|$ refers to the Euclidean norm on $\R_x^d$.
For $R>0$ and $z=(t,x) \in \R_t \times \R^d_x$  define parabolic cylinders \emph{in the past of $z$} by
\begin{equation}
B_R(z) = ( t-R^2, t] \times \{ \bar{x} \in \R^d_x \colon | \bar{x} - x| < R \},
\end{equation}
where again $ | \bar{x} - x| $ refers to the Euclidean norm on $\R^d_x$.
For a space-time function  $f \colon \R_t \times \R_x^d \to \R$ and for $R>0$ we define its rescaling $f_R$ by 
\begin{equation}
f_R(t,x) = R^{-d-2} f\Big(\frac{t}{R^2}, \frac{x}{R}\Big).
\end{equation}
The space-time version of the test function $\Psi$ used to define H\"older norms is
smooth and compactly supported with   $\int  \Psi \, dt\, dx = 1$. It is convenient (but probably not essential - see Proposition \ref{prop:norm-tilde}) to assume that  $\Psi$ factorises between the time and space-variables and has an approximate self-similarity property as in \cite[Section~2.1]{MR4164267}: 
More precisely, let  $\Phi(t,x) = \Phi^{(1)}(t) \Phi^{(2)}(x)$ be a smooth non-negative function compactly supported on the space-time ball
$-B_1(0)$ 
with $\int  \Phi \, dt\, dx = 1$.
  For $R>0$ and $n \in \N$ 
  set  $\Psi_{R,n} = \Phi_{2^{-1}R}  \ast \Phi_{2^{-2}R} \ast \ldots \ast \Phi_{R2^{-n}}$ and $\Psi_R =\lim_{n \to \infty }\Psi_{R,n}$.
  We observe that  $\Psi_R$ has a representation as the space-time convolution:
 \begin{equation}\label{eq:self-similarity}
 \Psi_R =\Psi_{R 2^{-n}} \ast  \Psi_{R,n} 
 \end{equation}
 of its rescaled version $\Psi_{R 2^{-n}}$ which is supported on the ball $-B_{R2^{-n}}(0)$ and $\Psi_{R,n}$ which is supported in $-B_{R(1-2^{-n})}(0)$.
This construction implies that $\Psi(t,x) = \Psi^{(1)}(t)\Psi^{(2)}(x)$  factorises between the space and time variable, where we assume that $\Psi^{(2)}$ coincides with the kernel used to define the spatial H\"older 
norms of negative regularity defined in Section ~\ref{sec:norms-def}. This decomposition is useful to relate distributional space-time H\"older norms with spatial H\"older norms (see equation \eqref{e:space-time-distribution-bound}).
 Furthermore, note that  $(\Psi \ast f)  (t,x)$, where $\ast$ denotes space-time convolution, only depends on values of $f$ ``in the past of $t$''.

\paragraph{Norms}

Using the parabolic scaling from the previous paragraph instead of the elliptic scaling from Section~\ref{sec:norms-def},
the H\"older norms can be defined analogously to Section~\ref{sec:norms-def}.
In that section, we defined the inhomogeneous versions of these norms in which distances and the scaling parameter $R$ are restricted to $(0,1]$. In this section, it is convenient to use the homogenous versions of the norms
in which there is no restriction on the distances and on $R$. The final statement
Theorem~\ref{thm:a_priori_estimates_v} applies  to unit balls where both definitions agree.

Thus the H\"older seminorms are defined with respect to the parabolic metric $d$,
i.e., for $\alpha \in (0,1)$ and for any $B \subseteq \R_t \times \R^d_x$, 
\begin{equation}\label{eq:definition-hoelder-seminorm}
[v]_{\alpha,B} = \sup_{z \neq \bar{z} \in B} \frac{|v(z)- v(\bar{z})|}{d(z,\bar{z})^\alpha} .
\end{equation}
We observe the following simple estimate for later use: 
\begin{equation}\label{eq:regularisation_error}
    | f(z) - f\ast \Psi_R(z) |
  \leq [f]_{\alpha, B_R(z)} R^\alpha.
\end{equation}
The localised Besov--H\"older norms of negative regularity are defined as in Section~\ref{sec:norms-def}.
Thus the homogenous version used in this section is defined, for $\alpha >0$ and $B \subseteq \R_t \times \R_x^d$, by 
\begin{equation}\label{eq:def_besov_norm}
\| Z \|_{-\alpha, B} := \sup_{\substack{z \in B, \,R>0  \\ B_R(z) \subseteq B}}  |\Psi_R * Z(z)| R^\alpha.
\end{equation}
From (the parabolic analogue) of Proposition~\ref{prop:besov-mult} we also recall the reconstruction or commutator estimate:
\begin{equation} \label{e:reconstruction}
| (v Z \ast \Psi_R) (z) - v(z) (Z\ast\Psi_R)(z) | \lesssim  R^{\beta -\alpha} [v]_{\beta, B_{2R}(z)}    \| Z \|_{-\alpha, B_{2R}(z)},
\end{equation}
valid if $0<\alpha<\beta$, and in particular
\begin{equation} \label{e:reconstruction2}
| (v Z \ast \Psi_R) (z)| \lesssim  \qB{ |v(z)|  + R^\beta[v]_{\beta, B_{2R}(z)}  } R^{-\alpha}  \| Z \|_{-\alpha, B_{2R}(z)}.
\end{equation}
 In the following estimate the self-similarity \eqref{eq:self-similarity} is convenient:
  Let $\eta$ be supported in $B_{2R}(0)$. By choosing $n$ such that $2^{-n}L \leq R < 2 \cdot 2^{-n}L$
  and $\Psi_L = \Psi_{L 2^{-n}} \ast  \Psi_{L,n}$, then 
\begin{align}
\notag
\| \eta f  \ast \Psi_L \| &=   \|  \eta f  \ast \Psi_{L2^{-n}} \ast  \Psi_{L,n}  \| \\
\notag
&\lesssim    \|  \eta f  \ast \Psi_{L2^{-n}}  \|_{B_{3R}(0)}  \| \Psi_{L,n} \|  | B_{3R}(0)|   \\
\notag
&\lesssim   R^{-\alpha} \frac{R^{2+d}}{L^{2+d}}\qa{ \sup_{r\leq R} r^{\alpha} \|  \eta f  \ast \Psi_{r}  \|_{B_{3R}(0)}}     \\
\label{e:reconstruction3}
 &  \lesssim R^{-\alpha}\frac{R^{2+d}}{L^{2+d}} \qB{\|\eta \|  + R^\beta [\eta]_\beta } \| f \|_{-\alpha, B_{5R}(0)}  ,
\end{align}
where we used that $ \eta f  \ast \Psi_{L2^{-n}}  $ is supported on $B_{3R}(0)$, $|B_{3R}(0)|\propto R^{d+2}$ denotes the volume of $B_{3R}(0)$
and $\|\Psi_{L,n}\|\propto L^{-d-2}$,
and in the last inequality applied \eqref{e:reconstruction}.

\subsection{Local Schauder estimate}\label{ss:Technical_tools}

We give a self-contained proof of the local Schauder estimate in the form required.
The main result is Corollary~\ref{Cor:local_Schauder_estimate}.
The first step is a low-regularity \emph{global} Schauder estimate, that we now proceed to discuss:
To capture the correct time-dependence of solutions of the heat equation we define H\"older norms on half-spaces:
\begin{equation*}
[v]_{\alpha,T} := [v]_{\alpha,(-\infty, T] \times \R_x^d } = \sup_{\substack{z,\bar{z} \in (-\infty,T] \times \R^d_x \\ z \neq \bar{z}} } \frac{ |v(z) - v(\bar{z} )|}{d(z,\bar{z})^\alpha} ,
\end{equation*}
and the corresponding negative regularity  distributional norm:
\begin{equation}\label{eq:def_global_Besov_norm}
\| f \|_{\alpha -2,T} := \| f \|_{\alpha -2,(\infty, T] \times \R_x^d } = \sup_{R>0} \sup_{z=(t,x) \colon t \leq T }  R^{2-\alpha} |\Psi_R\ast f|(z).
\end{equation}
With this notation in place, the global Schauder estimate takes the following form. Its proof is an adaptation of
\cite[Theorem 8.6.1]{MR1406091} to the present lower regularity context.

\begin{lemma}[Global Schauder estimate]\label{lem:Schauder_Lemma}
Let $ T \in \R$ and let $v$ be a continuous function on $(-\infty, T]  \times \R^d_x$  with compact support that satisfies 
\begin{equation}\label{eq:assumption_schauder}
(\partial_t - \Delta)v = f \qquad \text{ on } (-\infty, T)  \times \R^d_x
\end{equation} 
in the distributional sense. For $\alpha \in (0,1)$, if  $\| f \|_{\alpha -2,T}$ is finite, then so is $[v]_{\alpha,T}$ and 
\begin{equation}\label{eq:global_Hoelder_estimate}
[v]_{\alpha, T} \lesssim  \| f\|_{\alpha - 2,T} , 
\end{equation}
with an  implicit constant depending only on $d$ and $\alpha $.
\end{lemma}

\begin{proof}[Proof of Lemma~\ref{lem:Schauder_Lemma} ]

By a standard approximation argument we can make the qualitative assumption that $[v]_{\alpha,T} < \infty$; indeed if we only assume that $v$ is continuous and compactly supported, 
apply the result to a function $v$ which is regularised, e.g. by convolution with a kernel at scale $\varepsilon$. The estimates on regularisations of $f$ are uniform in $\varepsilon$ and one can pass to the limit 
$\varepsilon \to 0$. 
Without loss of generality we assume $T=0$ and omit the subscript $T$, i.e., we will write  $[v]_{\alpha}$ instead of $[v]_{\alpha, T}$ and 
$  \| f\|_{\alpha - 2} $ instead of $ \| f\|_{\alpha - 2,T}$ for the half-space norms throughout the proof.
In particular, all functions respectively distributions used in the following argument are only evaluated for 
negative times.  

We consider the quantity
\begin{equation}
V(z_0,R) =  \frac{1}{R^{\alpha}}\| v - v(z_0) \|_{B_{R}(z_0)} ,
\end{equation}
noting that clearly 
\begin{equation}\label{eq:different_characterisations_Hoelder}
[v]_\alpha = \sup_{R>0} \sup_{z \in (\infty, 0]  \times \R^d_x}  V(z,R).
\end{equation}
To bound this quantity we  fix $z_0$ and a $R>0$ and introduce two auxiliary  scales $ \underline{R}<R$ and  $\overline{R} > R$ connected via
\begin{align}\label{eq:choice_R_and_T}
 \varepsilon^{-1} \underline{R} =  R =  \varepsilon \overline{R}
\end{align}
for a (small) $\varepsilon>0$ to be fixed below. 
To lighten notation, 
we write $v_{\underline{R}} := v \ast \Psi_{\underline{R} }$ and $f_{\underline{R}} := f\ast \Psi_{\underline{R}}$ for the regularised functions which satisfy the regularised problem 
\begin{equation}\label{eq:assumption_schauder_R}
(\partial_t - \Delta)v_{\underline{R}} = f_{\underline{R}}.
\end{equation} 
Then, on the parabolic cylinder $B_{\overline{R}}(z_0)$ of radius $\overline{R}$ around  $z_0$ consider the decomposition $v_{\underline{R}} = v^{>} + v^{<}$
as follows: $ v^{>} $ solves
\begin{align}
\label{eq:def_high_frequency_equation}
(\partial_t - \Delta ) v^{>}  &= f_{\underline{R}} \mathbf{1}_{B_{\overline{R}}(z_0)}   & & \text{on } \mathbb{R}_t \times \mathbb{R}_x^d,
\end{align}
or in other words, $v^{>}$ is given by space-time convolution of  $f_{\underline{R}} \mathbf{1}_{B_{\overline{R}}(z_0)} $ with a Gaussian heat kernel. 
Observe that $v^{<} = v_{\underline{R}} - v^{>}$ satisfies $(\partial_t - \Delta ) v^{<} = 0$ on $B_{\overline{R}}(z_0)$.
The rationale for the notation is that $ v^{>} $ should capture the high and $ v^{<}$ the low frequencies of $v_{\underline{R}}$ near $z_0$.

On the one hand, we have
\commentrb{Should the $B_R$ be a $B_{\overline{R}}$?}
\begin{equation}\label{eq:high-frequency-bound}
\| v^{>} \|_{B_{\overline{R}}(z_0)} \lesssim {\overline{R}}^2  \| f_{\underline{R}} \|_{B_R(z_0)} \lesssim {\underline{R}}^\alpha  \Big( \frac{{\overline{R}}^2}{ {\underline{R}}^{2}} \Big) \| f \|_{\alpha -2  }, 
\end{equation}
where the first inequality follows (for example) by calculating the $L^1$ norm of the heat kernel, restricted to a parabolic space-time ball of radius $\overline{R}$.
The second inequality follows from the definition of $f_{\underline{R}}$ and that of the negative regularity Besov norm \eqref{eq:def_global_Besov_norm}.

On the other hand,  by standard regularity properties of caloric functions (see e.g.~\cite[Theorem~8.4.4]{MR1406091}),
\begin{equation}\label{eq:space_derivative_bound_small_frequencies}
\| \partial_{x_i} v^{<} \|_{B_{R}(z_0)} \lesssim {\overline{R}}^{-1} \| v^{>} -v(z_0) \|_{B_{\overline{R}}(z_0)}
\end{equation}
and 
\begin{equation}\label{eq:time_derivative_bound_small_frequencies}
\| \partial_{t} v^{<} \|_{B_{R}(z_0)} \lesssim {\overline{R}}^{-2} \| v^{>} -v(z_0)\|_{B_{\overline{R}}(z_0)}.
\end{equation}

For $z \in B_{R}(z_0)$ we get  from the triangle inequality 
\begin{align}\label{eq:decomposition_v_increments}
|v(z) - v(z_0)|  \leq |v(z) - v_{\underline{R}}(z)| + |v^{>}(z) - v^{>}(z_0)  | + |v^{<}(z) - v^{<}(z_0)  | + |v_{\underline{R}}(z_0) - v(z_0)|  .
\end{align}
For the first and last term on the right-hand side of \ref{eq:decomposition_v_increments} we invoke \eqref{eq:regularisation_error} and bound
\begin{equation}\label{eq:Regularisation_error}
 |v(z) - v_{\underline{R}}(z)| +  |v(z_0) - v_{\underline{R}}(z_0)| \leq 2 {\underline{R}}^{\alpha} [v]_{\alpha},
\end{equation}
while for the second term we use \eqref{eq:high-frequency-bound}
\begin{equation}\label{eq:high_frequency_error}
 |v^{>}(z) - v^{>}(z_0)  | \leq 2 \| v^{>} \|_{B_{R(z_0)}} \lesssim \underline{R}^\alpha  \Big( \frac{{\overline{R}}^2}{ \underline{R}^{2}} \Big) \| f \|_{\alpha -2  } ,
 \end{equation}
and for the third we write 
\begin{align}
\notag
 |v^{<}(z) - v^{<}(z_0)  | &\leq |x-x_0|  \| \partial_{x_i} v^{<} \|_{B_{R}(z_0)}  + |t-t_0|  \| \partial_{t} v^{<} \|_{B_{R}(z_0)} \\
 \notag
 & \lesssim  \frac{R}{\overline{R}}  \| v^{<} -v(z_0)  \|_{B_{\overline{R}}(z_0)} +  \frac{R^2}{{\overline{R}}^2}  \| v^{<} -v(z_0)  \|_{B_{\overline{R}}(z_0)}\\
 \notag
 & \lesssim \frac{R}{\overline{R}} \Big( \| v_{\underline{R}}  -v(z_0)\|_{B_{\overline{R}}(z_0)} +\|v^{>}\|_{B_{\overline{R}}(z_0)}\Big) \\
  & \lesssim \frac{R}{\overline{R}} \Big( \| v -v(z_0) \|_{B_{\overline{R}}(z_0)}  + {\underline{R}}^\alpha [v]_{\alpha} +  \underline{R}^\alpha  \Big( \frac{{\overline{R}}^2}{ T^{2}} \Big) \| f \|_{\alpha -2 }  \Big) \;,
  \label{eq:low_frequency_error}
\end{align}
where in the first line  $x,x_0$ and $t, t_0$ refer to the space and time components of the space-time points $z,z_0$ respectively, where in the second inequality  we invoke  \eqref{eq:space_derivative_bound_small_frequencies} 
and  \eqref{eq:time_derivative_bound_small_frequencies} as well as the fact that $z \in B_{R}(z_0)$ by assumption, where we  used $\frac{R}{\overline{R}} = \varepsilon \leq 1$ and the triangle inequality for $v^{<} = v_T - v^{<}$ in the third inequality, 
and finally \eqref{eq:high-frequency-bound} in the last inequality.

Summarising \eqref{eq:decomposition_v_increments}, \eqref{eq:Regularisation_error}, \eqref{eq:high_frequency_error} and \eqref{eq:low_frequency_error} and using once more that $\frac{R}{\overline{R}}\leq 1$ to absorb the term involving $f$ on the right-hand side of \eqref{eq:low_frequency_error} we get
\begin{align}
|v(z) - v(z_0)|
&\lesssim  {\underline{R}}^\alpha [v]_{\alpha} + {\underline{R}}^\alpha  \Big( \frac{{\overline{R}}^2}{ {\underline{R}}^{2}} \Big) \| f \|_{\alpha -2  } +  \frac{R}{\overline{R}}  \| v_{\underline{R}} -v(z_0) \|_{B_{\overline{R}}(z_0)}  \nnb
&=  {\underline{R}}^\alpha [v]_{\alpha} +   \varepsilon^{-4} {\underline{R}}^\alpha \| f \|_{\alpha -2  } +  \varepsilon  \| v_{\underline{R}} -v(z_0) \|_{B_{\overline{R}}(z_0)}  .
\end{align}
Multiplying by $R^{-\alpha}$, taking the supremum over $z \in B_R(z_0)$ and recalling the choices \eqref{eq:choice_R_and_T} for ${\underline{R}}$ and $\overline{R}$ depending on $R$, one obtains
\begin{equation}
V(z_0,R)  \lesssim  \varepsilon^\alpha [v]_{\alpha} + \varepsilon^{-4+\alpha}\| f \|_{\alpha -2  }   +  \varepsilon^{1-\alpha} \frac{1}{{\overline{R}}^\alpha}  \| v -v(z_0) \|_{B_{\overline{R}}(z_0)}  . 
\end{equation}
Finally, taking the supremum over $z_0 \in \R_t \times \R^d_x$, invoking \eqref{eq:different_characterisations_Hoelder} we obtain 
\begin{equation}
[v]_{\alpha} \lesssim \varepsilon^\alpha [v]_{\alpha}   + \varepsilon^{-4+\alpha}\| f \|_{\alpha -2  }   +  \varepsilon^{1-\alpha} [v]_{\alpha} ,
\end{equation}
so that by choosing $\varepsilon>0$ small enough, recalling that $\alpha \in (0,1)$ and using the a priori assumption that $[v]_{\alpha }< \infty$, the desired estimate \eqref{eq:global_Hoelder_estimate} follows. 
\end{proof}

The following corollary is obtained by post-processing the estimate \eqref{eq:global_Hoelder_estimate} to the form we will actually use. 

\begin{corollary}[Local Schauder estimate]\label{Cor:local_Schauder_estimate}
For $z_0 \in \R_t \times \R_x^d$ and $R>0$ assume that $v$ is a continuous function that satisfies \eqref{eq:assumption_schauder} on $B_{5R}(z_0)$ in the distributional sense.  Fix $\alpha \in (0,1)$ and $\kappa \in (0,1)$ small enough. Then  
\begin{equation}
[v]_{\alpha,B_R(z_0)} \lesssim R^{2-\alpha - 3\kappa}  \| f\|_{-3\kappa, B_{5R}(z_0)} + R^{-\alpha}\| v \|_{B_{5R}(z_0)}.
\end{equation}
The implicit constant depends only on $\alpha$ and $d$.
\end{corollary}

\begin{proof}
Without loss of generality we assume $z_0 =0$. Throughout the proof we omit the argument in balls around $0$, i.e. for $R >0$ we write $B_R$ instead of $B_R(0)$. 
Let $\eta \in C^{\infty}(\R_t \times \R_x^d)$ be a cut-off function with the following properties:
\begin{itemize}
\item $\eta \equiv 1$ on $B_R$ and $\eta(z) =0$ for $z = (t,x)$ with $t \leq  -4 R^2$ or with $|x| \geq 2R$.
\item We have the estimates  $\| \partial_x \eta \| \lesssim \frac{1}{R}$, $ \|\Delta\eta\|\lesssim \frac{1}{R^2}$ and $\| \partial_t \eta \| \lesssim \frac{1}{R^2}$.
\end{itemize} 
Then the function $v \eta $ (naturally defined as $\equiv 0$ outside of $B_{2R}$) satisfies
\begin{align}
\notag
( \partial_t - \Delta ) (v \eta) &= \eta ( \partial_t - \Delta ) v   + v ( \partial_t - \Delta )  \eta - 2 \nabla v \cdot \nabla \eta \\
\label{Cor:local_Schauder_decomposition_RHS_localised_v}
& = \eta f  + v  (\partial_t - \Delta )  \eta - 2 \nabla v \cdot \nabla \eta 
\end{align}
on all of $ (-\infty, 0)  \times \R_x^d$. Here  $\nabla$ refers to the spatial gradient and $\cdot$ is the canonical scalar product on $\R^d_x$. 
In order to apply the global Schauder bound from Lemma~\ref{lem:Schauder_Lemma} we estimate the $\bbL^\infty$ norm of the individual terms on the right-hand side convolved with $\Psi_L$ for $L>0$.
The argument splits into the cases $L\leq R$ and $L > R$. For $L \leq R$ the function  $\eta f \ast \Psi_L$ is supported in $B_{3R}$ and for $\bar{z} \in B_{3R}$ we have by~\eqref{e:reconstruction2}
\begin{equation}\label{Cor:local_Schauder_estimate_First_Term_RHS1}
| \eta f  \ast \Psi_L |(\bar{z})   \leq  (L^\alpha [\eta]_\alpha + \|\eta \| ) \| f \|_{-3\kappa, B_{2L}(\bar{z}) } L^{-3\kappa} 
\lesssim  \| f \|_{-3\kappa, B_{5R} } L^{-3\kappa}   ,
\end{equation}
where we use that by assumption on $\eta$ we have $(R^\alpha [\eta]_\alpha + \|\eta \| ) \lesssim 1$.
 For $L \geq R$ we apply \eqref{e:reconstruction3}:
\begin{equation}
  \| \eta f  \ast \Psi_L \| 
  \lesssim  (R^\alpha [\eta]_\alpha + \|\eta \| ) \| f \|_{-3\kappa, B_{5R}}  R^{-3\kappa}\frac{R^{2+d}}{L^{2+d}}.
  \label{Cor:local_Schauder_estimate_First_Term_RHS2}
\end{equation}
We can combine the estimates \eqref{Cor:local_Schauder_estimate_First_Term_RHS1} and \eqref{Cor:local_Schauder_estimate_First_Term_RHS2} into 
\begin{equation}\label{Cor:local_Schauder_estimate_First_Term_RHS3}
\| \eta f  \ast \Psi_L \| \lesssim R^{2 - \alpha - 3\kappa} \| f \|_{-3\kappa, B_{5R}} L^{\alpha - 2} .
\end{equation}
For the remaining terms a similar splitting into $L \leq R $ and $L > R$ is done. 
For the second term on the right-hand side of \eqref{Cor:local_Schauder_decomposition_RHS_localised_v} this yields
\begin{align}
\notag
\|  v  (\partial_t - \Delta )  \eta  \ast \Psi_L \| & \lesssim \|  v  (\partial_t - \Delta )  \eta  \|_{B_{2R}}  \min \Big\{1 , \frac{R^{2+d}}{L^{2+d}}   \Big\} \\
\notag
&\lesssim R^{-2} \| v \|_{B_{2R}}   \min \Big\{1 , \frac{R^{2+d}}{L^{2+d}}  \Big\} \\
\label{Cor:local_Schauder_estimate_Second_Term_RHS}
& \lesssim R^{-\alpha} \| v \|_{B_{2R}
}  L^{\alpha-2}.
\end{align}
Finally, for the third term on the right-hand side of \eqref{Cor:local_Schauder_decomposition_RHS_localised_v}  we write
\begin{align}
\notag
\| (\nabla v \cdot \nabla \eta ) \ast \Psi_L \| &\leq   \sum_{j=1}^d \|  v \cdot \partial_j \eta  \ast  \partial_j \Psi_L \| +   \|  v \Delta \eta  \ast  \Psi_L \| \\
\notag
&\lesssim \| v \|_{B_{2R}}  \Big(R^{-1}L^{-1} + R^{-2}  \Big)  \min \Big\{1 , \frac{R^{2+d}}{L^{2+d}}  \Big\}  \\
\label{Cor:local_Schauder_estimate_Third_Term_RHS}
& \lesssim R^{-\alpha} \| v \|_{B_{2R}} L^{\alpha-2}.
\end{align}
Combining \eqref{Cor:local_Schauder_estimate_First_Term_RHS3}, \eqref{Cor:local_Schauder_estimate_Second_Term_RHS}, \eqref{Cor:local_Schauder_estimate_Third_Term_RHS} and  recalling Lemma~\ref{lem:Schauder_Lemma} (and being a bit more generous in the $\bbL^\infty$ norm of $v$) the desired conclusion follows.
\end{proof}

\subsection{Small scale estimates for remainder equation}

We assume we are given a continuous function $v$ that satisfies the remainder equation~\eqref{eq:remainder_equation}  in the distributional sense  on some space-time cylinder $B$. 
The main aim of this subsection is to establish the following interior regularity estimate (Corollary \ref{cor:interior_regularity_estimate}) that permits to control a local high regularity ($\alpha$-H\"older) norm 
in terms of a low regularity ($\bbL^\infty$) norm of $v$ as well as distributional norms of $Z_1$, $Z_2$ and $Z_3$.

\begin{corollary}[Interior regularity estimate]\label{cor:interior_regularity_estimate}
Let $v $ be a continuous function on an open set $B \subseteq \R_t \times \R^d_x$ that solves \eqref{eq:remainder_equation} on $B$ for given space-time distributions $Z_1$, $Z_2$, $Z_3$. 
 Let $\alpha>2\kappa>0$ be sufficiently small, 
and let $R^*>0$ be small enough to ensure that 
\begin{equation}\label{eq:smallness-condition}
(R^*)^{2-\kappa} \| v \|_B  \|Z_1\|_{-\kappa,B}  + (R^*)^{2-2\kappa} \| Z_2 \|_{-2\kappa,B} \ll 1.
\end{equation}
Then:
\begin{align}
\notag
  \sup_{R \leq R^*} \sup_{z \colon B_{2R}(z) \subseteq B} R^\alpha [v]_{\alpha,B_R(z)}
  &\lesssim  (R^*)^{2} \| v \|_{B}^3 + (R^*)^{2-\kappa} \| v \|_{B}^2  \| Z_1 \|_{-\kappa, B}  \\
& + (R^*)^{2 - 2\kappa} \| v \|_B   \| Z_2 \|_{-2\kappa,B}    + (R^*)^{2 - 3\kappa }\| Z_3 \|_{-3\kappa,B} 
   + \| v \|_{B}. \label{eq:main_interior_regularity_estimate}
\end{align}
\end{corollary}

The proof relies on the multiplicative inequality \eqref{e:reconstruction}
and the local Schauder estimate of Corollary~\ref{Cor:local_Schauder_estimate} and is closely related to the ``local-in-time well-posedness theory'' of \cite{MR2016604},
albeit with taking some care of spatial dependency.

\begin{proof}
We fix a space-time point $z_0$ and a scale $R$. We aim to apply the local Schauder estimate, Corollary \ref{Cor:local_Schauder_estimate} to $v$ satisfying \eqref{eq:remainder_equation}. 
For  any space-time point $z$ and $L >0$  we calculate
\begin{align}
| (\partial_t - \Delta ) v \ast \Psi_L | (z)   &=  | (- \lambda v^3 -3v^2 Z_1 - 3vZ_2 - Z_3) \ast \Psi_L | (z)  \nnb
&\lesssim \| v \|^3_{B_L(z)}  + \| v \|_{B_{2L}(z)} \big(   [v]_{\alpha, B_{2L}(z) } L^\alpha + \| v \|_{B_{2L}(z)} \big)   \| Z_1 \|_{-\kappa, B_{2L}(z)} L^{-\kappa}  \nnb
& \qquad + \big(   [v]_{\alpha, B_{2L}(z) } L^\alpha + \| v \|_{B_{2L}(z)} \big) \| Z_2 \|_{-2\kappa, B_{2L}(z)} L^{-2\kappa}  + \| Z_3 \|_{-3\kappa, B_{2L}(z)} L^{-3 \kappa},
\end{align}
where we have made use of the elementary inequality $[v^2]_{\alpha,  B_{2L}(z)} \leq 2  [v]_{\alpha,  B_{2L}(z)} \| v\|_{ B_{2L}(z)}$.
Therefore,  Corollary \ref{Cor:local_Schauder_estimate} yields for $z_0$, $R$ with $B_{10R}(z_0) \subseteq B$
\begin{align}
[v]_{\alpha, B_{R}(z_0)}& \lesssim R^{2-\alpha - 3\kappa} \sup_{\substack{z,L\\ B_L(z)  \subseteq B_{5R}(z_0) } } L^{3\kappa}  \big|  (\partial_t - \Delta ) v \ast \Psi_L\big|(z) + R^{-\alpha}\| v \|_{B_{5R}(z)} \nnb
&\lesssim R^{2-\alpha}   \| v \|^3_{B_{5R}(z_0)}   
+  R^{2-\alpha-\kappa}       \| v \|_{B_{10R}(z_0)} \big(   [v]_{\alpha, B_{10R}(z_0) } R^\alpha + \| v \|_{B_{10R}(z_0)} \big)   \| Z_1 \|_{-\kappa, B_{10L}(z_0)}   \nnb
& \qquad +  R^{2-\alpha- 2\kappa}    \big(   [v]_{\alpha, B_{10R}(z_0) } R^\alpha + \| v \|_{B_{10R}(z_0)} \big) \| Z_2 \|_{-2\kappa, B_{10R}(z_0)}  
\nnb &\qquad +  R^{2-\alpha-3\kappa}    \| Z_3 \|_{-3\kappa, B_{10R}(z_0)} 
 + R^{-\alpha} \| v \|_{B_{10R}(z_0)}\nnb
 &=  [v]_{\alpha, B_{10R}(z_0) } \Big(   R^{2-\kappa}  \| v \|_{B_{10R}(z_0)} \| Z_1 \|_{-\kappa, B_{10R}(z_0)}    + R^{2- 2\kappa}  \| Z_2 \|_{-2\kappa, B_{10R}(z_0)}       \Big) \nnb
&\qquad + R^{2-\alpha}   \| v \|^3_{B_{5R}(z_0)}  
+   R^{2-\alpha-\kappa}       \| v \|^2_{B_{10R}(z_0)}    \| Z_1 \|_{-\kappa, B_{10R}(z_0)}   \nnb
& \qquad +  R^{2-\alpha- 2\kappa}  \| v \|_{B_{10R}(z_0)}  \| Z_2 \|_{-2\kappa, B_{10R}(z_0)}  
+  R^{2-\alpha-3\kappa}    \| Z_3 \|_{-3\kappa, B_{10R}(z_0)} 
\nnb
&\qquad  + R^{-\alpha} \| v \|_{B_{10R}(z_0)},
\end{align}
where the last step only consists of rearranging terms. 
Now for $\varepsilon>0$ to be fixed below, let $R^*$ be small enough to ensure that 
\begin{equation}\label{eq:smallness-conditionA}
(R^*)^{2-\kappa} \| v \|_B  \|Z_1\|_{-\kappa,B}  + (R^*)^{2-2\kappa} \| Z_2 \|_{-2\kappa,B} \leq  \varepsilon,
\end{equation}
where we recall that $B$ is the open set on which $v$ solves~\eqref{eq:remainder_equation}, 
and set 
\begin{equation}
Q(B) = \sup_{\substack{z,R \leq R^* \\B_{20R}(z)  \subseteq B}} R^{\alpha} [v]_{\alpha,B_R(z)}
\end{equation}
as well as 
\begin{align} 
F(B) &=  ({R^*})^{2}   \| v \|^3_{B}  
+  ( {R^*})^{2-\kappa}       \| v \|^2_{B}    \| Z_1 \|_{-\kappa, B}   
  +  ({R^*})^{2- 2\kappa}  \| v \|_{B}  \| Z_2 \|_{-2\kappa,  B} \nnb
 & \qquad +  (R^*)^{2-3\kappa}    \| Z_3 \|_{-3\kappa, B} 
 +  \| v \|_{B}.
\end{align}
Then we get 
\begin{align}\label{eq:absorption_bound}
Q(B) \lesssim \varepsilon  \sup_{\substack{z,R \leq R^* \\B_{20R}(z)  \subseteq B}} R^{\alpha} [v]_{\alpha,  B_{10R}(z)} + F(B).
\end{align}
It remains to observe that there exists a number $N$ (depending only on $d$) such that each of the balls $B_{ 10R}$ 
can be covered by $N$ parabolic balls  $B_{\frac{R}{3}} (z_i)$ for $i = 1, \ldots, N$ with $B_{10R}(z_i) \subseteq B$.
Therefore, by subadditivity of the $\alpha$-H\"older norm we get, for $z$ with $B_{20R}(z) \subseteq B$,
\begin{equation}
R^{\alpha} [v]_{\alpha,  B_{10R}(z)}  \leq R^{\alpha} \sum_{i=1}^N [v]_{\alpha, B_{\frac{R}{4} }(z_i)} \leq 3^{\alpha}  N Q(B),
\end{equation}
which turns \eqref{eq:absorption_bound} into 
\begin{align}\label{eq:absorption_bound2}
Q(B) \lesssim  \varepsilon Q(B) + F(B), 
\end{align}
and thus for $\varepsilon$ small enough $Q(B) \leq F(B)$. To turn this estimate into the desired form \eqref{eq:main_interior_regularity_estimate}, it only remains to replace the supremum over balls $B_R(z)$ such that $B_{20R}(z)$ is contained in $B$ into the supremum over balls  $B_R(z)$ such that $B_{2R}(z)$  is contained in $B$. This can be achieved easily by another covering argument that we omit.
\end{proof}

The statement of Corollary~\ref{cor:interior_regularity_estimate} takes a simpler form with  a specific choice of parameters and under the following assumption. 
\begin{assumption}\label{Assumptionc}
  Let $B$ be a parabolic cylinder, and let $c <1$, $\kappa<1$. We assume that 
\begin{equation}
  \| Z_1 \|_{-\kappa,B} \leq c^3 \| v \|_B^{1-\kappa}, \qquad
   \| Z_2 \|_{-2\kappa, B} \leq c^3 \| v \|_B^{2(1-\kappa)}, \qquad
  \| Z_3 \|_{-3\kappa, B} \leq c^3 \| v \|_B^{3(1-\kappa)}
                             .
\end{equation}
\end{assumption}
\begin{corollary}\label{cor:simplified_Schauder}
Let $B $ be a parabolic cylinder, assume Assumption~\ref{Assumptionc} for given $c \ll 1$, $\kappa \ll 1$, and let 
\begin{equation}
R^* =  \| v \|_B^{-1}.
\end{equation}
Then, for $\alpha>0$ sufficiently small, for all $R\leq R^*$ and all $z \in B$ for which $B_{2R}(z) \subseteq B$,
\begin{equation}
R^{\alpha }[v]_{\alpha, B_R(z)} \lesssim \| v \|_{B},
\end{equation}
where the implicit constant only depends on $\alpha$ and $d$.
\end{corollary}
\begin{proof}
We first verify that condition \eqref{eq:smallness-conditionA} holds. Indeed, we have
\begin{align}
  &(R^*)^{2-\kappa} \| v \|_B \|Z_1\|_{-\kappa,B}  + (R^*)^{2-2\kappa} \| Z_2 \|_{-2\kappa,B}
    \nnb
 &\leq  \| v  \|_B^{\kappa -2} \| v \|_B c^3 \| v \|_B^{1-\kappa} +   \| v \|_B^{2\kappa-2}  c^3  \| v \|_B^{2- 2\kappa }  
= 2  c^3 , 
\end{align}
so that \eqref{eq:smallness-conditionA} is satisfied for $c \ll 1$. We can thus invoke \eqref{eq:main_interior_regularity_estimate} which takes the form
\begin{equation}
   R^\alpha [v]_{\alpha,B_R(z)}
  \lesssim   \| v \|_{B} +  c^3 \| v \|_{B} +  c^3   \| v \|_B  
    + c^3 \| v \|_{B} 
    + \| v \|_{B},
    \label{eq:main_interior_regularity_estimate-bis}
 \end{equation}
 uniformly for $R \leq R^*$ and $z$ such that $B_{2R}(z) \subseteq B$.
\end{proof}

\subsection{Large scale estimates for remainder equation}\label{sec_largescale_remainder_eq}

The following lemma is essentially identical to \cite[Lemma  2.7]{MR4164267} (where the constant $\lambda$ is set to $1$) or \cite[Theorem 4.4]{MoinatWeberEJP} (specialised to the case $f(u) = \lambda u^3$ in the notation of that paper). It
is a consequence of the maximum principle.

\begin{lemma}[Maximum principle]\label{lem:maximum_principle}
Let $R_0, \tilde{R} >0$, $\lambda>0$ and let $z_0 \in \R_t \times \R_x^2$. Assume that $v$ satisfies 
\begin{equation}
(\partial_t - \Delta)v  + \lambda v^3 = f
\end{equation} 
 on $B_{R_0+\tilde{R}}(z_0)$, where $f$ is a bounded function.
Then 
\begin{equation}\label{eq:maximum-principle-estimate}
\| v \|_{B_{R_0}(z_0)} \lesssim \max \Big\{ \frac{1}{\tilde {R}}, \| f \|^{\frac13}_{B_{R_0+\tilde R}(z_0)}    \Big\} ,
\end{equation}
where the implicit constant depends on $\lambda$.
\end{lemma}

In order to apply Lemma~\ref{lem:maximum_principle} to $v$ satisfying  equation~\eqref{eq:remainder_equation} we regularise the equation by convolving with $\Psi_L$ at a suitable scale $L$ leading to
\begin{equation}\label{eq:remainder_equation_regularised}
(\partial_t - \Delta) v*\Psi_L + \lambda (v*\Psi_L)^3 = \lambda \big( (v*\Psi_L)^3 - (v^3)*\Psi_L\big)  + (-3v^2 Z - 3vZ_2 - Z_3)*\Psi_L ,
\end{equation}
where the extra term $\lambda ((v*\Psi_L)^3 - (v^3)*\Psi_L) $ accounts for the fact that regularisation and cubing do not commute.
The following lemma follows by applying the maximum principle estimate (Lemma~\ref{lem:maximum_principle}) to this equation, using the multiplicative inequality \eqref{e:reconstruction} 
to bound $v^2 Z_1*\Psi_L$ and $3vZ_2*\Psi_L$
and the interior regularity estimate (Corollary~\ref{cor:simplified_Schauder}) to bound local $C^{\alpha}$ norms of $v$ by local $\bbL^\infty$ norms.

The lemma is the key step in the proof of (the $\bbL^\infty$ part of) Theorem~\ref{thm:a_priori_estimates_v}.
Indeed, it states that either the $\bbL^\infty$ norm of $v$ on a set $B_R$ can be controlled by the distributional norms of $Z_1$, $Z_2$, $Z_3$ in which case the desired estimate holds automatically,
or its $\bbL^\infty$ norm shrinks by a factor $<1$ at distance $\overline{R}$ away from the boundary.
This is a non-linear damping estimate because $\overline{R}$ is itself proportional to the inverse of the  $\bbL^\infty$ norm of $v$,
i.e., it becomes smaller as $v$ becomes large.

\begin{lemma}\label{lem:recursion_input}
  Assume that $v$ satisfies \eqref{eq:remainder_equation} on some parabolic cylinder $B = B_{R}(z_0)$.
Let $\alpha>2\kappa>0$ be small enough and 
  let Assumption~\ref{Assumptionc} be satisfied with $c \ll 1$, $\kappa \ll 1$.
  For $\varepsilon \ll 1$,
  let $\overline{R} = \varepsilon^{-1} \| v \|_{B_R(z_0)}^{-1}$ and assume $\overline{R} < R$. Then
\begin{equation}\label{eq:non-linear-recursion-estimiate}
\| v \|_{B_{R - \overline{R}}(z_0) }  \lesssim (\varepsilon^{\frac{\alpha}{3}} + c)   \| v \|_{B_R(z_0)}.
\end{equation}
\end{lemma}

\begin{proof}
  To shorten notation we omit the argument $z_0$ in balls, i.e., we write $B_R$ for $B_R(z_0)$, write $B_{R - \overline{R}}$ for $B_{R - \overline{R}}(z_0)$, and so on.
  Also write $u_L=(u)_L=u*\Psi_L$ throughout this proof, where
  \begin{equation}
    L = \frac{\varepsilon}{2}  \| v \|_{B_R}^{-1},
  \end{equation}
  and consider the regularised equation \eqref{eq:remainder_equation_regularised}. 
  In order to apply  \eqref{eq:maximum-principle-estimate} to this equation, it is useful to 
 leave a bit of space to the boundary of $B_R$ and apply the estimate on the ball $B_{R-\underline{R}}$ where $\underline{R} = \| v \|_{B_R}^{-1}$,
 noting that for any $\bar{z} \in B_{R-\underline{R}}$ we have, 
by Corollary~\ref{cor:simplified_Schauder},
 \begin{equation}\label{eq:use_of_regularity_estimate}
   \underline{R}^\alpha [v]_{\alpha,B_{\underline{R}/2}(\bar{z})} \lesssim \| v \|_{B_R}
   ,\qquad
   L^\alpha [v]_{\alpha,B_{L}(\bar{z})} \lesssim \| v \|_{B_R}.
 \end{equation}
 Applying \eqref{eq:maximum-principle-estimate} with $R_0=R-\overline{R}$ and $\tilde R = \overline{R}-\underline{R}$, we arrive at
\begin{align}
  \| v_L \|_{B_{R - \overline{R}} }  \lesssim
  \max \bigg\{ \frac{1}{\overline{R} - \underline{R} },  \,  \|  (v_L)^3 - (v^3)_L \|_{B_{R-\underline{R} }}^{\frac13}, \, & \| (v^2 Z_1)_L  \|_{B_{R-\underline{R}}}^{\frac13} , \,  \nnb
&  \| (vZ_2)_L \|_{B_{R-\underline{R}}}^{\frac13}, \, \| ( Z_3)_L \|_{B_{R-\underline{R}}}^{\frac13} \bigg\} .
\end{align}
By assumption we have $(\overline{R} - \underline{R})^{-1}  \leq  2\overline{R}^{-1}  = 2 \varepsilon \| v \|_{B_R}$,
thus bringing the first term on the right-hand side into the desired form. 
For the second term we use the simple commutator estimate 
\begin{align}
\| (v_L)^3 - (v^3)_L \|_{B_{R-\underline{R}}} & \lesssim  L^\alpha \sup_{\bar{z} \in B_{R-\underline{R}}}[v]_{\alpha, B_{L}(\bar{z})} \| v \|_{B_{L}(\bar{z})}^2   \nnb
&\lesssim\Big( \frac{ L}{\underline{R}} \Big)^\alpha  \sup_{\bar{z} \in B_{R-\underline{R}}} \underline{R}^\alpha [v]_{\alpha, B_{\underline{R}/2}(\bar{z})} \| v \|_{B_{R}}^2   \lesssim  \varepsilon^\alpha \| v \|^3_{B_R  },
\end{align}
where in the last inequality we have used \eqref{eq:use_of_regularity_estimate} as well as
that $L/\underline{R} =\epsilon/2$ by the definitions of $L$ and $\underline{R}$.
Using the multiplicative inequality \eqref{e:reconstruction} to bound $(v^2 Z_1)_L$ and $(vZ_2)_L$ together with $R-\underline{R}+2L\leq R$, we arrive at 
\begin{align}
\| v_L \|_{B_{R - \overline{R}} }   \lesssim& \max \bigg\{ (\varepsilon + \varepsilon^{\frac{\alpha}{3}} )\| v \|_{B_R} , \nnb 
& \Big(   \|v \|_{B_{R}} \Big(  \sup_{\bar{z} \in B_{R-\underline{R}}}[v]_{\alpha, B_{ 2L}(\bar{z})} L^\alpha  + \| v \|_{{B_{R}}}  \Big) \| Z_1 \|_{-\kappa,B_R} L^{-\kappa} \Big)^{\frac13}, \nnb
&\Big( \big(  \sup_{\bar{z} \in B_{R-\underline{R}}}[v]_{\alpha, B_{2L}(\bar{z})} L^\alpha  + \| v \|_{ B_R
}  \big) \| Z_2 \|_{-2\kappa,B_R}L^{-2\kappa}\Big)^{\frac13}, \Big( \| Z _3\|_{-3\kappa,B_R} L^{-3 \kappa} \Big)^{\frac13} \bigg\},
\end{align}
which by using Assumption~\ref{Assumptionc} to bound the norms of $Z_1$, $Z_2$, $Z_3$ in terms of powers of $\| v \|_{B_R}$ and Corollary~\ref{cor:simplified_Schauder} to replace $ \sup_{\bar{z} \in B_{R-\underline{R}}}[v]_{\alpha, B_{ 2L}(\bar{z})} L^\alpha $ by $\| v \|_{B_R}$
 turns into 
\begin{equation}
  \| v_L \|_{B_{R - \overline{R}} }  \lesssim (\varepsilon^{\frac{\alpha}{3}} + c) \| v \|_{B_R}.
\end{equation}
 To obtain the desired conclusion it only remains to invoke \eqref{eq:use_of_regularity_estimate} one final time to see 
\begin{equation}
  \Big| \|v \|_{B_{R - \overline{R}}  } -  \|v_L \|_{B_{R - \overline{R}} } \Big|
  \leq \|v -v_L \|_{B_{R - \overline{R}} }
  \leq \sup_{\bar{z} \in B_{R-\underline{R}}}[v]_{\alpha, B_{L}(\bar{z})} L^\alpha  \lesssim \varepsilon^\alpha   \| v \|_{B_R(z_0)},
\end{equation}
which concludes the argument.
\end{proof}

\def\reference{\ref{thm:a_priori_estimates_v} }
\begin{proof}[Proof of Theorem~\reference]
Let 
\begin{equation}
\mathcal{Z} :=  \max \Big\{   \|Z_1 \|^{\frac{1}{1-\kappa}}_{-\kappa,B_2(0)}, \, \|Z_2 \|^{\frac{1}{2(1-\kappa)}}_{-2\kappa,B_2(0)} , \, \|Z_3 \|^{\frac{1}{3(1-\kappa)}}_{-3\kappa,B_2(0)} \Big\} .
\end{equation}
The non-linear recursion argument that permits to turn Assumption~\ref{Assumptionc} and  Lemma~\ref{lem:recursion_input}  into the $\bbL^\infty$ estimate 
\begin{equation}\label{eq:L-infty-estimate}
\| v \|_{B_{\frac32}(0)}   \lesssim 
\mathcal{Z}  +1,
\end{equation}
is identical to \cite[Section 4.6]{MR4164267} (see also \cite[Proof of Thm 1.4, Step 3]{bonnefoi2022priori})  and we do not reproduce it here. 
To obtain the bound on the $C^\alpha$ seminorm we invoke once more the interior regularity estimate Corollary~\ref{cor:interior_regularity_estimate} for $B = B_{\frac32}(0)$.
For $R^* =\min \{ \frac14, \varepsilon  \mathcal{Z}^{-1}  \}$ 
 and $\varepsilon>0$ small enough,  condition \eqref{eq:smallness-condition} is automatically satisfied. Indeed, for such $R^*$ we have using \eqref{eq:L-infty-estimate}
\begin{align}
&(R^*)^{2-\kappa} \| v \|_{B_{\frac32}(0)} \|Z\|_{-\kappa, B_{\frac32}(0)}  + (R^*)^{2-2\kappa} \| Z \|_{-2\kappa,B_{\frac32}(0)}  \nnb
&\lesssim \min\{  \varepsilon^{2-\kappa}  \mathcal{Z}^{-(2-\kappa)} , 1 \} (\mathcal{Z} +1) \|Z\|_{-\kappa, B_{\frac32}(0)}  + \varepsilon^{2-2\kappa}   \mathcal{Z}^{-(2-2\kappa)} \| Z \|_{-2\kappa, B_{\frac32}(0)}   \nnb
& \notag\lesssim \varepsilon^{2-\kappa} + \varepsilon^{1-\kappa}+  \varepsilon^{2-2\kappa}  . 
\end{align}
Thus for $\bar{z} \in B_{1}(0)$, using \eqref{eq:main_interior_regularity_estimate} and \eqref{eq:L-infty-estimate} to replace the powers of $\| v \|_{B_{\frac32}(0)}$ by $\mathcal{Z}+1$, we get
\begin{equation} 
  (R^*)^{\alpha }  [v]_{\alpha, B_{R^*}(\bar{z})} \lesssim \mathcal{Z}+1, 
\end{equation}
which yields the small-scale  bound 
\begin{equation}
\Big( \sup_{\substack{z \neq \bar{z} \in B_1(0) \\ d(z,\bar{z}) \leq R^{*}  }}   \frac{ |v(z) - v (\bar{z}) |}{d(z,\bar{z}) ^\alpha} \Big)^{\frac{1}{1-\alpha}} \lesssim \mathcal{Z}+1.
\end{equation}
For $z \neq \bar{z} \in B_1(0)$ with $d(z, \bar{z}) \geq R^*$ we simply write
\begin{equation}
\frac{|v(z) - v(\bar{z}) |}{d(z,\bar{z})^\alpha} \leq 2 (R^{*})^{-\alpha} \| v \|_{B_1(0)} \lesssim \mathcal{Z}^{1+\alpha}+1,
\end{equation}
which completes the argument. 
\end{proof}

\section{Proof of gradient bound on auxiliary diffusion}\label{app_gradientX}
In this section we prove Lemma~\ref{lemm_gradientX}. 
Let $\alpha>0$ and consider the equation:
\begin{equation}
\mathrm{d}\hat X^x_r
=
b(r,\hat X_r)\, \mathrm{d}r+\sqrt{2}\,\mathrm{d}W_r
,\qquad
\hat X^x_0=x\in\R^2
,
\label{eq_diff_appendix}
\end{equation}
where $b \in \mathbb L^\infty(\R_+,C^\alpha(\R^2,\R^2))$, 
defined as the space of $\R^2$-valued functions of space-time that have components with $C^\alpha$ norm bounded uniformly in time. 
Define more generally $C^{n+\alpha}(\R^2,\R^2)$ as the set of functions with bounded derivatives up to order $n$, 
with order $n$ derivatives having components in $C^\alpha(\R^2,\R^2)$.
\medskip

\noindent {\bf Notation.} In this appendix, H\"older norms will always be on the full space with no weight. 
To use similar notations as~\cite{MR2593276} we will write $\|\cdot\|_{C^\alpha}$ for the norm in $C^\alpha(\R^2)$. 
Abusing notations, 
we also write $\psi \in C^\alpha$ for a tensor $(\psi_i)_{i\in I}$ on some finite index set $I$ to mean that each component $\psi_i$ ($i\in I$) is in $C^\alpha(\R^2,\R)$, 
and $\|\psi\|_{C^\alpha}$ for the norm:
\begin{equation}
\|\psi\|_{C^\alpha}
:=
\big| (\|\psi_i\|_{C^\alpha})_{i\in I}\big|_2
=
\Big[\sum_{i\in I} \|\psi_i\|_{C^\alpha}^2\Big]^{1/2}
.
\end{equation}
For $\psi:\R_+\times\R^2\to\R^I$, we also write $\|\psi\|_0$ for the uniform space-time norm $\sup_{r\geq 0}\sup_{z\in\R^2}|\psi(r,z)|_2$. 
 
Expectation on a probability space on which all $\hat X^x_\cdot$ are defined ($x\in\R^2$) is denoted by $\Eg$, with $\Pg$ for the associated probability measure.
%
\begin{proposition}\label{prop_gradientX_appendix}
Let $\alpha\in(0,1)$ and $b \in \mathbb L^\infty(\R_+,C^\alpha(\R^2,\R^2))$. 
There is a constant $C(\alpha)>0$ independent of $b$ such that:
\begin{equation}
\forall t\geq 0,\qquad
\sup_{\substack{x,y\in\R^2 \\ x\neq y}}\frac{1}{|x-y|}
\sup_{u\leq t}\Eg\Big[|\hat X^x_{u}-\hat X^y_u|^2\Big]^{1/2}
\lesssim
\exp\Big[C(\alpha)\, t\, \big[1\vee  \sup_{u\leq t}\|b(u,\cdot)\|_{C^{\alpha}}^{\frac{2}{1-\alpha}}\big]\Big]
. 
\end{equation}
\end{proposition}
Proposition~\ref{prop_gradientX_appendix} is proven in~\cite[Theorem 5]{MR2593276}, except that the constants on the right-hand side are not made explicit there.
Therefore we sketch its proof below, 
following~\cite{MR2593276}, but keeping track of constants. 
Before we do so, we first prove Lemma~\ref{lemm_gradientX} assuming Proposition~\ref{prop_gradientX_appendix}.
\begin{proof}[Proof of Lemma~\ref{lemm_gradientX}]
Recall that $X^x_\cdot$ is the solution of the following SDE:
\begin{equation}
  \mathrm{d}X^x_u = -\nabla \psi_{u+s}(X^x_u)\, \mathrm{d}u + \sqrt{2}\,\mathrm{d} W_u, \qquad X^x_0=x
  ,
\end{equation}
with $W_\cdot$ a standard Brownian motion and $\psi$ defined in~\eqref{eq_def_psi}. 
Let $\ell,\ell'\in\N$ and $\ell^*=\max\{\ell,\ell'\}$. 
Let $\chi:\R^2\to[0,1]$ be a smooth function with compact support in $[-L^{\ell^*+1}/2-1,L^{\ell^*+1}/2+1]^2$, 
equal to $1$ on $[-L^{\ell^*+1}/2,L^{\ell^*+1}/2]^2$. 
With this definition $\chi(X^x_\cdot)$ is constant equal to $1$ on the following set 
(recall~\eqref{eq_def_I_ell} for the definition of $I^x_{\ell}(s,t)$):
\begin{equation}
\Big\{\sup_{u\leq t-s}|X^x_u|\leq L^{\ell+1}/2-4\Big\}
\supset 
I^x_\ell(s,t)
.
\end{equation}
%
Consider then $\hat X^x_\cdot$ with the same noise $\sqrt{2}W_\cdot$ as $X^x_\cdot$ but drift replaced by the truncated version:
\begin{equation}
b(u,y) 
= 
-\chi(y)\nabla\psi_{s+u\wedge t}(y)
,\qquad 
u\geq 0, y\in\R^2
.
\end{equation}
The definition~\eqref{eq_def_CtL} of $C_{t,L}(\ell^*)$ then implies:
\begin{equation}
t\sup_{u\leq t-s}\|b(u,\cdot)\|_{C^\alpha}^{\frac{2}{1-\alpha}}
\leq 
C_{t,L}(\ell^*)^{\frac{1}{1-\alpha}}
.
\end{equation}
The diffusion $\hat X$ is now of the form~\eqref{eq_diff_appendix}, and:
\begin{equation}
(X^x_{u})_{u\leq t-s}
=
(\hat X^x_u)_{u\leq t-s} \quad \text{ on }\quad \Big\{\sup_{u\leq t-s}|X^x_u|\leq L^{\ell^*+1}/2-1\Big\}
.
\end{equation}
This implies, recalling from~\eqref{eq_def_J_ell} that $J^{x,y}_{\ell,\ell'}(s,t) = I^x_{\ell}(s,t)\cap I^y_{\ell'}(s,t)$ and that the law $\Qg_{t-s}$ of the $X^z_\cdot$ ($z\in\R^2$) is defined in~\eqref{eq_def_law_aux_diffusion}:
\begin{align}
\sup_{u\leq t-s}\Eg_{\Qg_{t-s}}\Big[{\bf 1}_{J^{x,y}_{s,t}(\ell,\ell')}|X^x_{u}-X^y_u|^2\Big]^{1/2}
&\leq 
\sup_{u\leq t-s}\Eg_{\Qg_{t-s}}\Big[{\bf 1}_{\sup_{u\leq t-s}\max\{|X^x_u|,|X^y_u|\}\leq L^{\ell^*+1}/2-1} |X^x_{u}-X^y_u|^2\Big]^{1/2}
\nnb
&\leq 
\sup_{u\leq t-s}\Eg_{\Qg_{t-s}}\Big[|\hat X^x_{u}-\hat X^y_u|^2\Big]^{1/2}
.
\end{align}
Lemma~\ref{lemm_gradientX} thus indeed follows from Proposition~\ref{prop_gradientX_appendix} as claimed. 
\end{proof}
We now prove Proposition~\ref{prop_gradientX_appendix}. 
The conditions on $b$ ensure a solution to~\eqref{eq_diff_appendix} exists and can be written as $(\phi_{0,s}(x))_{s\geq 0}$, 
where $(\phi_{s,t}(x))_{0\leq s\leq t,x\in\R^2}$ is a stochastic flow of diffeomorphisms, see~\cite[Theorem 5]{MR2593276}. 
If $b$ were Lipschitz continuous, estimates on the space derivative of this flow would be classical. 
In the present H\"older continuous case, the proof in~\cite{MR2593276} involves a change of coordinate that maps the SDE~\eqref{eq_diff_appendix} to another, 
more regular SDE with Lipschitz drift and diffusion coefficients. 
Theorem 5 in~\cite{MR2593276} then states that, for any $p\geq 1$:
\begin{equation}
\sup_{x\in\R^2}\sup_{s\leq t}\Eg\big[|D\phi_{0,s}(x)|_{2}^p\big]
<\infty
.
\end{equation}
Since:
\begin{equation}
\sup_{\substack{x,y\in\R^2 \\ x\neq y}}\frac{1}{|x-y|}
\sup_{u\leq t}\Eg\Big[|\hat X^x_{u}-\hat X^y_u|^2\Big]^{1/2}
\leq 
\sup_{x\in\R^2}\sup_{u\leq t}\Eg\big[|D\phi_{0,u}(x)|^2_{2}\big]^{1/2}
,
\label{eq_bound_gradient_Dphi_appendix}
\end{equation}
this directly implies a non-quantitative version of Proposition~\ref{prop_gradientX_appendix}. 
We follow the argument of~\cite{MR2593276} step by step to get quantitative bounds. 
The first step is a bound in the case of an SDE with Lipschitz coefficients. 
This is classical (see e.g.~\cite[Theorem 3.1]{KunitaStFlour}), 
but we need explicit bounds so we provide a proof below.
\begin{lemma}[Lipschitz case]\label{lemm_SDE_Lipschitz_coeff}
Consider the following SDE on $\R^2$:
\begin{equation}
\mathrm{d}\tilde X^x_s 
=
\tilde b(s,\tilde X^x_s)\, \mathrm{d}u + \sqrt{2}\, \tilde\sigma(s,\tilde X^x_s)\, \mathrm{d}W_s,
\qquad 
\tilde X^x_0=x\in\R^2
.
\end{equation}
Assume that $\tilde b,\tilde\sigma\in\bbL^\infty(\R_+,C^{1+\alpha})$. 
There is then a family $(\varphi_{r,s}(x))_{0\leq r\leq s,x\in\R^2}$ of diffeomorphisms such that $\tilde X^x_s=\varphi_{0,s}(x)$ and $\varphi_{r,\cdot}\in\bbL^\infty(\R_+,C^{1+\beta})$ for each $\beta\in(0,\alpha)$ and each $r\geq 0$. 
In addition,
\begin{equation}
\sup_{s\leq t}\Eg\big[|\varphi_{0,s}(y)-\varphi_{0,s}(y')|^2\big]^{1/2}
\leq 
|y-y'|e^{3t(\|D\tilde b\|_0+2\||D\tilde \sigma|^2_2\|_0)}
,
\label{eq_bound_diffeo_Lipschitz_case}
\end{equation}
where we recall that $\|\cdot\|_0$ stands for the supremum in space and time.
\end{lemma}
\begin{proof}
We only prove the bound~\eqref{eq_bound_diffeo_Lipschitz_case}, 
referring to~\cite[Theorem 3.1]{KunitaStFlour} for the rest. 
By Ito's formula, 
\begin{align}
\varphi_{0,s}(y)-\varphi_{0,s}(y')
&=
y-y'+ \int_0^s [\tilde b(u,\varphi_{0,u}(y))-\tilde b(u,\varphi_{0,u}(y'))]\, du
\nnb
&\quad 
+\sqrt{2}\int_0^s [\tilde \sigma(u,\varphi_{0,u}(y))-\tilde \sigma(u,\varphi_{0,u}(y'))]\, dW_u
.
\end{align}
Taking squares and expectations and using  
the fact that $\tilde b,\tilde \sigma$ are Lipschitz gives:
\begin{align}
\Eg\big[\, |\varphi_{0,s}(y)-\varphi_{0,s}(y')\,|^2\big]
\leq 
3|y-y'|^2 + 3\int_0^s \big[\|D\tilde b\|_0| + 2\||D\tilde \sigma|^2_2\|_0\big]\, \Eg\big[ \, |\varphi_{0,u}(y)-\varphi_{0,u}(y')|^2\, \big]\, du
.
\end{align}
Gronwall inequality concludes the proof. 
\end{proof}
We now explain how to transform the SDE~\eqref{eq_diff_appendix} with irregular drift into a regular one of the form given in Lemma~\ref{lemm_SDE_Lipschitz_coeff}. 
Consider the vector-valued solution $\psi_\lambda$ of:
\begin{equation}
(\partial_t+ L^b-\lambda)\psi_\lambda 
= 
-b,
\qquad 
L^b
: u\in C^2(\R^2,\R)\mapsto 
\Delta u +b\cdot D u
.
\label{eq_pde_appendix}
\end{equation}
According to~\cite[Theorem 2.4]{MR2748616}, 
for any $\lambda\geq 1$, 
there is a unique solution $\psi_\lambda\in \mathbb L^\infty(\R_+,C^\alpha(\R^2,\R^2))$ of~\eqref{eq_pde_appendix}. 
It satisfies a Schauder estimate, 
stated as Theorem 2.4 in~\cite{MR2748616}. 
\begin{lemma}\label{lemm_Schauder}
There is $C(\alpha)>0$ such that the solution $\psi_\lambda$ of~\eqref{eq_pde_appendix} satisfies the following Schauder estimate:
\begin{equation}
\sup_{s\geq 0}\|\psi_\lambda(s,\cdot)\|_{C^{2+\alpha}}
\leq
C(\alpha)
\sup_{s\geq 0}\|b(s,\cdot)\|_{C^{\alpha}}
\label{eq_Schauder}
.
\end{equation}
\end{lemma}
\begin{remark}
The fact that the constant in the right-hand side of~\eqref{eq_Schauder} does not depend on $\lambda$ is not stated explicitly in~\cite{MR2748616} and not immediately evident in the argument, 
so we provide some details at the end of the section.
\end{remark}
Define also $\Psi_\lambda(s,x) = x+\psi_\lambda(s,x)$. 
We will use $\psi_\lambda,\Psi_\lambda$, 
more regular than $b$, 
to define a so-called conjugate SDE with regular coefficients associated with~\eqref{eq_diff_appendix}. 
To do so, we start with the following claim, bounding $D\psi_\lambda$ for large enough $\lambda$. 
\begin{lemma}\label{lemm_bound_Dpsi_lambda}
There is $c(\alpha)>0$ such that, if:
\begin{equation}
\lambda
=
\max\Big\{\big[4 c(\alpha)\sup_{s\geq 0}\|b(s,\cdot)\|_{C^\alpha}\big]^{\frac{2}{1-\alpha}},1\Big\}
,
\label{eq_def_lambda_appD}
\end{equation}
then:
\begin{equation}
\sup_{s\geq 0}\|D\psi_\lambda(s,\cdot)\|_{C^\alpha}\leq 1/3
.
\label{eq_Dpsi}
\end{equation}
\end{lemma}
Note that Lemma~\ref{lemm_bound_Dpsi_lambda} implies in particular that $\Psi_\lambda(t,\cdot)$ is a $C^2$ diffeomorphism at each time, 
with first and second space derivatives bounded uniformly in time by~\eqref{eq_Schauder}. 
Similarly, $\Psi_\lambda^{-1}$ has bounded first and second space derivatives uniformly in time (by $(1-\|D\psi_\lambda\|_0)^{-1}\leq 2$ for the first derivative). %
\begin{proof}[Proof of Lemma~\ref{lemm_bound_Dpsi_lambda}]
We follow the proof of~\cite[Lemma 4]{MR2593276}. 
If $P_u=e^{u\Delta}$ ($u\geq 0$) denotes the heat semigroup, 
the solution $\psi_\lambda$ has the following representation:
\begin{equation}
\psi_\lambda(s,x)
=
-\int_0^\infty e^{-\lambda u}P_u \big[c_\lambda(s+u,\cdot)\big](x)\, du
,\qquad 
c_\lambda(s,\cdot) 
:=
b(s,\cdot) + (b(s,\cdot),D\psi_\lambda(s,\cdot))
.
\end{equation}
Note that $c_\lambda(s,\cdot)$ is continuous and bounded.  Differentiating under the integral and using the existence of $c(\alpha)>0$ such that:
\begin{equation}
\|P_u g\|_{C^{1+\alpha}}
\leq
\frac{c(\alpha)}{u^{(1+\alpha)/2}} \|g\|_\infty
,\qquad
g\text{ continuous and bounded}
,
\label{eq_def_calpha_appD}
\end{equation}
we find:
\begin{align}
\sup_{s\geq 0}\|D\psi_\lambda(s,\cdot)\|_{C^\alpha}
&\leq 
c(\alpha)\sup_{s\geq 0} \int_0^\infty e^{-\lambda u}\|P_u c_\lambda(s+u,\cdot)\|_{C^{1+\alpha}}\, du
\nnb
&\leq 
\frac{c(\alpha) \sup_{s\geq 0}\|b(s,\cdot)\|_{C^\alpha}}{\lambda^{(1-\alpha)/2}}\big(1 + \sup_{s\geq 0}\|D\psi_{\lambda}(s,\cdot)\|_{0}\big)
.
\end{align}
Thus, when $\lambda$ satisfies~\eqref{eq_def_lambda_appD} with $c(\alpha)$ given by~\eqref{eq_def_calpha_appD},
\begin{align}
\sup_{s\geq 0}\|D\psi_{\lambda}(s,\cdot)\|_{C^\alpha}
&\leq 
\frac{c(\alpha)\sup_{s\geq 0}\|b(s,\cdot)\|_{C^\alpha}}{\lambda^{(1-\alpha)/2}}\frac{1}{1-\frac{c(\alpha)\sup_{s\geq 0}\|b(s,\cdot)\|_{C^\alpha}}{\lambda^{(1-\alpha)/2}}}
\nnb
&\leq 
\frac{1}{3}
\quad \text{ when }\quad 
\lambda=\max\Big\{\big[4 c(\alpha)\sup_{s\geq 0}\|b(s,\cdot)\|_{C^\alpha}\big]^{\frac{2}{1-\alpha}},1\Big\}
.
\label{eq_Dpsi_bis}
\end{align}
\end{proof}
In the following we just write $\psi,\Psi$ for $\psi_\lambda,\Psi_\lambda$, 
with $\psi_s,\Psi_s = \psi(s,\cdot),\Psi(s,\cdot)$.   
Define now a diffusion $\tilde X^y$ as follows:
\begin{equation}
\tilde X^y_t 
=
y + \int_s^t\tilde b(u,\tilde X^y_u)\, du +  \sqrt{2}\int_{s}^t\tilde \sigma(u,\tilde X^y_u)\, dW_u
,
\label{eq_SDE_tildeX_qt_bd}
\end{equation}
where the drift $\tilde b$ and diffusion matrix $\tilde \sigma$ are given by:
\begin{equation}
\tilde b(t,y)
=
-\lambda\psi(t,\Psi^{-1}(t,y)),\qquad 
\tilde\sigma(t,y)
=
D\Psi(t,\Psi^{-1}(t,y))
.
\end{equation}
Note that $\tilde b,\tilde \sigma$ are continuous in time and in $\mathbb L^\infty(\R_+,C^{1+\alpha})$. 
In particular they are Lipschitz in space uniformly in time and, 
as $D\Psi=\id + D\psi$ by definition and $\|D\psi\|_0\leq 1/3$ by Lemma~\ref{lemm_bound_Dpsi_lambda}:
\begin{align}
\|D\tilde b\|_{0}
&\leq 
\lambda  \big\|D\psi(t,\Psi^{-1}(t,y)) \big[D\Psi(t,\Psi^{-1}(t,y))\big]^{-1}\big\|_{0}
\leq 
\lambda
,\nnb
\|D\tilde\sigma\|_0
&\leq 
\frac{3}{2}\| D^2\Psi\|_0
\leq
\frac{3C(\alpha)}{2}\sup_{s\geq 0}\|b(s,\cdot)\|_{C^\alpha} 
,
\label{eq_bound_bsigma_appendix}
\end{align}
where the second line is the Schauder estimate~\eqref{eq_Schauder}.

Since $\tilde X$ now has Lipschitz drift and diffusion coefficient, 
its solution can be written in terms of diffeomorphisms $(\varphi_{0,s}(x))_{s\geq 0,x\in\R^2}$, 
for which Lemma~\ref{lemm_SDE_Lipschitz_coeff} and the definition of $\lambda$ in Lemma~\ref{lemm_bound_Dpsi_lambda} give the following bound: 
for some constant $C'(\alpha)>0$, 
\begin{equation}
\sup_{s\leq t}\Eg\big[|\varphi_{0,s}(y)-\varphi_{0,s}(y')|^2\big]^{1/2}
\leq 
3|y-y'| \, e^{t\, C'(\alpha) [1\vee \sup_{s\geq 0}\|b(s,\cdot)\|^{2/(1-\alpha)}_{C^\alpha}]}
.
\end{equation}
From this we deduce a similar estimate for $\tilde X^y_t-\tilde X^{y'}_t = \phi_{0,t}(y)-\phi_{0,t}(y')$, 
where we recall that $\phi_{0,t}(y)$ are the diffeomorphisms appearing in~\eqref{eq_bound_gradient_Dphi_appendix}.  
Indeed, 
$\tilde X^y_t$ solves~\eqref{eq_SDE_tildeX_qt_bd} with initial condition $y=\Psi(s,x)$ whenever $X$ solves~\eqref{eq_diff_appendix}, or in other words:
\begin{equation}
\phi_{0,s}
=
\Psi_{s}^{-1}\circ\varphi_{0,s}\circ\Psi_0,\qquad 
\Psi_0(x) =x 
.
\end{equation}
Thus:
\begin{align}
\Eg\big[|\phi_{0,s}(y)-\phi_{0,s}(y')|^2\big]^{1/2}
&\leq
\|D(\Psi^{-1})\|_0\, \Eg\big[|\varphi_{0,s}(y)-\varphi_{0,s}(y')|^2\big]^{1/2}
\nnb
&\leq 
\frac{3}{2}\Eg\big[|\varphi_{0,s}(y)-\varphi_{0,s}(y')|\big]
\nnb
&\leq 
\frac{9}{2}\, |y-y'| \, \exp\Big[sC'(\alpha)\big[1\vee  \sup_{u\geq 0}\|b(u,\cdot)\|^{2/(1-\alpha)}_{C^\alpha}\big]\Big]
.
\end{align}
This concludes the proof of Proposition~\ref{prop_gradientX_appendix} assuming the Schauder estimate of Lemma~\ref{lemm_Schauder}, 
proven next. 
\begin{proof}[Proof of Lemma~\ref{lemm_Schauder}]
We follow the proof of Theorem 2.4 in~\cite{MR2748616}, 
keeping track of constants. 
Let $\psi$ denote a solution of~\eqref{eq_pde_appendix}. 
Recall the following interpolation inequality. 
for any $\epsilon>0$, 
there is $C(\epsilon,\alpha)>0$ such that:
\begin{equation}
\|\psi(s,\cdot)\|_{C^{2}}
\leq 
\epsilon\|\psi(s,\cdot)\|_{C^{2+\alpha}}  + C(\epsilon,\alpha)\|\psi(s,\cdot)\|_{\infty}
.
\end{equation}
Then $\epsilon=1/2$ and $\|\psi(s,\cdot)\|_{C^{2+\alpha}}\leq \|\psi(s,\cdot)\|_{C^{2}} + \|D^2\psi(s,\cdot)\|_{C^\alpha}$ give ($C(\alpha):=C(1/2,\alpha)$):
\begin{align}
\sup_{s\geq 0}\|\psi(s,\cdot)\|_{C^{2+\alpha}}
&\leq
C(\alpha)\|\psi\|_0 + 2\sup_{s\geq 0}\|D^2\psi(s,\cdot)\|_{C^\alpha}
.
\end{align}
The uniform norm is bounded independently of $\lambda$ by Theorem 4.1 in~\cite{MR2748616} (recall~\eqref{eq_Dpsi} for the definition of $\lambda$):
\begin{equation}
\|\psi\|_0
\leq 
\frac{\|b\|_0}{\lambda}
\leq 
\|b\|_0
.
\label{eq_bound_unif_norm_psi}
\end{equation}
It therefore remains to estimate the H\"older norm of $D^2\psi$. 
This is done as in the proof of Theorem 2.4(i) in~\cite{MR2748616} (see middle of p18), 
with simpler computations due to $L^b$ (recall~\eqref{eq_pde_appendix}) having diffusion matrix equal to the identity. 
One obtains:
\begin{align}
\sup_{s\geq 0}\|D^2\psi(s,\cdot)\|_{C^\alpha}
&\leq 
C(\alpha)\sup_{s\geq 0}\|b(s,\cdot)\|_{C^\alpha}\big(1+\sup_{s\geq 0}\|D\psi(s,\cdot)\|_{C^\alpha}\big)
\nnb
&\leq
\frac{4C(\alpha)}{3}\sup_{s\geq 0}\|b(s,\cdot)\|_{C^\alpha}
,
\end{align}
where the second line comes from the estimate~\eqref{eq_Dpsi} of $\|D\psi\|_{C^\alpha}$. 
Together with~\eqref{eq_bound_unif_norm_psi} this concludes the proof.
\end{proof}

\section*{Acknowledgements}

We thank Martin Hairer and Jonathan Mattingly for very helpful and interesting discussions on the time-shifted Girsanov method.

R.B. acknowledges funding from NSF grant DMS-2348045.

B.D. acknowledges funding from the European Union's Horizon 2020 research and innovation programme under the Marie Sk\l odowska-Curie grant agreement No 101034255.

H.W. is funded by the European Union (ERC, GE4SPDE, 101045082) and by the Deutsche Forschungsgemeinschaft (DFG, German Research Foundation) under Germany's Excellence Strategy EXC 2044 -390685587, Mathematics Münster: Dynamics-Geometry-Structure.

\bibliography{all}
\bibliographystyle{plain}

\end{document}